\documentclass[11pt]{amsart}
\usepackage{amssymb,amsmath,amsfonts,amsthm,graphics,mathrsfs,amscd,hyperref}
\usepackage[hmargin=1in,vmargin=0.5in]{geometry}
\usepackage[all,cmtip]{xy}

\theoremstyle{plain}

\newtheorem{definition}[equation]{Definition}

\newtheorem{lemma}[equation]{Lemma}
\newtheorem{proposition}[equation]{Proposition}
\newtheorem{theorem}[equation]{Theorem}

\newtheorem{remark}[equation]{Remark}

\newtheorem{example}[equation]{Example}

\numberwithin{equation}{subsection}

\title{Theorem of existence and completeness for holomorphic Poisson structures}
\author{Chunghoon Kim}
\thanks{This paper is based on the first part of the author's Ph.D. thesis \cite{Kim14}. While writing this paper, the author was partially supported by NRF grant 2011-0027969.}
\email{ckim042@gmail.com}            

\begin{document}

\maketitle
\begin{abstract}
In this paper, we define a concept of a family of compact holomorphic Poisson manifolds on the basis of Kodaira-Spencer's deformation theory and deduce the integrability condition. We prove an  analogue of their `Theorem of existence for complex analytic structures' under some analytic assumption, and establish an analogue of their `Theorem of completeness for complex analytic structures' in the context of holomorphic Poisson deformations.
\end{abstract}

\tableofcontents                                      

\section{Introduction}\label{section1}

In this paper, we study deformations of holomorphic Poisson structures in the framework of Kodaira and Spencer's deformation theory of complex analytic structures(\cite{Kod58},\cite{Kod60}). The main difference from Kodaira and Spencer's deformation theory is that  for deformations of a compact holomorphic Poisson manifold, we deform not only its complex structures, but also holomorphic Poisson structures. We will briefly review Kodaira-Spencer's main idea and show how we can extend their idea in the context of deformations of holomorphic Poisson structures.

 Kodaira and Spencer's main idea of deformations of complex analytic structures is as follows \cite[p.182]{Kod05}. A $n$-dimensional compact complex manifold $M$ is obtained by glueing domains $U_1,...,U_n$ in $\mathbb{C}^n:M=\bigcup_{j=1}^n U_j$ where $\mathfrak{U}=\{U_j|j=1,...,n\}$ is a locally finite open covering of $M$, and that each $U_j$ is a polydisk:
 \begin{align*}
 U_j=\{z_j\in \mathbb{C}^n||z_j^1|<1,...,|z_j^n|<1\}
 \end{align*}
and for $p\in U_j\cap U_k$, the coordinate transformation
\begin{align*}
f_{jk}:z_k\to z_j=(z_j^1,...,z_j^n)=f_{jk}(z_k)
\end{align*}
transforming the local coordinates $z_k=(z_k^1,...,z_k^n)=z_k(p)$ into the local coordinates $z_j=(z_j^1,...,z_j^n)=z_j(p)$ is  biholomorphic. According to Kodaira,

 \begin{quote}
 \textit{``A deformation of $M$ is considered to be the glueing of the same polydisks $U_j$ via different identification. In other words, replacing $f_{jk}^{\alpha}(z_k) $ by the functions $f_{jk}^{\alpha}(z_k,t)=f^{\alpha}_{jk}(z_k,t_1,...,t_m),$ $ f_{jk}(z_k,0)=f_{jk}^{\alpha}(z_k)$ of $z_k$, and the parameter $t=(t_1,...,t_m)$, we obtain deformations $M_t$ of $M=M_0$ by glueing the polydisks $U_1,...,U_n$ by identifying $z_k\in U_k$ with $z_j=f_{jk}(z_k,t)\in U_j$"}
\end{quote}

We extend the main idea of Kodaira-Spencer in the context of deformations of holomorphic Poisson structures. A $n$-dimensional compact holomorphic Poisson manifold $M$ is a compact complex manifold such that the structure sheaf $\mathcal{O}_M$ is a sheaf of Poisson algebras (we refer to \cite{Lau13} for general information on Poisson geometry). The holomorphic Poisson structure is encoded in a holomorphic section (a holomorphic bivector field) $\Lambda \in H^0(M,\wedge^2 \Theta_M)$ with $[\Lambda,\Lambda]=0$, where $\Theta_M$ is the sheaf of germs of holomorphic vector fields on $M$ and the bracket $[-,-]$ is the Schouten bracket on $M$. In the sequel a holomorphic Poisson manifold will be denoted by $(M,\Lambda)$. For deformations of a compact holomorphic Poisson manifold $(M,\Lambda)$, we extend the idea of Kodaira and Spencer. A $n$-dimensional compact holomorphic Poisson manifold is obtained by glueing the domains $U_1,...,U_n$ in $\mathbb{C}^n$: $M=\bigcup_{j=1}^n U_j$ where $\mathfrak{U}=\{U_j|j=1,...,n\}$ is a locally finite open covering of $M$ and each $U_j$ is a polydisk 
 \begin{align*}
 U_j=\{z_j\in \mathbb{C}^n||z_j^1|<1,...,|z_j^n|<1\}
 \end{align*}
equipped with a holomorphic bivector field $\Lambda_j=\sum_{\alpha,\beta=1}^n g_{\alpha\beta}^j(z_j) \frac{\partial}{\partial z_j^{\alpha}}\wedge \frac{\partial}{\partial z_j^{\beta}}$ such that $g_{\alpha\beta}^j(z_j)=-g_{\beta\alpha}^j(z_j)$ with $[\Lambda_j,\Lambda_j]=0$ on $U_j$ and for $p\in U_j\cap U_k$, the coordinate transformation
\begin{align*}
f_{jk}:z_k\to z_j=(z_j^1,...,z_j^n)=f_{jk}(z_k)
\end{align*}
transforming the local coordinates $z_k=(z_k^1,...,z_k^n)=z_k(p)$ into the local coordinates $z_j=(z_j^1,...,z_j^n)=z_j(p)$ is  a biholomorphic `Poisson' map.

 Deformations of a compact holomorphic Poisson manifold $(M,\Lambda)$ is the glueing of the Poisson polydisks $(U_j,\Lambda_j(t))$ parametrized by $t$. That is, replacing $f_{jk}^{\alpha}(z_k)$ by $f_{jk}^{\alpha}(z_k,t) ( f_{jk}(z_k,0)=f_{jk}^{\alpha}(z_k)$ of $z_k$), replacing $\Lambda_j=\sum_{\alpha,\beta=1}^n g_{\alpha\beta}^j(z_j) \frac{\partial}{\partial z_j^{\alpha}}\wedge \frac{\partial}{\partial z_j^{\beta}}$ by $\Lambda_j(t)=\sum_{\alpha,\beta=1}^n g_{\alpha\beta}^j(z_j,t) \frac{\partial}{\partial z_j^{\alpha}}\wedge \frac{\partial}{\partial z_j^{\beta}}$ with $[\Lambda_j(t),\Lambda_j(t)]=0$ and $\Lambda_j(0)=\Lambda_j$, and the parmeter $t=(t_1,...,t_m)$, we obtain deformations $(M_t,\Lambda_t)$ by gluing the Poisson polydisks $(U_1,\Lambda_1(t)),...,(U_n,\Lambda_n(t))$ by identifying $z_k\in U_k$ with $z_j=f_{jk}(z_k,t)\in U_j$. The work on deformations of holomorphic Poisson structures is based on this fundamental idea.

In section \ref{section2}, we define a family of compact holomorphic Poisson manifolds, called a Poisson analytic family in the framework of Kodaira-Spencer's deformation theory. In other words, when we ignore Poisson structures, a family of compact holomorphic Poisson manifolds is just a family of compact complex manifolds in the sense of Kodaira and Spencer. So deformations of compact holomorphic Poisson manifolds means that we deform complex structures as well as Poisson structures.

In section \ref{section3}, we show that infinitesimal deformations of a holomorphic Poisson manifold $(M,\Lambda_0)$ in a Poisson analytic family are encoded in the first `degree-shifted by $1$' truncated holomorphic Poisson cohomology group. More precisely, an infinitesimal deformation is realized as an element in the first hypercohomology group $\mathbb{H}^1(M,\Theta_M^\bullet)$ of a complex of sheaves $\Theta_M^\bullet:\Theta_M\to \wedge^2 \Theta_M\to \cdots\to \wedge^n \Theta_M\to 0$ induced by $[\Lambda_0,-]$. Analogously to deformations of complex structure, we define so called Poisson Kodaira-Spencer map where the Kodaira-Spencer map is realized as a component of the Poisson Kodaira-Spencer map. 

In section \ref{section4}, we study the integrability condition for a Poisson analytic family. Kodaira showed that given a family of deformations of  a compact complex manifold $M$, locally the family is represented by a $C^{\infty}$ vector $(0,1)$-form $\varphi(t)\in A^{0,1}(M,T_M)$ with $\varphi(0)=0$ satisfying the integrability condition $\bar{\partial}\varphi(t)-\frac{1}{2}[\varphi(t),\varphi(t)]=0$ (see \cite{Kod05} \S 5.3. Here $T_M$ is the holomorphic tangent bundle of $M$ and we use the notation $A^{0,1}(M,T_M)$ instead of $\mathscr{L}^{0,1}(T_M)$ in \cite{Kod05}). We  show that given a family of deformations of a compact holomorphic Poisson manifold $(M,\Lambda_0)$, locally the family is represented by a $C^{\infty}$ vector $(0,1)$-form $\varphi(t)$ with $\varphi(0)=0$ and a $C^{\infty}$ bivector $\Lambda(t)\in A^{0,0}(M,\wedge^2 T_M)$ with $\Lambda(0)=\Lambda_0$ satisfying the integrability condition $[\Lambda(t),\Lambda(t)]=0, \bar{\partial} \Lambda(t)-[\Lambda(t),\varphi(t)]=0$, and $\bar{\partial}\varphi(t)-\frac{1}{2}[\varphi(t),\varphi(t)]=0$. Replacing $\varphi(t)$ by $-\varphi(t)$ and putting $\Lambda'(t):=\Lambda(t)-\Lambda_0$ so that we have $\Lambda'(0)=0$, the integrability condition is equivalent to $L(\varphi(t)+\Lambda'(t))+\frac{1}{2}[\varphi(t)+\Lambda'(t),\varphi(t)+\Lambda'(t)]=0$ where $L=\bar{\partial}+[\Lambda_0,-]$. Then $\varphi(t)+\Lambda'(t)$ is a solution of Maurer Cartan equation of the following differential graded Lie algebra 
\begin{align}\label{tt76}
\mathfrak{g}=(\bigoplus_{i\geq 0} g_i,g_i=\bigoplus_{p+q-1=i,p\geq 0, q\geq 1} A^{0,p}(M,\wedge^q T_M),L=\bar{\partial}+[\Lambda_0,-],[-,-]),
\end{align}
where $[-,-]$ is the Schouten bracket on $M$, and $A^{0,p}(M,\wedge^q T_M)$ is the global section of $\mathscr{A}^{0,p}(\wedge^q T_M)$ the sheaf of germs of $C^{\infty}$-section of $\wedge^p \bar{T}_M^*\otimes \wedge^q T_M$. Here $\bar{T}_M^*$ is the dual bundle of antiholomorphic tangent bundle $\bar{T}_M$ (see \cite{Kod05} p.108). We remark that the integrability condition was proved in more general context in the language of generalized complex geometry (See \cite{Gua11}). As $H^1(M,\Theta_M)$ is realized as a subspace of the second cohomology group of a compact complex manifold $M$ in the sense of generalized complex geometry, $\mathbb{H}^1(M,\Theta_M^\bullet)$ is realized as a subspace of the second cohomology group of a compact holomorphic Poisson manifold $(M,\Lambda_0)$ in the sense of generalized complex geometry. In this paper, we deduce the integrability condition by extending Kodaira-Spencer's original approach, that is, by starting from a concept of a geometric family (a Poisson analytic family).

In section \ref{section5}, under some  analytic assumption, we establish an analogous theorem to the following theorem of Kodaira and Spencer (\cite{Kodaira58},\cite{Kod05} p.270).
\begin{theorem}[Theorem of existence for complex analytic structures]
Let $M$ be a compact complex manifold and suppose $H^2(M,\Theta)=0$. Then there exists a complex analytic family $(\mathcal{M},B,\omega)$ with $0\in B\subset \mathbb{C}^m$ satisfying the following conditions:
\begin{enumerate}
\item $\omega^{-1}(0)=M$
\item The Kodaira-Spencer map $\rho_0:\frac{\partial}{\partial t}\to \left(\frac{\partial M_t}{\partial t}\right)_{t=0}$ with $M_t=\omega^{-1}(t)$ is an isomorphism of $T_0(B)$ onto $H^1(M,\Theta_M):T_0(B)\xrightarrow{\rho_0} H^1(M,\Theta_M)$.
\end{enumerate}
\end{theorem}
Similarly, we prove `Theorem of existence for deformations of holomorphic Poisson structures' (see Theorem \ref{theorem of existence}).
\begin{theorem}[Theorem of existence for  holomorphic Poisson structures]\label{theorem of existence}
Let $(M,\Lambda_0)$ be a compact holomorphic Poisson manifold such that the associated Laplacian operator $\Box$ $($induced from the operator $\bar{\partial}+[\Lambda_0,-]$$)$ is strongly elliptic and of diagonal type. Suppose that $\mathbb{H}^2(M,\Theta^\bullet)=0$. Then there exists a Poisson analytic family $(
\mathcal{M},\Lambda,B,\omega)$ with $0\in B\subset \mathbb{C}^m$ satisfying the following conditions:
\begin{enumerate}
\item $\omega^{-1}(0)=(M,\Lambda_0)$
\item The Poisson Kodaira-Spencer map $\varphi_0:\frac{\partial}{\partial t}\to\left(\frac{\partial (M_t,\Lambda_t)}{\partial t}\right)_{t=0}$ with $(M_t,\Lambda_t)=\omega^{-1}(t)$ is an isomorphism of $T_0(B)$ onto $\mathbb{H}^1(M,\Theta_M^\bullet):T_0 B\xrightarrow{\rho_0} \mathbb{H}^1(M,\Theta_M^\bullet)$.
\end{enumerate}
\end{theorem}
The proof is rather formal. The proof follows from the Kuranishi's method presented in \cite{Mor71}. The reason for the assumption on the associated Laplacian operator $\Box$ (induced from the operator $\bar{\partial}+[\Lambda_0,-]$) is for applying the Kuranishi's method in the holomorphic Poisson context.

In section \ref{section6}, we establish an analogous theorem to the following theorem of Kodaira and Spencer (\cite{KS58},\cite{Kod05} p.284).

\begin{theorem}[Theorem of completeness for complex analytic structures]\label{kodairacomplete}
Let $(\mathcal{M},B,\omega)$ be a complex analytic family of deformations of a compact complex manifold $M_0=\omega^{-1}(0)$, $B$ a domain of $\mathbb{C}^m$ containing $0$. If the Kodaira-Spencer map $\rho_0:T_0 (B)\to H^1(M_0,\Theta_{M_0})$ is surjective, the complex analytic family $(\mathcal{M},B,\omega)$ is complete at $0\in B$.
\end{theorem}

Similarly, we prove the following theorem which is an analogue of `Theorem of completeness' by Kodaira-Spencer.  

\begin{theorem}[Theorem of completeness for holomorphic Poisson structures]
Let $(\mathcal{M},\Lambda_{\mathcal{M}},B,\omega)$ be a Poisson analytic family of deformations of a compact holomorphic Poisson manifold $(M,\Lambda_0)=\omega^{-1}(0)$, $B$ a domain of $\mathbb{C}^m$ containing $0$. If the Poisson Kodaira-Spencer map $\varphi_0:T_0 (B) \to \mathbb{H}^1(M,\Theta_M^\bullet)$ is surjective, the Poisson analytic family $(\mathcal{M},\Lambda_{\mathcal{M}}, B,\omega)$ is complete at $0\in B$.
\end{theorem}

\section{Families of compact holomorphic Poisson manifolds}\label{section2}

\begin{definition}$($compare \cite{Kod05} p.59$)$\label{definition}
Suppose that given a domain $B\subset \mathbb{C}^m$, there is a set $\{(M_t,\Lambda_t)|t \in B\}$ of $n$-dimensional compact holomorphic Poisson manifolds $(M_t,\Lambda_t)$, depending on $t=(t_1,...,t_m)\in B$. We say that $\{(M_t,\Lambda_t)|t\in B\}$ is a family of compact holomorphic Poisson manifolds or a Poisson analytic family of compact holomorphic Poisson manifolds if there exists a holomorphic Poisson manifold $(\mathcal{M},\Lambda)$ and a holomorphic map $\omega:\mathcal{M}\to B$ satisfing the following properties
\begin{enumerate}
\item $\omega^{-1}(t)$ is a compact holomorphic Poisson submanifold of $(\mathcal{M},\Lambda)$ for each $t\in B$.
\item $(M_t,\Lambda_t)=\omega^{-1}(t)(M_t$ has the induced Poisson holomorphic structure $\Lambda_t$ from $\Lambda)$.
\item The rank of Jacobian of $\omega$ is equal to $m$ at every point of $\mathcal{M}$.
\end{enumerate}
We will denote a Poisson analytic family by $(\mathcal{M},\Lambda,B,\omega)$. We also call $(\mathcal{M},\Lambda,B,\omega)$ a Poisson analytic family of deformations of a compact holomorphic Poisson manifold $(M_{t_0},\Lambda_{t_0})$ for each fixed $t_0\in B$.
\end{definition}

\begin{remark}
When we ignore Poisson structures, a Poisson analytic family $(\mathcal{M},\Lambda,B,\omega)$ is a complex analytic family $(\mathcal{M},B,\omega)$ in the sense of Kodaira-Spencer $($see \cite{Kod05} p.59$)$.
\end{remark}

\begin{remark}\label{tt61}
Given a Poisson analytic family $(\mathcal{M},\Lambda,B,\omega)$ as in Definition $\ref{definition}$,  we can choose a locally finite open covering $\mathcal{U}=\{\mathcal{U}_j\}$ of $\mathcal{M}$ such that $\mathcal{U}_j$ are coordinate polydisks with a system of local complex coordinates $\{z_1,...,z_j,...\}$, where a local coordinate function $z_j:p\to z_j(p)$ on $\mathcal{U}_j$ satisfies $z_j(p)=(z_j^1(p),...,z_j^n(p),t_1,...,t_m)$, and $t=(t_1,...,t_m)=\omega(p)$. Then for a fixed $t_0\in B$, $\{p\mapsto (z_j^1(p),...,z_j^n(p))| \mathcal{U}_j \cap M_{t_0}\ne \emptyset\}$ gives a system of local complex coordinates on $M_{t_0}$. In terms of these coordinates, $\omega$ is the projection given by $(z_j,t)=(z_j^1,...,z_j^n,t_1,...,t_m)\to (t_1,...,t_m)$. For $j,k$ with $\mathcal{U}_j\cap \mathcal{U}_k\ne \emptyset$, we denote the coordinate transformations from $z_k$ to $z_j$ by $f_{jk}:(z_k^1,...,z_k^n,t)\to (z_j^1,...,z_j^n,t)=f_{jk}(z_k^1,...,z_k^n,t)$$($for the detail, see \cite{Kod05} p.60$)$. 
 
On the other hand, since $(M_t,\Lambda_t) \hookrightarrow (\mathcal{M},\Lambda)$ is a holomorphic Poisson submanifold for each $t\in B$ and $\mathcal{M}=\bigcup_t M_t$, the holomorphic Poisson structure $\Lambda$ on $\mathcal{M}$ can be expressed in terms of local coordinates as $\Lambda=\sum_{\alpha,\beta=1}^n g_{\alpha \beta}^j(z_j^1,...,z_j^n,t)\frac{\partial}{\partial{z_j^{\alpha}}}\wedge \frac{\partial}{\partial{z_j^{\beta}}}$ on $\mathcal{U}_j$, where $g_{\alpha\beta}^j(z_j,t)=g_{\alpha\beta}^j(z_j^1,...,z_n^n,t)$ is holomorphic with respect to $(z_j,t)$ with $g_{\alpha\beta}^j(z_j,t)=-g_{\beta\alpha}^j(z_j,t)$. For a fixed $t^0$, the holomorphic Poisson structure $\Lambda_{t^0}$ on $M_{t_0}$ is given by $\sum_{\alpha,\beta=1}^n g_{\alpha \beta}^j(z_j^1,...,z_j^n,t^0)\frac{\partial}{\partial{z_j^{\alpha}}}\wedge \frac{\partial}{\partial{z_j^{\beta}}}$ on $\mathcal{U}_j\cap M_{t_0}$.
\end{remark}

\begin{remark}\label{restriction}
Let $(\mathcal{M},\Lambda,B,\omega)$ be a Poisson analytic family. Let $\Delta$ be an open set of $B$. Then the restriction $(\mathcal{M}_{\Delta}=\omega^{-1}(\Delta),\Lambda|_{M_{\Delta}},\Delta,\omega|_{\mathcal{M}_{\Delta}})$ is also a Poisson analytic family. We will denote the family by $(\mathcal{M}_{\Delta},\Lambda_{\Delta},\Delta,\omega)$.
\end{remark}

\begin{example}[complex tori]$($\cite{Kod58} $p.408$$)$
Let $S$ be the space of $n\times n$ matrices $s=(s_{\beta}^{\alpha})$ with $\det(Im(s) ) >0$, where $\alpha$ denotes the row index and $\beta$ the column index, and $Im(s)$ is the imaginary part of $s$. For each matrix $s\in S$ we define an $n\times 2n$ matrix $\omega(s)=(\omega_j^{\alpha}(s))$ by
\begin{equation*} 
\omega_j^{\alpha}(s)=
\begin{cases}
\delta_j^{\alpha},\,\,\,\,\,\,\,\,\,\, $\text{for $1\leq j\leq n$}$\\
s_\beta^{\alpha},\,\,\,\,\,\,\,\,\ $\text{for $j=n+\beta, 1\leq \beta \leq n$}$
\end{cases}
\end{equation*}

Let $G$ be the discontinuous abelian group of analytic automorphisms of $\mathbb{C}^n\times S$ generated by $g_j:(z,s)\to (z+\omega_j(s),s),\,\,\,\,\, j=1,...,2n,$
where $\omega_j(s)=(\omega_j^1(s),...,\omega_j^{\alpha}(s),...,\omega_j^n(s))$ is th $j$-th column vector of $\omega(s)$. The quotient space $\mathcal{M}=\mathbb{C}^n\times S/G$ and $\pi:\mathcal{M}\to S$ induced from the canonical projection $\mathbb{C}^n\times S\to S$ forms a complex analytic family of complex tori. We will put a holomorphic Poisson structure on $\mathcal{M}$ to make a Poisson analytic family. A holomorphic bivector field of the form $\Lambda=\sum_{i,j=1}^nf_{ij}(s)\frac{\partial}{\partial z_i}\wedge \frac{\partial}{\partial z_j}$ on $\mathbb{C}^n\times S$ where $f_{ij}(s)=f_{ij}(z,s)$ are holomorphic functions on $\mathbb{C}^n\times S$, independent of $z$, is a $G$-invariant bivector field on $\mathbb{C}^n\times S$. So this induces a holomorphic bivector field on $\mathcal{M}$. Since $f_{ij}(s)$ are independent of $z$, we have $[\Lambda,\Lambda]=0$. So $(\mathcal{M},\Lambda,S, \pi)$ is a Poisson analytic family.

\end{example}

\begin{example}[Hirzebruch-Nagata surface]$($\cite{Uen99} $p.13$$)$
Take two $\mathbb{C}\times \mathbb{P}_{\mathbb{C}}^1\times \mathbb{C}$ and write the coordinates as $(u,(\xi_0:\xi_1),t), (v,(\eta_0:\eta_1),t))$, respectively, where $u,v,t$ are the coordinates of $\mathbb{C}$ and $(\xi_0:\xi_1),(\eta_0:\eta_1)$ are the homogeneous coordinates of $\mathbb{P}_{\mathbb{C}}^1$.
By patching two $\mathbb{C}\times \mathbb{P}_{\mathbb{C}}^1\times \mathbb{C}$ together by relation
\begin{equation*}\label{relation}
\begin{cases}
u=1/v, \\
(\xi_0:\xi_1)=(\eta_0:v^m\eta_1+tv^k\eta_0), \,\,\,\,\,m-2\leq 2k \leq  m,\,\,\, \text{where} \,\,\, m,k\,\,\, \text{are natural numbers}\\
t=t,
\end{cases}
\end{equation*}
we obtain a complex analytic family $\pi:\mathcal{S}\to \mathbb{C}$ which is induced from the natural projection $\mathbb{C}\times \mathbb{P}_{\mathbb{C}}^1\times \mathbb{C}\to \mathbb{C}$ to the third component. We will put a holomorphic Poisson structure $\Lambda$ on $\mathcal{S}$ so that $(\mathcal{S},\Lambda,\mathbb{C},\pi)$ is a Poisson analytic family. $S$ has four affine covers. For one $\mathbb{C}\times \mathbb{P}_\mathbb{C}^1\times\mathbb{C}$ with coordinate $(u,(\xi_0:\xi_1),t)$, we have two affine covers, namely, $\mathbb{C}\times \mathbb{C}\times \mathbb{C}$ and $\mathbb{C}\times \mathbb{C}\times \mathbb{C}$. They are glued via  $\mathbb{C}\times (\mathbb{C}-\{0\})\times \mathbb{C}$ and  $\mathbb{C}\times (\mathbb{C}-\{0\})\times \mathbb{C}$ by $(u,x=\frac{\xi_1}{\xi_0},t)\mapsto (u,y=\frac{\xi_0}{\xi_1},t)=(u,\frac{1}{x},t)$. Similarly for another $\mathbb{C}\times \mathbb{P}_\mathbb{C}^1\times \mathbb{C}$, two affine covers are glued via $\mathbb{C}\times (\mathbb{C}-\{0\})\times \mathbb{C}$ and $\mathbb{C}\times (\mathbb{C}-\{0\})\times \mathbb{C}$ by $(v,w=\frac{\eta_1}{\eta_0},t)\mapsto (v,z,t)=(v,\frac{1}{w}=\frac{\eta_0}{\eta_1},t)$. We put holomorphic Poisson structures on each four affine covers which define a global bivector field  $\Lambda$ with $[\Lambda,\Lambda]=0$  on $\mathcal{S}$. On $(u,x,t)$ coordinate, we give $g(t)x^2\frac{\partial}{\partial u}\wedge \frac{\partial}{\partial x}$, where $g(t)$ is any holomorphic function depending only on $t$. On $(u,y,t)$ coordinate, we give $-g(t)\frac{\partial}{\partial u}\wedge \frac{\partial}{\partial y}$. On $(v,w,t)$ coordinate, we give $-g(t)v^{2k-m+2}(wv^{m-k}+t)^2\frac{\partial}{\partial v}\wedge\frac{\partial}{\partial w}$. On $(v,z,t)$ coordinate, we give $g(t)v^{2k-m+2}(v^{m-k}+tz)^2\frac{\partial}{\partial v}\wedge \frac{\partial}{\partial z}$. Then $(\mathcal{S},\Lambda,\mathbb{C},\pi)$ is a Poisson analytic family.
\end{example}

\begin{example}[Hopf surfaces]

We construct an one parameter Poisson analytic family of general Hopf surfaces.
An automorphism of $W\times \mathbb{C}$ given by $g:(z_1,z_2,t)\to(az_1+tz_2^m,bz_2,t)$ where $0<|a|\leq |b| <1$ and $b^m-a=0$ $($i.e $a=b^m$$)$, generates an infinite cyclic group $G$, which properly discontinuous and fixed point free. Hence $\mathcal{M}:=W\times \mathbb{C}/G$ is a complex manifold. Since the projection of $W\times \mathbb{C}$ to $\mathbb{C}$ commutes with $g$, it induces a holomorphic map $\omega$ of $\mathcal{M}$ to $\mathbb{C}$. So $(\mathcal{M},\mathbb{C},\omega)$ is a complex analytic family. Since $g^n$ is given by $g^n:(z_1,z_2,t)\to(z_1',z_z',t')=(a^n z_1+na^{n-1}t z_2^m,b^n z_2,t)$,
we have 
\begin{equation*}
\frac{\partial}{\partial z_1}=a^n\frac{\partial}{\partial z_1'},\,\,\,\,\, \frac{\partial}{\partial z_2}=mna^{n-1}t z_2^{m-1}\frac{\partial}{\partial z_1'}+b^n\frac{\partial}{\partial z_2'},\,\,\,\,\, \frac{\partial}{\partial z_1}\wedge \frac{\partial}{\partial z_2}=a^nb^n\frac{\partial}{\partial z_1'}\wedge \frac{\partial}{\partial z_2'}
\end{equation*}
Then $f(t)z_2^{m+1}\frac{\partial}{\partial z_1}\wedge \frac{\partial}{\partial z_2}$ where $f(t)$ is any holomorphic function, independent of $z$, is a $G$-invariant holomorphic bivector field on $W\times \mathbb{C}$ and so define a holomorphic Poisson structure on $\mathcal{M}$. Hence $(\mathcal{M},f(t)z_2^{m+1}\frac{\partial}{\partial z_1}\wedge \frac{\partial}{\partial z_2},\mathbb{C},\omega)$ is a Poisson analytic family of Poisson Hopf surfaces. 
\end{example}

\section{Infinitesimal deformations}\label{section3}

\subsection{Infinitesimal deformations and truncated holomorphic Poisson cohomology}\

In this subsection, we show that given a Poisson analytic family $(\mathcal{M},\Lambda,B,\omega)$, an infinitesimal deformation of a compact holomorphic Poisson manifold $\omega^{-1}(t)=(M_t,\Lambda_t)$ with dimension $n$ is captured by an element in the first hypercohomology group of the complex of sheaves $\Theta_{M_t}^\bullet: \Theta_{M_t}\to \wedge^2 \Theta_{M_t}\to \cdots \to \wedge^n \Theta_{M_t}\to 0$ induced by $[\Lambda_t,-]$ analogously to how an infinitesimal deformation of a compact complex manifold $M_t$ is captured by an element in the first cohomology group $H^1(M_t,\Theta_t)$.

Let $(M,\Lambda_0)$ be a compact holomorphic Poisson manifold and consider the complex of sheaves 
\begin{align}\label{complex}
\Theta_M^\bullet:\Theta_M\xrightarrow{[\Lambda_0,-]}\wedge^2 \Theta_M\xrightarrow{[\Lambda_0,-]}\cdots \xrightarrow{[\Lambda_0,-]} \wedge^n \Theta_M\to 0
\end{align}
where $\Theta_M$ is the sheaf of germs of holomorphic vector fields on $M$. Let $\mathcal{U}=\{U_j\}$ be sufficiently fine open covering of $M$ such that $U_j$ are coordinate polydisks of $M$, that is, $U_j=\{(z_j^1,...,z_j^n)\in \mathbb{C}^n||z_j^{\alpha}|<r_j^{\alpha},\alpha=1,...,n\}$ where $z_j=(z_j^1,...z_j^n)$ is a local coordinate on $U_j$ and $r_j^{\alpha}>0$ is a constant. Then we can compute the hypercohomology group of the complex of sheaves $(\ref{complex})$ by the following \u{C}ech resolution (see \cite{EV92} Appendix). Here $\delta$ is the \u{C}ech map.
\begin{center}
$\begin{CD}
@A[\Lambda_0,-]AA\\
C^0(\mathcal{U},\wedge^3 \Theta_M)@>-\delta>>\cdots\\
@A[\Lambda_0,-]AA @A[\Lambda_0,-]AA\\
C^0(\mathcal{U},\wedge^2 \Theta_M)@>\delta>> C^1(\mathcal{U},\wedge^2 \Theta_M)@>-\delta>>\cdots\\
@A[\Lambda_0,-]AA @A[\Lambda_0,-]AA @A[\Lambda_0,-]AA\\
C^0(\mathcal{U},\Theta_M)@>-\delta>>C^1(\mathcal{U},\Theta_M)@>\delta>>C^2(\mathcal{U},\Theta_M)@>-\delta>>\cdots\\
\end{CD}$
\end{center}

\begin{definition}
We say that the $i$-th `degree-shifted by $1$' truncated holomorphic Poisson cohomology group of a holomorphic Poisson manifold $(M,\Lambda_0)$ is the $i$-th hypercohomology group associated with the complex of sheaves $(\ref{complex})$, and is denoted by $\mathbb{H}^i(M,\Theta_M^\bullet)$.
\end{definition}

\begin{remark}
In \cite{Wei99}, the holomorphic Poisson cohomology for a holomorphic Poisson manifold $(M,\Lambda_0)$ is defined by the $i$-th hypercohomology group of complex of sheaves $\mathcal{O}_M\to \Theta_M\to \wedge^2 \Theta_M \to \cdots \to\wedge^n \Theta_M\to 0$ induced by $[\Lambda_0,-]$.  Since there is no role of the structure sheaf $\mathcal{O}_M$ in deformations of compact holomorphic Poisson manifolds, we truncate the complex of sheaves to get $0\to \Theta_M\to \wedge^2 \Theta_M\to \cdots \wedge^n \Theta_M\to 0$. In \cite{Kim14}, the author used the expression $HP^i(M,\Lambda_0)$ for the $i$-th truncated holomorphic Poisson cohomology group to maintain notational consistency with \cite{Nam09} by which this present work was inspired. However we shift the degree after truncation to get $\Theta_M\to \wedge^2 \Theta_M\to \cdots \wedge^n \Theta_M\to 0$ since it looks more natural by the general philosophy of deformation theory so that the $0$-th cohomology group corresponds to infinitesimal Poisson automorphisms, the first cohomology group corresponds to infinitesimal Poisson deformations and the third cohomology group corresponds to obstructions $($see the third part of the author's Ph.D. thesis \cite{Kim14}$)$.
\end{remark}

 We will relate the first `degree-shifted by $1$' truncated holomorphic Poisson cohomology group $\mathbb{H}^1(M_t,\Theta_{M_t}^\bullet)$ to infinitesimal deformations of $\omega^{-1}(t)=(M_t,\Lambda_t)$ in a Poisson  analytic family $(\mathcal{M},\Lambda,B,\omega)$ for each $t\in B$. As in Remark $\ref{tt61}$, let $\mathcal{U}=\{\mathcal{U}_j\}$ be an open covering of $\mathcal{M}$ such that $\mathcal{U}_j$ are coordinate polydisks of $\mathcal{M}$, $\{(z_j,t)\}=\{(z_j^1,...,z_j^n,t_1,...,t_m)\}$ is a local complex coordinate system on $\mathcal{U}_j$, and $z_j^{\alpha}=f_{jk}^{\alpha}(z_k^1,...,z_k^n,t_1,...,t_m),\alpha=1,...,n$ 
is a holomorphic transition function from $z_k$ to $z_j$. The Poisson structure $\Lambda$ is expressed in terms of local complex coordinate system on $\mathcal{U}_j$ as 
\begin{align}\label{poisson}
\Lambda=\Lambda_j=\sum_{\alpha,\beta=1}^n g_{\alpha \beta}^{j}(z_j,t)\frac{\partial}{\partial z_{j}^{\alpha}}\wedge \frac{\partial}{\partial z_{j}^{\beta}}
\end{align}
where $g^{j}_{\alpha \beta}(z_j,t)$ is a holomorphic function on $\mathcal{U}_j$ with $g_{\alpha\beta}^j(z_j,t)=-g_{\beta\alpha}^j(z_j,t)$ and we have
\begin{align}\label{tt67}
[\Lambda,\Lambda]=[\sum_{\alpha,\beta=1}^n g_{\alpha \beta}^{j}(z_j,t)\frac{\partial}{\partial z_{j}^{\alpha}}\wedge \frac{\partial}{\partial z_{j}^{\beta}},\sum_{\alpha,\beta=1}^n g_{\alpha \beta}^{j}(z_j,t)\frac{\partial}{\partial z_{j}^{\alpha}}\wedge \frac{\partial}{\partial z_{j}^{\beta}}]=0
\end{align}
Since $f_{jk}(z_k,t)=(f_{jk}^1(z_k,t),...,f_{jk}^n(z_k,t),t_1,...,t_m)$ is a Poisson map, we have
\begin{align}\label{tt56}
 g_{\alpha \beta}^j(f_{jk}^1(z_k,t),...,f_{jk}^n(z_k,t))=\sum_{r,s=1}^n g_{rs}^k(z_k,t)\frac{\partial f_{jk}^{\alpha}}{\partial z_k^r}\frac{\partial f_{jk}^{\beta}}{\partial z_k^s}
\end{align}
 on $\mathcal{U}_j\cap \mathcal{U}_k$. Set $\mathcal{U}_j^t:=\mathcal{U}_j\cap M_t$. Then for each $t\in B$, $\mathcal{U}^t:=\{\mathcal{U}_j^t\}$ is an open covering of $M_t$. Recall that $\Lambda_t$ is the Poisson structure on $M_t$ induced from $(\mathcal{M},\Lambda)$. Let $\frac{\partial}{\partial t}=\sum_{\lambda=1}^m c_{\lambda}\frac{\partial}{\partial t_{\lambda}}$, $c_{\lambda}\in \mathbb{C}$ be a tangent vector of $B$. Then we have

\begin{proposition}\label{gg}
\begin{align*}
(\{\lambda_j(t)=\sum_{\alpha,\beta=1}^n \frac{\partial g_{\alpha \beta}^{j}(z_j,t)}{\partial t}\frac{\partial}{\partial z_{j}^{\alpha}}\wedge \frac{\partial}{\partial z_{j}^{\beta}}\}, \{\theta_{jk}(t)=\sum_{\alpha=1}^n \frac{\partial f_{jk}^{\alpha}(z_k,t)}{\partial t}\frac{\partial}{\partial z_j^{\alpha}}\})\in C^0(\mathcal{U}^t,\wedge^2 \Theta_{M_t})\oplus C^1(\mathcal{U}^t,\Theta_{M_t})
\end{align*}
define a 1-cocycle and call its cohomology class in $\mathbb{H}^1(M_t,\Theta_{M_t}^\bullet)$ the infinitesimal $($Poisson$)$ deformation along $\frac{\partial}{\partial t}$. This expression is independent of the choice of system of local coordinates.
\end{proposition}
\begin{proof}
First we note that $\delta(\{\theta_{jk}(t)\})=0$ (See \cite{Kod05} p.201). Second, by taking the derivative of $(\ref{tt67})$ with respect to $t$,  we have $[\sum_{\alpha,\beta=1}^n g_{\alpha \beta}^{j}(z_j,t)\frac{\partial}{\partial z_{j}^{\alpha}}\wedge \frac{\partial}{\partial z_{j}^{\beta}},\sum_{\alpha,\beta=1}^n \frac{\partial g_{\alpha \beta}^{j}(z_j,t)}{\partial t}\frac{\partial}{\partial z_{j}^{\alpha}}\wedge \frac{\partial}{\partial z_{j}^{\beta}}]=0$. It remains to show that $\delta(\{\lambda_j(t)\})+[\Lambda_t,\{\theta_{jk}\}]=0$. More precisely, on $\mathcal{U}_{j}^t\cap \mathcal{U}_k^t\ne \emptyset$, we show that $\lambda_{k}(t)-\lambda_{j}(t)+[\Lambda_t,\theta_{jk}(t)]=0$. In other words,
\begin{align}\label{equ1}
\sum_{r,s=1}^n \frac{\partial g^k_{rs}}{\partial t}\frac{\partial}{\partial z^{r}_k}\wedge\frac{\partial}{\partial z_k^{s}}-\sum_{\alpha,\beta=1}^n \frac{\partial g^j_{\alpha \beta}}{\partial t}\frac{\partial}{\partial z^{\alpha}_j}\wedge\frac{\partial}{\partial z_j^{\beta}}+[\sum_{r,s=1}^n g_{rs}^{j}(z_j,t)\frac{\partial}{\partial z_{j}^{r}}\wedge \frac{\partial}{\partial z_{j}^{s}},\sum_{c=1}^n \frac{\partial f_{jk}^{c}(z_k,t)}{\partial t}\frac{\partial}{\partial z_j^{c}}]=0
\end{align}
Since $z_j^{\alpha}=f_{jk}^{\alpha}(z_k^1,...,z_k^n,t_1,...,t_m)$ for $\alpha=1,...,n$, we have $\frac{\partial}{\partial z_k^{r}}=\sum_{a=1}^{n}\frac{\partial f_{jk}^a}{\partial z_k^{r}}\frac{\partial}{\partial z_j^a}$ for $r=1,...,n$. Hence the first term of $(\ref{equ1})$ is
\begin{align*}
\sum_{r,s=1}^n \frac{\partial g^k_{rs}}{\partial t}\frac{\partial}{\partial z^{r}_k}\wedge\frac{\partial}{\partial z_k^{s}}=\sum_{r,s,a,b=1}^n \frac{\partial g_{rs}^k}{\partial t}\frac{\partial f_{jk}^a}{\partial z_k^r}\frac{\partial f_{jk}^b}{\partial z_k^s}\frac{\partial}{\partial z_j^a}\wedge \frac{\partial}{\partial z_j^b}
\end{align*}
We compute the third term of $(\ref{equ1})$:
\begin{align*}
&\sum_{r,s,c=1}^n [g_{rs}^{j}(z,t)\frac{\partial}{\partial z_{j}^{r}}\wedge \frac{\partial}{\partial z_{j}^{s}},\frac{\partial f_{jk}^{c}(z_k,t)}{\partial t}\frac{\partial}{\partial z_j^{c}}]=\sum_{r,s,c=1}^n ([g_{rs}^j \frac{\partial}{\partial z_j^r},\frac{\partial f_{jk}^c}{\partial t} \frac{\partial}{\partial z_j^c}]\wedge \frac{\partial}{\partial z_j^s}-g_{rs}^j[\frac{\partial}{\partial z_j^s},\frac{\partial f_{jk}^c}{\partial t}\frac{\partial}{\partial z_j^c}]\wedge \frac{\partial}{\partial z_j^r})\\
&=\sum_{r,s,c=1}^n (g_{rs}^j\frac{\partial}{\partial z_j^r}\left(\frac{\partial f_{jk}^c}{\partial t}\right) \frac{\partial}{\partial z_j^c}\wedge \frac{\partial}{\partial z_j^s}-\frac{\partial f_{jk}^c}{\partial t}\frac{\partial g_{rs}^j}{\partial z_j^c}\frac{\partial}{\partial z_j^r}\wedge \frac{\partial}{\partial z_j^s}+g_{rs}^j\frac{\partial}{\partial z_j^s}\left(\frac{\partial f_{jk}^c}{\partial t}\right)\frac{\partial}{\partial z_j^r}\wedge \frac{\partial}{\partial z_j^c})
\end{align*}
By considering the coefficients of $\frac{\partial}{\partial z_j^a}\wedge \frac{\partial}{\partial z_j^b}$, $(\ref{equ1})$ is equivalent to
\begin{align}\label{equ2}
\sum_{r,s=1}^n \frac{\partial g_{rs}^k}{\partial t}\frac{\partial f_{jk}^a}{\partial z_k^r}\frac{\partial f_{jk}^b}{\partial z_k^s}-\frac{\partial g_{ab}^j}{\partial t}-\sum_{c=1}^n \frac{\partial g_{ab}^j}{\partial z_j^c}\frac{\partial f_{jk}^c}{\partial t}+\sum_{c=1}^n (g_{cb}^j\frac{\partial}{\partial z_j^c}\left(\frac{\partial f_{jk}^a}{\partial t}\right)+g_{ac}^j\frac{\partial}{\partial z_j^c}\left(\frac{\partial f_{jk}^b}{\partial t}\right))=0
\end{align}

On the other hand, from $(\ref{tt56})$, we have
\begin{align}\label{tt77}
g_{ab}^j(f_{jk}^1(z_k,t),...,f_{jk}^n(z_k,t),t_1,...,t_m)=\sum_{r,s=1}^n g_{rs}^k \frac{\partial f_{jk}^a}{\partial z_k^r}\frac{\partial f_{jk}^b}{\partial z_k^s}\,\,\,\,\,\,\text{on}\,\,\,\mathcal{U}_j \cap \mathcal{U}_k
\end{align} 
By taking the derivative of $(\ref{tt77})$ with respect to $t$, we have
\begin{align*}
\sum_{c=1}^n \frac{\partial g_{ab}^j}{\partial z_j^c}\frac{\partial f_{jk}^c}{\partial t}+\frac{\partial g_{ab}^j}{\partial t}=\sum_{r,s=1}^n \frac{\partial g_{rs}^k}{\partial t}\frac{\partial f_{jk}^a}{\partial z_k^r}\frac{\partial f_{jk}^b}{\partial z_k^s}+\sum_{r,s=1}^ng_{rs}^k(\frac{\partial}{\partial z_k^r}\left(\frac{\partial f_{jk}^a}{\partial t}\right)\frac{\partial f_{jk}^b}{\partial z_k^s}+\frac{\partial f_{jk}^a}{\partial z_k^r}\frac{\partial}{\partial z_k^s}\left(\frac{\partial f_{jk}^b}{\partial t}\right) )
\end{align*}
Hence $(\ref{equ2})$ is equivalent to
\begin{align}\label{tt57}
\sum_{c=1}^n (g_{cb}^j\frac{\partial}{\partial z_j^c}\left(\frac{\partial f_{jk}^a}{\partial t}\right)+g_{ac}^j\frac{\partial}{\partial z_j^c}\left(\frac{\partial f_{jk}^b}{\partial t}\right))=\sum_{r,s=1}^n g_{rs}^k(\frac{\partial}{\partial z_k^r}\left(\frac{\partial f_{jk}^a}{\partial t}\right)\frac{\partial f_{jk}^b}{\partial z_k^s}+\frac{\partial f_{jk}^a}{\partial z_k^r}\frac{\partial}{\partial z_k^s}\left(\frac{\partial f_{jk}^b}{\partial t}\right) )
\end{align}
Indeed, the left hand side and right hand side of $(\ref{tt57})$ coincide: from $(\ref{tt56})$
{\small{\begin{align*}
\sum_{c=1}^n (g_{cb}^j\frac{\partial}{\partial z_j^c}\left(\frac{\partial f_{jk}^a}{\partial t}\right)+g_{ac}^j\frac{\partial}{\partial z_j^c}\left(\frac{\partial f_{jk}^b}{\partial t}\right))&=\sum_{r,s,c=1}^n (g_{rs}^k\frac{\partial f_{jk}^c}{\partial z_k^r}\frac{\partial f_{jk}^b}{\partial z_k^s}\frac{\partial}{\partial z_j^c}\left(\frac{\partial f_{jk}^a}{\partial t}\right)+g_{rs}^k\frac{\partial f_{jk}^a}{\partial z_k^r}\frac{\partial f_{jk}^c}{\partial z_k^s}\frac{\partial}{\partial z_j^c}\left(\frac{\partial f_{jk}^b}{\partial t}\right))\\
&=\sum_{r,s=1}^n g_{rs}^k(\frac{\partial}{\partial z_k^r}\left(\frac{\partial f_{jk}^a}{\partial t}\right)\frac{\partial f_{jk}^b}{\partial z_k^s}+\frac{\partial f_{jk}^a}{\partial z_k^r}\frac{\partial}{\partial z_k^s}\left(\frac{\partial f_{jk}^b}{\partial t}\right) )
\end{align*}}}
This proves the first claim. It remains to show that $(\{\lambda_j(t)\},\{\theta_{jk}(t)\})$ is independent of the choice of systems of local coordinates. We can show that the infinitesimal deformation does not change under the refinement of the open covering (See \cite{Kod05} p.190). Since we can choose a common refinement for two system of local coordinates, it is sufficient to show that given two local coordinates $x_j=(z_j,t)$ and $u_j=(w_j,t)$ on each $\mathcal{U}_j$, the infinitesimal Poisson deformation $(\{\pi_j(t)\},\{\eta_{jk}(t)\})$ with respect to $\{u_j\}$ coincides with $(\{\lambda_j(t)\},\{\theta_{jk}(t)\})$ with respect to $\{x_j\}$. Let the Poisson structure $\Lambda$ in (\ref{poisson}) be expressed in terms of local coordinates $u_j$ as $\Lambda=\Pi_j=\sum_{\alpha,\beta=1}^n \Pi_{\alpha \beta}^{j}(w_j,t)\frac{\partial}{\partial w_{j}^{\alpha}}\wedge \frac{\partial}{\partial w_{j}^{\beta}}$. Let $(w_k,t)\to (w_j,t)=(e_{jk}(w_k,t),t)$ be the coordinate transformation of $\{u_j\}$ on $\mathcal{U}_j\cap \mathcal{U}_k\ne \emptyset$. Now we set
\begin{align*}
\eta_{jk}(t)=\sum_{\alpha=1}^n \frac{\partial e_{jk}^\alpha(w_k,t)}{\partial t}\frac{\partial}{\partial w_j^{\alpha}},\,\,\,w_k=e_{kj}(w_j,t),\,\,\,\,\,\,\,\,\pi_j(t)=\sum_{\alpha,\beta=1}^n \frac{\partial \Pi_{\alpha \beta}^{j}(w_j,t)}{\partial t}\frac{\partial}{\partial w_{j}^{\alpha}}\wedge \frac{\partial}{\partial w_{j}^{\beta}}
\end{align*}
We show that $(\{\lambda_j(t)\}),\{\theta_{jk}(t)\})$ is cohomologous to $(\{\pi_j(t)\},\{\eta_{jk}(t)\})$. Let $w_j^{\alpha}=h_j^{\alpha}(z_j^1,...,z_j^n,t),\alpha=1,...,n$,  define the coordinate transformation from $x_j=(z_j,t)$ to $u_j=(w_j,t)$ which is a Poisson map.
So we have $\frac{\partial}{\partial z_j^r}=\sum_{a=1}^n \frac{\partial h_j^a}{\partial z_j^r}\frac{\partial}{\partial w_j^a}$ and the following relation holds 
\begin{align}\label{yh}
\Pi_{\alpha\beta}^j(h_j^1(z_j,t),...,h_{j}^n(z_j,t),t)=\sum_{r,s=1}^n g_{rs}^j(z_j,t)\frac{\partial h_{j}^{\alpha}}{\partial z_j^r}\frac{\partial h_{j}^{\beta}}{\partial z_j^s}.
\end{align}
Set $\theta_j(t)=\sum_{\alpha=1}^n \frac{\partial h_j^{\alpha}(z_j,t)}{\partial t} \frac{\partial }{\partial w_j^{\alpha}}, \,\,\,\,\,w_j^{\alpha}=h_j^{\alpha}(z_j,t)$. Then we claim that $(\lambda_j(t),\theta_{jk}(t))-(\pi_j(t),\eta_{jk}(t))=\theta_k(t)-\theta_j(t)-[\Lambda_t, \theta_j(t)]=-\delta(-\theta_j(t))+[\Lambda_t,-\theta_j(t)]$, which means $(\{\lambda_j(t)\},\{\theta_{jk}(t)\})$ is cohomologous to $(\{\pi_j(t)\},\{\eta_{jk}(t)\})$. Since $\delta(\{\theta_j(t)\})=\{\theta_{jk}(t)\}-\{\eta_{jk}(t)\}$(for the detail, see \cite{Kod05} p.191-192), we only need to see $\lambda_j(t)-\pi_j(t)+[\Lambda_t(=\Pi_t),\theta_j(t)]=0$. Equivalently,
{\small{\begin{align*}
\sum_{r,s=1}^n \frac{\partial g_{rs}^{j}(z_j,t)}{\partial t}\frac{\partial }{\partial z_j^r}\wedge \frac{\partial}{\partial z_j^s}-\sum_{\alpha,\beta=1}^n \frac{\partial \Pi_{\alpha\beta}^j(w_j,t)}{\partial t}\frac{\partial}{\partial w_j^{\alpha}}\wedge \frac{\partial}{\partial w_j^{\beta}}+[\sum_{\alpha,\beta=1}^n \Pi_{\alpha\beta}^j(w_j,t)\frac{\partial}{\partial w_j^{\alpha}}\wedge \frac{\partial}{\partial w_j^{\beta}},\sum_{c=1}^n \frac{\partial h_j^{c}(z_j,t)}{\partial t} \frac{\partial }{\partial w_j^{c}}]=0
\end{align*}}}
which follows from taking the derivative (\ref{yh}) with respect to $t$ as in the proof of the first claim. 
\end{proof}

\begin{definition}[(holomorphic) Poisson Kodaira-Spencer map]\label{mapping}
Let $(\mathcal{M},\Lambda,B,\omega)$ be a Poisson analytic family, where $B$ is a domain of $\mathbb{C}^m$. As in Remark $\ref{tt61}$, let $\mathcal{U}=\{\mathcal{U}_j\}$ be an open covering of $\mathcal{M}$, and $(z_j,t)$ a local complex coordinate system on $\mathcal{U}_j$. The Poisson structure $\Lambda$ is expressed as $\sum_{\alpha,\beta=1}^n g_{\alpha \beta}^{j}(z_j,t)\frac{\partial}{\partial z_{j}^{\alpha}}\wedge \frac{\partial}{\partial z_{j}^{\beta}}$ on $\mathcal{U}_j$ where $g^{j}_{\alpha \beta}(z_j,t)$ is a holomorphic function with $g_{\alpha\beta}^j(z_j,t)=-g_{\beta\alpha}^j(z_j,t)$. For a tangent vector $\frac{\partial}{\partial t}=\sum_{\lambda=1}^{m} c_{\lambda}\frac{\partial}{\partial t_{\lambda}},c_{\lambda} \in \mathbb{C}$, of $B$, we put  
\begin{align*}
\frac{\partial \Lambda_t}{\partial t}:=\sum_{\alpha,\beta=1}^n\left[\sum_{\lambda=1}^{m}c_{\lambda}\frac{\partial g_{\alpha \beta}^{j}(z_j,t)}{\partial t_{\lambda}}\right] \frac{\partial}{\partial z_{j}^{\alpha}}\wedge \frac{\partial}{\partial z_{j}^{\beta}}
\end{align*}
The $($holomorphic$)$ Poisson Kodaira-Spencer map is defined to be a $\mathbb{C}$-linear map  
\begin{align*}
\varphi_t:T_t(B) &\to \mathbb{H}^1(M_t,\Theta_{M_t}^\bullet)\\
\frac{\partial}{\partial t} &\mapsto \left[\rho_t\left(\frac{\partial}{\partial t}\right)\left(=\frac{\partial{M}_t}{\partial t}\right), \frac{\partial{\Lambda_t}}{\partial t}\right]=\frac{\partial (M_t,\Lambda_t)}{\partial t}
\end{align*}
where $\rho_t:T_t(B)\to H^1(M_t,\Theta_t)$ is the Kodaira-Spencer map of the complex analytic family $(\mathcal{M},B,\omega)$ $($see \cite{Kod05} $p.201$$)$.
\end{definition}

\section{Integrability condition}\label{section4}

In a complex analytic family $(\mathcal{M},B,\omega)$ of deformations of  a complex manifold $M=\omega^{-1}(0)$, the deformations near $M$ are represented by  $C^{\infty}$ vector $(1,0)$-forms $\varphi(t) \in A^{0,1}(M,T_M)$ on $M$ satisfying $\varphi(0)=0$ and the integrability condition $\bar{\partial} \varphi(t)-\frac{1}{2}[\varphi(t),\varphi(t)]=0$ where $t \in \Delta$ a sufficiently small polydisk in $B$ (see \cite{Kod05} section \S 5.3). In this section, we show that in a Poisson analytic family $(\mathcal{M},B,\Lambda,\omega)$ of deformations of a compact holomorphic Poisson manifold $(M,\Lambda_0)=\omega^{-1}(0)$, the deformations near $(M,\Lambda_0)$ are represented by $C^{\infty}$ vector $(0,1)$-forms $\varphi(t)\in A^{0,1}(M,T_M)$ and $C^{\infty}$ bivectors $\Lambda(t)\in A^{0,0}(M,\wedge^2 T_M)$ satisfying $\varphi(0)=0$, $\Lambda(0)=\Lambda_0$ and the integrability condition $\bar{\partial}(\varphi(t)+\Lambda(t))+\frac{1}{2}[\varphi(t)+\Lambda(t),\varphi(t)+\Lambda(t)]=0$. To deduce the integrability condition, we extend Kodaira's approach (\cite{Kod05} section \S 5.3) in the context of a Poisson analytic family.

\subsection{Preliminaries}\label{prill} \

We extend the argument of \cite{Kod05} p.259-261 (to which we refer for the detail) in the context of a Poisson analytic family. We tried to maintain notational consistency with \cite{Kod05}. 

Let $(\mathcal{M}, \Lambda, B,\omega)$ be a Poisson analytic family of compact Poisson holomorphic manifolds, where $B$ is a domain of $\mathbb{C}^m$ containing the origin $0$. Define $|t|=\max_{\lambda}|t_{\lambda}|$ for $t=(t_1,...,t_m)\in \mathbb{C}^m$, and let $\Delta=\Delta_r =\{t\in \mathbb{C}^m||t|<r\}$ the polydisk of radius $r>0$. If we take a sufficiently small $\Delta \subset B$, then $(\mathcal{M}_{\Delta},\Lambda_\Delta)=\omega^{-1}(\Delta)$ is represented in the form
\begin{align*}
(\mathcal{M}_{\Delta},\Lambda_\Delta)=\bigcup_j (U_j\times \Delta,\Lambda|_{U_j\times \Delta})
\end{align*}
We denote a point of $U_j$ by $\xi_j=(\xi_j^1,...,\xi_j^n)$ and its holomorphic Poisson structure $\Lambda|_{U_j\times \Delta}$ by $\sum_{\alpha,\beta=1}^ng_{\alpha \beta}^j(\xi_j,t) \frac{\partial}{\partial \xi_j^{\alpha}}\wedge \frac{\partial}{\partial \xi_j^{\beta}}$ on $U_j\times \Delta$ with $g_{\alpha\beta}^j(\xi_j,t)=-g_{\beta\alpha}^j(\xi_j,t)$. For simplicity, we assume that $U_j=\{\xi_j\in \mathbb{C}^m||\xi_j|<1\}$ where $|\xi|=\max_a|\xi_j^a|$. $(\xi_j,t)\in U_j\times \Delta$ and $(\xi_k,t)\in U_k\times \Delta$ are the same point on $\mathcal{M}_{\Delta}$ if $\xi_j^{\alpha}=f_{jk}^{\alpha}(\xi_k,t)$, $\alpha=1,...,n$ where $f_{jk}(\xi_k,t)$ is a Poisson holomorphic map of $\xi_{k}^1,...,\xi_k^n,t_1,...,t_m$, defined on $U_k\times \Delta \cap U_j\times \Delta$, and so we have the following relation
\begin{align}\label{vv4}
g_{\alpha \beta}^j(f_{jk}^1(\xi_k,t),...,f_{jk}^n(\xi_k,t))=\sum_{r,s=1}^n g_{rs}^k(\xi_k,t)\frac{\partial f_{jk}^{\alpha}}{\partial \xi_k^r}\frac{\partial f_{jk}^{\beta}}{\partial \xi_k^s}
\end{align}
We note that $\omega^{-1}(t_0)=(M_{t_0},\Lambda_{t_0})=\bigcup_j (U_j,\sum_{\alpha,\beta=1}^ng_{\alpha \beta}^j(\xi_j,t_0) \frac{\partial}{\partial \xi_j^{\alpha}}\wedge \frac{\partial}{\partial \xi_j^{\beta}})$ for $t_0\in\Delta$

By \cite{Kod05} Theorem 2.3, when we ignore complex structures and Poisson structures, $M_t$ is diffeomorphic to $M_0=\omega^{-1}(0)$ as differentiable manifolds for each $t\in \Delta$. We put $M:=M_0$. By \cite{Kod05} Theorem 2.5, if we take a sufficiently small $\Delta$, there is a diffeomorphism $\Psi$ of $M\times \Delta$ onto $\mathcal{M}_{\Delta}$ as differentiable manifolds such that $\omega\circ \Psi$ is the projection $M\times \Delta \to \Delta$. Let $z=(z_1,...,z_n)$ be local complex coordinates of $M=M_0$. Then we have $\omega\circ \Psi(z,t)=t,\,\,\,\,\, t\in \Delta$. For $\Psi(z,t)\in U_j\times \Delta$, put 
\begin{align}\label{pp00}
\Psi(z,t)=(\xi_j^1(z,t),...,\xi_j^n(z,t),t_1,...,t_m).
\end{align}
Then each component $\xi_j^{\alpha}=\xi_j^{\alpha}(z,t)$, $\alpha=1,...,n$ is a $C^{\infty}$ function. If we identify $\mathcal{M}_{\Delta}=\Psi(M\times \Delta)$ with $M\times \Delta$ via $\Psi$, $(\mathcal{M}_{\Delta},\Lambda_\Delta)$ is considered as a holomorphic Poisson manifold with the complex structure defined on the $C^{\infty}$ manifold $M\times \Delta$ by the system of local coordinates on $U_j\times \Delta$
\begin{align*}
\{(\xi_j,t)|j=1,2,3,...\},\,\,\,\,\, (\xi_j,t)=(\xi_j^1(z,t),...,\xi_j^n(z,t),t_1,...,t_m).
\end{align*}
and the holomorphic Poisson structure given by on $U_j\times \Delta$
\begin{align}\label{tt23}
\{\sum_{\alpha,\beta=1}^n g_{\alpha \beta}^j(\xi_j(z,t),t)\frac{\partial}{\partial \xi_j^{\alpha}}\wedge \frac{\partial}{\partial \xi_j^{\beta}}|j=1,2,3,...\}
\end{align}

We note that since $(z_1,...,z_n)$ and $(\xi_j^1(z,0),...,\xi_j^n(z,0))$ are local complex coordinates on $M=M_0$,
\begin{align}\label{holomorphic}
\text{$\xi_j^{\alpha}(z,0)$ are holomorphic functions of $z_1,...,z_n$, $\alpha=1,...,n$}
\end{align}
 We also note that if we take $\Delta$ sufficiently small, we have
\begin{align}\label{det}
\det\left(\frac{\partial \xi_j^{\alpha}(z,t)}{\partial z_{\lambda}}\right)_{\alpha,\lambda=1,...,n}\ne 0
\end{align}
for any $t\in \Delta$.

With this preparation, we identify the holomorphic Poisson deformations near $(M,\Lambda_0)$ in the Poisson analytic family $(\mathcal{M},\Lambda,B,\omega)$ with $\varphi(t)+\Lambda(t)$ where $\varphi(t)$ is a $C^{\infty}$ vector $(0,1)$-form and $\Lambda(t)$ is a $C^{\infty}$ bivector on $M$ for $t\in \Delta$.

\subsection{Identification of the deformations of complex structures with $\varphi(t)\in A^{0,1}(M,T_M)$}\

Put $\mathcal{U}_j=\Psi^{-1}(U_j\times \Delta)$. Then $\mathcal{U}_j\subset M\times \Delta$ is the domain of $\xi_j^{\alpha}(z,t)$. From $(\ref{det})$, we can define a $(0,1)$-form $\varphi^{\lambda}_j(z,t)=\sum_{v=1}^n \varphi^{\lambda}_{jv}(z,t)d\bar{z}_v$ in the following way:
 
\begin{equation*}
\left(
\begin{matrix}
\varphi_j^1(z,t)\\
\vdots \\
\varphi_j^n(z,t)
\end{matrix}
\right)
:=
\left(
\begin{matrix}
\frac{\partial \xi_j^1}{\partial z_1} & \dots & \frac{\partial \xi_j^1}{\partial z_n}\\
\vdots & \vdots\\
\frac{\partial \xi_j^n}{\partial z_1} & \dots & \frac{\partial \xi_j^n}{\partial z_n}
\end{matrix}
\right)^{-1}
\left(
\begin{matrix}
\bar{\partial} \xi_j^1\\
\vdots \\
\bar{\partial} \xi_j^n
\end{matrix}
\right)
\end{equation*}

Then the coefficients $\varphi_{jv}^{\alpha}(z,t)$ are $C^{\infty}$ functions on $\mathcal{U}_j$ and $\bar{\partial}\xi_j^{\alpha}(z,t)=\sum_{\lambda=1}^{n} \varphi_j^{\lambda}(z,t)\frac{\partial \xi_j^{\alpha}(z,t)}{\partial z_{\lambda}},\alpha=1,...,n$. So we have

\begin{equation}\label{matrix1}
\frac{\partial \xi_j^{\alpha}}{\partial \bar{z}_v}=\sum_{\lambda=1}^n \varphi_{jv}^\lambda(z,t) \frac{\partial \xi_j^\alpha}{\partial z_\lambda}
\end{equation}
\begin{lemma}\label{c}
On $\mathcal{U}_j \cap \mathcal{U}_k$, we have
\begin{align*}
\sum_{\lambda=1}^n \varphi_j^{\lambda}(z,t)\frac{\partial}{\partial z_{\lambda}}=\sum_{\lambda=1}^n \varphi_k^{\lambda}(z,t)\frac{\partial}{\partial z_{\lambda}}
\end{align*}
\end{lemma}
\begin{proof}
See \cite{Kod05} p.262.
\end{proof}

If for $(z,t)\in \mathcal{U}_j$, we define
\begin{align}\label{b}
\varphi(z,t):=\sum_{\lambda=1}^n \varphi_j^{\lambda}(z,t) \frac{\partial}{\partial z_{\lambda}}=\sum_{\lambda=1}^n\varphi^\lambda(z,t)\frac{\partial}{\partial z_\lambda}=\sum_{v,\lambda=1}^n  \varphi_v^{\lambda}(z,t) d\bar{z}_v \frac{\partial}{\partial z_{\lambda}}
\end{align}
 By Lemma \ref{c}, $\varphi(t)=\varphi(z,t)\in A^{0,1}(M,T_M)$ is a $C^{\infty}$ vector $(0,1)$-form on $M$ for every $t\in\Delta$ and we have
\begin{align}\label{tt07}
\text{$\varphi(0)=0$,\,\,\,\,\,\,\,\,$\bar{\partial}\varphi(t)-\frac{1}{2}[\varphi(t),\varphi(t)]=0$}
\end{align}
(see \cite{Kod05} p.263,p.265). We also point out that
 
\begin{theorem}\label{text}
If we take a sufficiently small polydisk $\Delta$ as in subsection $\ref{prill}$, then for $t\in \Delta$, a local $C^{\infty}$ function $f$ on $M$ is holomorphic  with respect to the complex structure $M_t$ if and only if $f$ satisfies the equation
\begin{align*}
(\bar{\partial}-\varphi(t))f=0
\end{align*}
\end{theorem}
\begin{proof}
See \cite{Kod05} Theorem 5.3 p.263.
\end{proof}

\subsection{Identification of the deformations of Poisson structures with $\Lambda(t)\in A^{0,0}(M,\wedge^2 T_M)$}\

For the holomorphic Poisson structure $\sum_{\alpha,\beta=1}^n g_{\alpha\beta}^j(\xi_j(z,t),t) \frac{\partial}{\partial \xi_j^{\beta}}\wedge \frac{\partial}{\partial \xi_j^{\beta}}$ on each $U_j\times \Delta$ from $(\ref{tt23})$, there exists the unique bivector field  $\Lambda_j(z,t):=\sum_{r,s=1}^n h_{rs}^j(z,t)\frac{\partial}{\partial z_r}\wedge \frac{\partial}{\partial z_s}$ on  $\mathcal{U}_j=\Psi^{-1}(U_j\times \Delta)$ such that 
\begin{align}\label{tt304}
\sum_{r,s=1}^n h_{rs}^j(z,t)\frac{\partial \xi_j^{\alpha}}{\partial z_r}\frac{\partial \xi_j^{\beta}}{\partial z_s}=g_{\alpha\beta}^j(\xi_j(z,t),t).
\end{align}
Indeed, from $(\ref{det})$,
we set
{\tiny{\begin{equation*}
\left(
\begin{matrix}
h_{11}^j(z,t)& \dots & h_{1n}^j(z,t)\\
\vdots & \vdots &\vdots\\
h_{n1}^j(z,t)& \dots & h_{nn}^j(z,t)
\end{matrix}
\right)
:=
\left(
\begin{matrix}
\frac{\partial \xi_j^1}{\partial z_1} & \dots & \frac{\partial \xi_j^1}{\partial z_n}\\
\vdots & \vdots &\vdots\\
\frac{\partial \xi_j^n}{\partial z_1} & \dots & \frac{\partial \xi_j^n}{\partial z_n}
\end{matrix}
\right)^{-1}
\left(
\begin{matrix}
g_{11}^j(\xi_j(z,t))& \dots & g_{1n}^j(\xi_j(z,t))\\
\vdots & \vdots &\vdots\\
g_{n1}^j(\xi_j(z,t)) & \dots & g_{nn}^j(\xi_j(z,t))
\end{matrix}
\right)
\left(
\begin{matrix}
\frac{\partial \xi_j^1}{\partial z_1} & \dots & \frac{\partial \xi_j^n}{\partial z_1}\\
\vdots & \vdots &\vdots\\
\frac{\partial \xi_j^1}{\partial z_n} & \dots & \frac{\partial \xi_j^n}{\partial z_n}
\end{matrix}
\right)^{-1}
\end{equation*}}}
We note that since $g_{\alpha\beta}^j(\xi_j(z,t))=-g_{\beta\alpha}^j(\xi_j(z,t))$, we have $h_{rs}^j(z,t)=-h_{sr}^j(z,t)$.
\begin{lemma}\label{e}
On $\mathcal{U}_j\cap \mathcal{U}_k$, we have $h_{rs}^j(z,t)=h_{rs}^k(z,t)$. 
 
\end{lemma}

\begin{proof}
From $(\ref{tt304})$, $(\ref{vv4})$ and $\frac{\partial \xi_j^{\alpha}}{\partial z_r}=\sum_{p=1}^n\frac{\partial \xi_k^p}{\partial z_r}\frac{\partial \xi_j^{\alpha}}{\partial \xi_k^p}$, we have
\begin{align*}
\sum_{r,s=1}^n h_{rs}^j(z,t)\frac{\partial \xi_j^{\alpha}}{\partial z_r}\frac{\partial \xi_j^{\beta}}{\partial z_s}&=g_{\alpha\beta}^j(\xi_j(z,t),t)=\sum_{p,q=1}^n g_{pq}^k(\xi_k(z,t),t)\frac{\partial \xi_j^{\alpha}}{\partial \xi_k^p}\frac{\partial \xi_j^{\beta}}{\partial \xi_k^q}\\
&=\sum_{p,q,r,s=1}^n h_{rs}^k(z,t)\frac{\partial \xi_k^{p}}{\partial z_r}\frac{\partial \xi_k^{q}}{\partial z_s}\frac{\partial \xi_j^{\alpha}}{\partial \xi_k^p}\frac{\partial \xi_j^{\beta}}{\partial \xi_k^q}=\sum_{r,s=1}^n h_{rs}^k(z,t)\frac{\partial \xi_j^{\alpha}}{\partial z_r}\frac{\partial \xi_j^{\beta}}{\partial z_s}.
\end{align*}
From $(\ref{det})$, we have $h_{rs}^j(z,t)=h_{rs}^k(z,t)$.
\end{proof}

If for $(z,t)\in \mathcal{U}_j$, we define 
\begin{align}\label{f}
\Lambda(z,t):=\sum_{r,s=1}^n h_{rs}^j(z,t)\frac{\partial}{\partial z_r}\wedge \frac{\partial}{\partial z_s}=\sum_{r,s=1}^n h_{rs}(z,t)\frac{\partial}{\partial z_r}\wedge \frac{\partial}{\partial z_s}.
\end{align}
By Lemma \ref{e}, $\Lambda(t):=\Lambda(z,t)\in A^{0,0}(M,\wedge^2 T_M)$ is a $C^{\infty}$ bivector field on $M$ for every $t\in \Delta$ with $\Lambda(0)=\Lambda_0$.

\begin{theorem}\label{1thm}
If we take a sufficiently small polydisk $\Delta$ as in subsection $\ref{prill}$, then for the Poisson structure $\sum_{\alpha,\beta=1}^n g_{\alpha\beta}^j(\xi_j,t) \frac{\partial}{\partial \xi_j^{\beta}}\wedge \frac{\partial}{\partial \xi_j^{\beta}}$ on $U_j\times \Delta$ for each $j$, there exists the unique bivector field  $\Lambda_j(t)=\sum_{r,s=1}^n h_{rs}^j(z,t)\frac{\partial}{\partial z_r}\wedge \frac{\partial}{\partial z_s}$ on $\mathcal{U}_j$ satisfying
\begin{enumerate}
\item $\sum_{r,s=1}^n h_{rs}^j(z,t)\frac{\partial \xi_j^{\alpha}}{\partial z_r}\frac{\partial \xi_j^{\beta}}{\partial z_s}=g_{\alpha\beta}^j(\xi_j(z,t),t)$
\item $\Lambda_j(t)$ are glued together to define a $C^{\infty}$ bivector field $\Lambda(t)$ on $M\times \Delta$ 
\item for each $j$, $[\Lambda_j(t),\Lambda_j(t)]=0$. Hence we have $[\Lambda(t),\Lambda(t)]=0$
\end{enumerate}
\end{theorem}

We will use the following lemma to prove the theorem.

\begin{lemma}\label{formula}
If $\sigma=\sum_{\alpha,\beta=1}^n \sigma_{\alpha\beta}\frac{\partial }{\partial z_\alpha}\wedge \frac{\partial}{\partial z_\beta}$ with $\sigma_{\alpha\beta}=-\sigma_{\beta\alpha}$, then 
$[\sigma,\sigma]=0$ is equivalent to 
\begin{align*}
\sum_{l=1}^n (\sigma_{lk}\frac{\partial \sigma_{ij}}{\partial z_l}+\sigma_{li}\frac{\partial \sigma_{jk}}{\partial z_l}+\sigma_{lj}\frac{\partial \sigma_{ki}}{\partial z_l})=0
\end{align*}
for each $1\leq i,j,k \leq n$.
\end{lemma}

\begin{proof}[Proof of Theorem $\ref{1thm}$]
We have already showed $(1)$ and $(2)$. It remains to show $(3)$. We note that 
\begin{align}\label{tt55}
[\sum_{\alpha,\beta=1}^n g_{\alpha\beta}^j(\xi_j,t) \frac{\partial}{\partial \xi_j^{\alpha}}\wedge\frac{\partial}{\partial \xi_j^{\beta}},\sum_{\alpha,\beta=1}^n g_{\alpha\beta}^j(\xi_j,t) \frac{\partial}{\partial \xi_j^{\alpha}}\wedge\frac{\partial}{\partial \xi_j^{\beta}}]=0.
\end{align}
 Since $g_{\alpha\beta}^j(\xi_j(z,t),t)=\sum_{a,b=1}^n h_{ab}^j(z,t)\frac{\partial \xi_j^{\alpha}}{\partial z_a}\frac{\partial \xi_j^{\beta}}{\partial z_b}$ is holomorphic with respect to $\xi_j=(\xi_j^{\alpha}),\alpha=1,...,n$, we have
\begin{align}\label{nb1}
\frac{\partial}{\bar{\partial} \xi_j^{\alpha}}\left( \sum_{a,b=1}^n h_{ab}^j(z,t)\frac{\partial \xi_j^{\alpha}}{\partial z_a}\frac{\partial \xi_j^{\beta}}{\partial z_b}\right)=\sum_{a,b=1}^n \frac{\partial}{\bar{\partial}\xi_j^{\alpha}}\left(h_{ab}^j(z,t)\frac{\partial \xi_j^{\alpha}}{\partial z_a}\frac{\partial \xi_j^{\beta}}{\partial z_b} \right)=0
\end{align}
 In the following, for simplicity, we denote $\xi_j^{\alpha}(z_j,t)$ by $\xi_{\alpha}$ and $h_{ab}^j(z,t)$ by $h_{ab}$. By (\ref{tt55}), Lemma \ref{formula} and $(\ref{nb1})$, and by the property $h_{ab}=-h_{ba}$ and $\frac{\partial}{\partial z_a}=\sum_{l=1}^n \frac{\partial \xi_l}{\partial z_a}\frac{\partial}{\partial \xi_l}+\sum_{l=1}^n \frac{\partial \bar{\xi}_l}{\partial z_a}\frac{\partial}{\partial \bar{\xi}_l}$, we have

\begin{align*}
0=&\sum_{a,b,c,d,l=1}^n( h_{ab}\frac{\partial \xi_l}{\partial z_a}\frac{\partial \xi_k}{\partial z_b}\frac{\partial}{\partial \xi_l}\left(h_{cd}\frac{\partial \xi_i}{\partial z_c}\frac{\partial \xi_j}{\partial z_d}\right)+h_{ab}\frac{\partial \xi_l}{\partial z_a}\frac{\partial \xi_i}{\partial z_b}\frac{\partial}{\partial \xi_l}\left(h_{cd}\frac{\partial \xi_j}{\partial z_c}\frac{\partial \xi_k}{\partial z_d}\right)+h_{ab}\frac{\partial \xi_l}{\partial z_a}\frac{\partial \xi_j}{\partial z_b}\frac{\partial}{\partial \xi_l}\left(h_{cd}\frac{\partial \xi_k}{\partial z_c}\frac{\partial \xi_i}{\partial z_d}\right))\\
+&\sum_{a,b,c,d,l=1}^n (h_{ab}\frac{\partial \bar{\xi}_l}{\partial z_a}\frac{\partial \xi_k}{\partial z_b}\frac{\partial}{\partial \bar{\xi}_l}\left(h_{cd}\frac{\partial \xi_i}{\partial z_c}\frac{\partial \xi_j}{\partial z_d}\right)+h_{ab}\frac{\partial \bar{\xi}_l}{\partial z_a}\frac{\partial \xi_i}{\partial z_b}\frac{\partial}{\partial \bar{\xi}_l}\left(h_{cd}\frac{\partial \xi_j}{\partial z_c}\frac{\partial \xi_k}{\partial z_d}\right)+h_{ab}\frac{\partial \bar{\xi}_l}{\partial z_a}\frac{\partial \xi_j}{\partial z_b}\frac{\partial}{\partial \bar{\xi}_l}\left(h_{cd}\frac{\partial \xi_k}{\partial z_c}\frac{\partial \xi_i}{\partial z_d}\right))\\
=&\sum_{a,b,c,d,l=1}^n (h_{ab}\frac{\partial \xi_l}{\partial z_a}\frac{\partial \xi_k}{\partial z_b}\frac{\partial h_{cd}}{\partial \xi_l}\frac{\partial \xi_i}{\partial z_c}\frac{\partial \xi_j}{\partial z_d}+h_{ab}\frac{\partial \xi_l}{\partial z_a}\frac{\partial \xi_k}{\partial z_b}h_{cd}\frac{\partial}{\partial \xi_l}\left(\frac{\partial \xi_i}{\partial z_c}\right)\frac{\partial \xi_j}{\partial z_d}+h_{ab}\frac{\partial \xi_l}{\partial z_a}\frac{\partial \xi_k}{\partial z_b}h_{cd}\frac{\partial \xi_i}{\partial z_c}\frac{\partial}{\partial \xi_l}\left(\frac{\partial \xi_j}{\partial z_d}\right))\\
+&\sum_{a,b,c,d,l=1}^n(h_{ab}\frac{\partial \xi_l}{\partial z_a}\frac{\partial \xi_i}{\partial z_b}\frac{\partial h_{cd}}{\partial \xi_l}\frac{\partial \xi_j}{\partial z_c}\frac{\partial \xi_k}{\partial z_d}+h_{ab}\frac{\partial \xi_l}{\partial z_a}\frac{\partial \xi_i}{\partial z_b}h_{cd}\frac{\partial}{\partial \xi_l}\left(\frac{\partial \xi_j}{\partial z_c}\right)\frac{\partial \xi_k}{\partial z_d}+h_{ab}\frac{\partial \xi_l}{\partial z_a}\frac{\partial \xi_i}{\partial z_b}h_{cd}\frac{\partial \xi_j}{\partial z_c}\frac{\partial}{\partial \xi_l}\left(\frac{\partial \xi_k}{\partial z_d}\right))\\
+&\sum_{a,b,c,d,l=1}^n(h_{ab}\frac{\partial \xi_l}{\partial z_a}\frac{\partial \xi_j}{\partial z_b}\frac{\partial h_{cd}}{\partial \xi_l}\frac{\partial \xi_k}{\partial z_c}\frac{\partial \xi_i}{\partial z_d}+h_{ab}\frac{\partial \xi_l}{\partial z_a}\frac{\partial \xi_j}{\partial z_b}h_{cd}\frac{\partial}{\partial \xi_l}\left(\frac{\partial \xi_k}{\partial z_c}\right)\frac{\partial \xi_i}{\partial z_d}+h_{ab}\frac{\partial \xi_l}{\partial z_a}\frac{\partial \xi_j}{\partial z_b}h_{cd}\frac{\partial \xi_k}{\partial z_c}\frac{\partial}{\partial \xi_l}\left(\frac{\partial \xi_i}{\partial z_d}\right))\\
+&\sum_{a,b,c,d,l=1}^n(h_{ab}\frac{\partial \bar{\xi}_l}{\partial z_a}\frac{\partial \xi_k}{\partial z_b}\frac{\partial h_{cd}}{\partial \bar{\xi}_l}\frac{\partial \xi_i}{\partial z_c}\frac{\partial \xi_j}{\partial z_d}+h_{ab}\frac{\partial \bar{\xi}_l}{\partial z_a}\frac{\partial \xi_k}{\partial z_b}h_{cd}\frac{\partial}{\partial \bar{\xi}_l}\left(\frac{\partial \xi_i}{\partial z_c}\right)\frac{\partial \xi_j}{\partial z_d}+h_{ab}\frac{\partial \bar{\xi}_l}{\partial z_a}\frac{\partial \xi_k}{\partial z_b}h_{cd}\frac{\partial \xi_i}{\partial z_c}\frac{\partial}{\partial \bar{\xi}_l}\left(\frac{\partial \xi_j}{\partial z_d}\right))\\
+&\sum_{a,b,c,d,l=1}^n(h_{ab}\frac{\partial \bar{\xi}_l}{\partial z_a}\frac{\partial \xi_i}{\partial z_b}\frac{\partial h_{cd}}{\partial \bar{\xi}_l}\frac{\partial \xi_j}{\partial z_c}\frac{\partial \xi_k}{\partial z_d}+h_{ab}\frac{\partial \bar{\xi}_l}{\partial z_a}\frac{\partial \xi_i}{\partial z_b}h_{cd}\frac{\partial}{\partial \bar{\xi}_l}\left(\frac{\partial \xi_j}{\partial z_c}\right)\frac{\partial \xi_k}{\partial z_d}+h_{ab}\frac{\partial \bar{\xi}_l}{\partial z_a}\frac{\partial \xi_i}{\partial z_b}h_{cd}\frac{\partial \xi_j}{\partial z_c}\frac{\partial}{\partial \bar{\xi}_l}\left(\frac{\partial \xi_k}{\partial z_d}\right))\\
+&\sum_{a,b,c,d,l=1}^n(h_{ab}\frac{\partial \bar{\xi}_l}{\partial z_a}\frac{\partial \xi_j}{\partial z_b}\frac{\partial h_{cd}}{\partial \bar{\xi}_l}\frac{\partial \xi_k}{\partial z_c}\frac{\partial \xi_i}{\partial z_d}+h_{ab}\frac{\partial \bar{\xi}_l}{\partial z_a}\frac{\partial \xi_j}{\partial z_b}h_{cd}\frac{\partial}{\partial \bar{\xi}_l}\left(\frac{\partial \xi_k}{\partial z_c}\right)\frac{\partial \xi_i}{\partial z_d}+h_{ab}\frac{\partial \bar{\xi}_l}{\partial z_a}\frac{\partial \xi_j}{\partial z_b}h_{cd}\frac{\partial \xi_k}{\partial z_c}\frac{\partial}{\partial \bar{\xi}_l}\left(\frac{\partial \xi_i}{\partial z_d}\right))\\
=&\sum_{a,b,c,d=1}^n(h_{ab}\frac{\partial h_{cd}}{\partial z_a}\frac{\partial \xi_k}{\partial z_b}\frac{\partial \xi_i}{\partial z_c}\frac{\partial \xi_j}{\partial z_d}+h_{ab}\frac{\partial \xi_k}{\partial z_b}h_{cd}\frac{\partial^2 \xi_i}{\partial z_a\partial z_c}\frac{\partial \xi_j}{\partial z_d}+h_{ab}\frac{\partial \xi_k}{\partial z_b}h_{cd}\frac{\partial \xi_i}{\partial z_c}\frac{\partial^2 \xi_j}{\partial z_a\partial z_d})\\
+&\sum_{a,b,c,d=1}^n (h_{ab}\frac{\partial h_{cd}}{\partial z_a}\frac{\partial \xi_i}{\partial z_b}\frac{\partial \xi_j}{\partial z_c}\frac{\partial \xi_k}{\partial z_d}+h_{ab}\frac{\partial \xi_i}{\partial z_b}h_{cd}\frac{\partial^2 \xi_j}{\partial z_a\partial z_c}\frac{\partial \xi_k}{\partial z_d}+h_{ab}\frac{\partial \xi_i}{\partial z_b}h_{cd}\frac{\partial \xi_j}{\partial z_c}\frac{\partial^2 \xi_k}{\partial z_a\partial z_d})\\
+&\sum_{a,b,c,d=1}^n(h_{ab}\frac{\partial h_{cd}}{\partial z_a}\frac{\partial \xi_j}{\partial z_b}\frac{\partial \xi_k}{\partial z_c}\frac{\partial \xi_i}{\partial z_d}+h_{ab}\frac{\partial \xi_j}{\partial z_b}h_{cd}\frac{\partial^2 \xi_k}{\partial z_a\partial z_c}\frac{\partial \xi_i}{\partial z_d}+h_{ab}\frac{\partial \xi_j}{\partial z_b}h_{cd}\frac{\partial \xi_k}{\partial z_c}\frac{\partial^2 \xi_i}{\partial z_a\partial z_d})\\
=&\sum_{a,b,c,d=1}^n(h_{ab}\frac{\partial h_{cd}}{\partial z_a}\frac{\partial \xi_k}{\partial z_b}\frac{\partial \xi_i}{\partial z_c}\frac{\partial \xi_j}{\partial z_d}+h_{ab}\frac{\partial h_{cd}}{\partial z_a}\frac{\partial \xi_i}{\partial z_b}\frac{\partial \xi_j}{\partial z_c}\frac{\partial \xi_k}{\partial z_d}+h_{ab}\frac{\partial h_{cd}}{\partial z_a}\frac{\partial \xi_j}{\partial z_b}\frac{\partial \xi_k}{\partial z_c}\frac{\partial \xi_i}{\partial z_d})\\
=&\sum_{a,b,c,d=1}^n\left(h_{ab}\frac{\partial h_{cd}}{\partial z_a}+h_{ac}\frac{\partial h_{db}}{\partial z_a}+h_{ad}\frac{\partial h_{bc}}{\partial z_a}\right)\frac{\partial \xi_i}{\partial z_c}\frac{\partial \xi_j}{\partial z_d}\frac{\partial \xi_k}{\partial z_b}
\end{align*}
From $(\ref{det})$, we have $\sum_{a=1}^n h_{ab}\frac{\partial h_{cd}}{\partial z_a}+h_{ac}\frac{\partial h_{db}}{\partial z_a}+h_{ad}\frac{\partial h_{bc}}{\partial z_a}=0$ for each $b,c,d$. So by Lemma \ref{formula}, $[\Lambda_j(t),\Lambda_j(t)]=0$.
\end{proof}

\begin{remark}\label{renj}
For the compact holomorphic Poisson manifold $(M_t,\Lambda_t)$ for each $t\in \Delta$ in the Poisson analytic family $(\mathcal{M},\Lambda,B,\omega)$, we showed that there exists a bivector field $\Lambda(t)$ on $M=M_0$ with $[\Lambda(t),\Lambda(t)]=0$ for $t\in \Delta$ by Theorem $\ref{1thm}$. Let $J_t:T_{\mathbb{R}}M\to T_{\mathbb{R}}M$ with $J_t^2=-id$ be the almost complex structure associated to the complex structure $M_t$ $($induced by $\varphi(t))$ where $T_{\mathbb{R}}M$ is a real tangent bundle of the underlying differentiable manifold $M$. Then $J_t$ induces a type decomposition of complexified tangent bundle $T_{\mathbb{C}} M=T^{1,0}_{M_t}\oplus T^{0,1}_{M_t}$ $($see \cite{Kob69} Chapter IX section 2$)$ so that we have $\wedge^2 T_{\mathbb{C}} M=\wedge^2 T_{M_t}^{1,0} \oplus T_{M_t}^{1,0}\otimes T_{M_t}^{0,1} \oplus \wedge^2 T_{M_t}^{0,1}$. If $\Lambda$ is a $C^{\infty}$ section of $\wedge^2 T_{\mathbb{C}} M $ on $M$, then we denote by $\Lambda^{2,0}$ the component of $\wedge^2 T_{M_t}^{1,0}$, by $\Lambda^{1,1}$ the component of $T_{M_t}^{0,1}\otimes T_{M_t}^{0,1}$, and by $\Lambda^{0,2}$ the component of $\wedge^2 T_{M_t}^{0,1}$.  So we have $\Lambda=\Lambda^{2,0}+\Lambda^{1,1}+\Lambda^{0,2}$. We call $\Lambda^{2,0}$ the type $(2,0)$-part of $\Lambda$. With this notation, the type $(2,0)$-part of  $\Psi_{*} \Lambda(t)$ is $\Lambda_t$ for $t\in \Delta$, where $\Psi_*\Lambda(t)$ is the bivector field induced from $\Lambda(t)$ via diffeomorphism $\Psi$ in $(\ref{pp00})$. So we can say that $\Lambda(t)^{2,0}=\Lambda_t$.
\end{remark}

\begin{remark}\label{tt45}
Let $\Lambda$ be a $C^{\infty}$-section of $\wedge^k T_\mathbb{C}M$. From $\wedge^k T_\mathbb{C} M=\bigoplus_{p+q=k} \wedge^p T_{M_t}^{1,0}\otimes \wedge^q T_{M_t}^{0,1}$, we can define the type $(p,q)$ part $\Lambda^{p,q}$ of $\Lambda$ in an obvious way as in Remark $\ref{renj}$.
\end{remark}

Next we discuss the condition when a given $C^{\infty}$ bivector field $\Lambda\in A^{0,0}(M, \wedge^2 T_M)$ on $M$ with $[\Lambda,\Lambda]=0$ gives a holomorphic bivector field $\Lambda^{2,0}\in A^{0,0}(M_t, \wedge^2 T_{M_t})$ with respect to the complex structure $M_t$ induced by $\varphi(t)$. Before proceeding our discussion, we recall the Schouten bracket $[-,-]$ on $\bigoplus_{i\geq 0}\bigoplus_{p+q-1=i,p\geq 0,q\geq 1} A^{0,p}(M,\wedge^q T_M)$ (see (\ref{tt76})) which we need for the computation of the integrability condition (\ref{yghb}). The Schouten bracket $[-,-]$ is defined in the following way:
\begin{align*}
[-,-]:A^{0,p}(M,\wedge^q T_M)\times A^{0,p'}(M,\wedge^{q'} T_M)\to A^{0,p+p'}(M,\wedge^{q+q'-1} T_M)
\end{align*}
In local coordinates it is given by
\begin{align}\label{tt00}
[fd\bar{z}_I\frac{\partial}{\partial z_J},gd\bar{z}_K\frac{\partial}{\partial z_L}]=(-1)^{|K|(|J|+1)} d\bar{z}_I\wedge d\bar{z}_K [f\frac{\partial}{\partial z_J},g\frac{\partial}{\partial z_L}]\end{align}
where $f,g$ are $C^{\infty}$ functions on $M$ and $d\bar{z}_I=d\bar{z}_{i_1}\wedge \cdots \wedge d\bar{z}_{i_{|I|}}$, $\frac{\partial}{\partial z_J}=\frac{\partial}{\partial z_{j_1}}\wedge \cdots \wedge \frac{\partial}{\partial z_{j_{|L|}}}$ (similarly for $d\bar{z}_K, \frac{\partial}{\partial z_L})$. Then
\begin{align}
\mathfrak{g}=(\bigoplus_{i\geq0} g_i, g_i=\bigoplus_{p+q-1=i,p\geq 0, q\geq 1} A^{0,p}(M,\wedge^q T_M),L=\bar{\partial}+[\Lambda,-],[-,-]),
\end{align}
is a differential graded Lie algebra. So we have the following properties: for $a\in A^{0,p}(M,\wedge^q T_M), b\in A^{0,p'}(M,\wedge^{q'} T_M)$, and $c\in A^{0,p''}(M,\wedge^{q''} T_M)$
\begin{enumerate}
\item $[a,b]=-(-1)^{(p+q+1)(p'+q'+1)}[b,a]$
\item $[a,[b,c]]=[[a,b],c]+(-1)^{(p+q+1)(p'+q'+1)}[b,[a,c]]$
\item $\bar{\partial}[a,b]=[\bar{\partial} a,b]+(-1)^{p+q+1}[a,\bar{\partial} b]$
\end{enumerate}

\begin{theorem}\label{m}
If we take a sufficiently small polydisk $\Delta$ as in subsection $\ref{prill}$, then for $t\in \Delta$, a type $(2,0)$-part $\Lambda^{2,0}$ of a $C^{\infty}$ bivector field  $\Lambda=\sum_{r,s=1}^n h_{rs}(z)\frac{\partial}{\partial z_{r}}\wedge \frac{\partial}{\partial z_{s}}$ on $M$ is holomorphic with respect to the complex structure $M_t$ induced by $\varphi(t)$ if and only if it satisfies the equation
\begin{align*}
\bar{\partial}\Lambda-[\Lambda,\varphi(t)]=0
\end{align*}
Moreover, if $[\Lambda,\Lambda]=0$, then $[\Lambda^{2,0},\Lambda^{2,0}]=0$.

\end{theorem}

\begin{proof}
We note that the type $(2,0)$-part of $\Lambda=\sum_{r,s=1}^n h_{rs}(z)\frac{\partial}{\partial z_{r}}\wedge \frac{\partial}{\partial z_{s}}$ with respect to complex structure $M_t$ is $\sum_{r,s,\alpha,\beta=1}^n h_{rs}\frac{\partial \xi_j^\alpha}{\partial z_{r}}\frac{\partial \xi_j^\beta}{\partial z_{s}}\frac{\partial}{\partial \xi_j^\alpha}\wedge \frac{\partial}{\partial \xi_j^\beta}$. Hence by Theorem \ref{text}, it suffices to show that 
\begin{align}\label{mm9}
\text{For each $\alpha,\beta$, $(\bar{\partial}-\varphi(t))(\sum_{r,s=1}^n h_{rs}\frac{\partial \xi_j^\alpha}{\partial z_{r}}\frac{\partial \xi_j^\beta}{\partial z_{s}})=0$ if and only if $\bar{\partial}\Lambda-[\Lambda,\varphi(t)]=0$}
\end{align}
 First we note that from $(\ref{b})$ and $(\ref{tt00})$, we have
 \begin{align}\label{tt01}
 &\bar{\partial}\Lambda-[\Lambda,\varphi(t)]\\
 &=\sum_{r,s, v=1}^n \frac{\partial h_{rs}}{\partial \bar{z}_v}d\bar{z}_v\frac{\partial}{\partial z_{r}}\wedge \frac{\partial}{\partial z_{s}}-\sum_{r,s, v,\lambda=1}^n [h_{rs}\frac{\partial}{\partial z_{r}}\wedge \frac{\partial}{\partial z_{s}},\varphi_v^{\lambda} d\bar{z}_v \frac{\partial}{\partial z_{\lambda}}]\notag\\
 &=\sum_{r,s, v=1}^n \frac{\partial h_{rs}}{\partial \bar{z}_v}d\bar{z}_v\frac{\partial}{\partial z_{r}}\wedge \frac{\partial}{\partial z_{s}}+\sum_{r,s, v,\lambda=1}^n [h_{rs}\frac{\partial}{\partial z_{r}}\wedge \frac{\partial}{\partial z_{s}},\varphi_v^{\lambda} \frac{\partial}{\partial z_{\lambda}}]d\bar{z}_v \notag\\
 &=\sum_{r,s, v=1}^n \frac{\partial h_{rs}}{\partial \bar{z}_v}d\bar{z}_v\frac{\partial}{\partial z_{r}}\wedge \frac{\partial}{\partial z_{s}}+\sum_{r,s, v,\lambda=1}^n(h_{rs}\frac{\partial \phi_v^{\lambda}}{\partial z_{r}}\frac{\partial}{\partial z_{\lambda}}\wedge \frac{\partial}{\partial z_{s}}-\varphi_{v}^{\lambda} \frac{\partial h_{rs}}{\partial z_{\lambda}}\frac{\partial}{\partial z_{r}}\wedge \frac{\partial}{\partial z_{s}}+h_{rs} \frac{\partial \varphi_v^{\lambda}}{\partial z_{s}}\frac{\partial}{\partial z_{r}}\wedge \frac{\partial}{\partial z_{\lambda}})d\bar{z}_v.\notag
\end{align} 
  By considering the coefficients of $d\bar{z}_v\frac{\partial}{\partial z_{r}}\wedge \frac{\partial}{\partial z_{s}}$ in $(\ref{tt01})$, $\bar{\partial}\Lambda-[\Lambda,\varphi(t)]=0$ is equivalent to 
\begin{align}\label{mm3}
\frac{\partial h_{rs}}{\partial \bar{z}_v}+\sum_{c=1}^n (h_{cs}\frac{\partial \varphi^{r}_v}{\partial z_c}-\varphi^c_v\frac{\partial h_{rs}}{\partial z_c}+h_{r c}\frac{\partial \varphi^{s}_v}{\partial z_c})]=0\,\,\,\,\,\text{for each $r,s,v$.}
\end{align}
On the other hand, from $(\ref{b})$, $(\bar{\partial}-\varphi(t))(\sum_{r,s=1}^n h_{rs}\frac{\partial \xi_j^\alpha}{\partial z_{r}}\frac{\partial \xi_j^\beta}{\partial z_{s}})=0$ for each $\alpha,\beta$ is equivalent to
\begin{align}\label{mm5}
\sum_{r,s=1}^n (\frac{\partial h_{rs}}{\partial \bar{z}_v}\frac{\partial \xi_j^\alpha}{\partial z_{r}}\frac{\partial \xi_j^\beta}{\partial z_{s}}+&h_{rs}\frac{\partial}{\partial z_{r}}\left(\frac{\partial \xi_j^\alpha}{\partial \bar{z}_v}\right)\frac{\partial \xi_j^\beta}{\partial z_{s}}+h_{rs}\frac{\partial \xi_j^\alpha}{\partial z_{r}}\frac{\partial}{\partial z_{s}}\left(\frac{\partial \xi_j^\beta}{\partial \bar{z}_v}\right))\\
&-\sum_{r,s,c=1}^n \varphi^{c}_v(\frac{\partial h_{rs}}{\partial z_{c}}\frac{\partial \xi_j^\alpha}{\partial z_{r}}\frac{\partial \xi_j^\beta}{\partial z_{s}}+h_{rs}\frac{\partial^2 \xi_j^\alpha}{\partial z_{r} \partial z_{c}}\frac{\partial \xi_j^\beta}{\partial z_{s}}+h_{rs}\frac{\partial \xi_j^\alpha}{\partial z_{r}}\frac{\partial^2 \xi_j^\beta}{\partial z_{s}\partial z_{c}})=0\,\,\,\,\,\,\,\text{for each $\alpha,\beta,v$}\notag
\end{align}
From $(\ref{matrix1})$, we have $\frac{\partial \xi_j^\alpha}{\partial \bar{z}_v}=\sum_{c=1}^n \frac{\partial \xi_j^\alpha}{\partial z_c}\varphi^c_v$ and $\frac{\partial \xi_j^\beta}{\partial \bar{z}_v}=\sum_{c=1}^n \frac{\partial \xi_j^\beta}{\partial z_c}\varphi^c_v$.
So $(\ref{mm5})$ is equivalent to
\begin{align}\label{mm8}
&\sum_{r,s=1}^n \frac{\partial h_{rs}}{\partial \bar{z}_v}\frac{\partial \xi_j^\alpha}{\partial z_{r}}\frac{\partial \xi_j^\beta}{\partial z_{s}}+\sum_{r,s,c=1}^n(h_{rs}\left(\frac{\partial^2 \xi_j^\alpha}{\partial z_{r} \partial z_c}\varphi^c_v+\frac{\partial \xi_j^\alpha}{\partial z_c}\frac{\partial \varphi^c_v}{\partial z_{r}}\right)\frac{\partial \xi_j^\beta}{\partial z_{s}}+h_{rs}\frac{\partial \xi_j^\alpha}{\partial z_{r}}\left(\frac{\partial^2 \xi_j^\beta}{\partial z_{s} \partial z_c}\varphi^c_v+\frac{\partial \xi_j^\beta}{\partial z_c}\frac{\partial \varphi^c_v}{\partial z_{s}}\right))\\
&-\sum_{r,s,c=1}^n \varphi^{c}_v(\frac{\partial h_{rs}}{\partial z_{c}}\frac{\partial \xi_j^\alpha}{\partial z_{r}}\frac{\partial \xi_j^\beta}{\partial z_{s}}+h_{rs}\frac{\partial^2 \xi_j^\alpha}{\partial z_{r} \partial z_{c}}\frac{\partial \xi_j^\beta}{\partial z_{s}}+h_{rs}\frac{\partial \xi_j^\alpha}{\partial z_{r}}\frac{\partial^2 \xi_j^\beta}{\partial z_{s}\partial z_{c}})\notag\\
&=\sum_{r,s=1}^n \frac{\partial h_{rs}}{\partial \bar{z}_v}\frac{\partial \xi_j^\alpha}{\partial z_{r}}\frac{\partial \xi_j^\beta}{\partial z_{s}}+\sum_{r,s,c=1}^n(h_{rs}\frac{\partial \xi_j^\alpha}{\partial z_c}\frac{\partial \varphi^c_v}{\partial z_{r}}\frac{\partial \xi_j^\beta}{\partial z_{s}}+h_{rs}\frac{\partial \xi_j^\alpha}{\partial z_{r}}\frac{\partial \xi_j^\beta}{\partial z_c}\frac{\partial \varphi^c_v}{\partial z_{s}})-\sum_{r,s,c=1}^n \varphi^{c}_v\frac{\partial h_{rs}}{\partial z_{c}}\frac{\partial \xi_j^\alpha}{\partial z_{r}}\frac{\partial \xi_j^\beta}{\partial z_{s}}
=0 \notag
\end{align}
So $(\ref{mm8})$ is equivalent to
\begin{align}\label{tt62}
\sum_{r,s=1}^n [\frac{\partial h_{rs}}{\partial \bar{z}_v}+\sum_{c=1}^n (h_{cs}\frac{\partial \varphi^{r}_v}{\partial z_c}-\varphi^c_v\frac{\partial h_{rs}}{\partial z_c}+h_{r c}\frac{\partial \varphi^{s}_v}{\partial z_c})]\frac{\partial \xi_j^\alpha}{\partial z_{r}}\frac{\partial \xi_j^\beta}{\partial z_{s}}=0 \,\,\,\,\,\text{for each $\alpha,\beta,v$.}
\end{align}
From $(\ref{det})$, the equation $(\ref{tt62})$ is equivalent to
\begin{align}\label{mm4}
\frac{\partial h_{rs}}{\partial \bar{z}_v}+\sum_{c=1}^n (h_{cs}\frac{\partial \varphi^{r}_v}{\partial z_c}-\varphi^c_v\frac{\partial h_{rs}}{\partial z_c}+h_{r c}\frac{\partial \varphi^{s}_v}{\partial z_c})=0\,\,\,\,\,\text{for each $r,s,v$}.
\end{align}
Note that $(\ref{mm3})$ is same to $(\ref{mm4})$, which proves $(\ref{mm9})$.

For the second statement of Theorem $\ref{m}$, we note that 
\begin{align*}
\Lambda&=\sum_{r,s=1}^n h_{rs}\frac{\partial}{\partial z_r}\wedge \frac{\partial}{\partial z_s}=\sum_{r,s,i,j=1}^n( h_{rs}\frac{\partial \xi_j^\alpha}{\partial z_r}\frac{\partial \xi_j^\beta}{\partial z_s}\frac{\partial}{\partial \xi_j^\alpha}\wedge \frac{\partial}{\partial \xi_j^\beta}+ 2h_{rs}\frac{\partial \xi_j^\alpha}{\partial z_r}\frac{\partial \bar{\xi}_j^\beta}{\partial z_s}\frac{\partial}{\partial \xi_j^\alpha}\wedge \frac{\partial}{\partial \bar{\xi}_j^\beta}+ h_{rs}\frac{\partial \xi_j^\alpha}{\partial z_r}\frac{\partial \bar{\xi}_j^\beta}{\partial z_s}\frac{\partial}{\partial \bar{\xi}_j^\alpha}\wedge \frac{\partial}{\partial \bar{\xi}_j^\beta})\\
&=\Lambda^{2,0}+\Lambda^{1,1}+\Lambda^{2,0}.
\end{align*}
 Since $[\Lambda,\Lambda]=0$, the type $(3,0)$ part $[\Lambda,\Lambda]^{3,0}=[\Lambda^{2,0},\Lambda^{2,0}]+[\Lambda^{2,0},\Lambda^{1,1}]^{3,0}=0$ (see Remark \ref{tt45}). Since $\Lambda^{2,0}$ is holomorphic with respect to the complex structure induced by $\varphi(t)$, we have $[\Lambda^{2,0},\Lambda^{1,1}]^{3,0}=0$. Hence $[\Lambda^{2,0},\Lambda^{2,0}]=0$.
\end{proof}

\begin{remark}
A $C^{\infty}$ complex bivector field $\Lambda\in A^{0,0}(M,\wedge^2 T_M)$ on $M$ with $[\Lambda,\Lambda]=0$ gives a Poisson bracket $\{-,-\}$ on $C^{\infty}$ complex valued functions on $M$. We point out that when we restrict the Poisson bracket $\{-,-\}$ to holomorphic functions with respect to the complex structure $M_t$ induced by $\varphi(t)$, this is exactly the $($holomorphic$)$ Poisson bracket induced from $\Lambda^{2,0}$ when $\bar{\partial}\Lambda-[\Lambda,\varphi(t)]=0$. 
\end{remark}

\begin{remark}
By Theorem $\ref{m}$, $\varphi(t)$ in $(\ref{b})$  and $\Lambda(t)$ in $(\ref{f})$ satisfy
\begin{align}\label{ll00}
\bar{\partial}\Lambda(t)-[\Lambda(t),\varphi(t)]=0 \,\,\,\,\,\,\,\, \text{for each $t$}
\end{align}
and $\Lambda(t)^{2,0}=\Lambda_t$ for each $t$ $($see Remark $\ref{renj}$$)$.
\end{remark}

\subsection{Expression of  infinitesimal deformations in terms of $\varphi(t)$ and $\Lambda(t)$}\

In this subsection, we study how an infinitesimal deformation of $(M,\Lambda_0)=\omega^{-1}(0)$ in the Poisson analytic family $(\mathcal{M},\Lambda,B,\omega)$ (in subsection \ref{prill}) is represented in terms of $\varphi(t)$ (\ref{b}) and $\Lambda(t)$ (\ref{f}). Recall that an infinitesimal deformation at $(M,\Lambda_0)$ is captured by an element $\left(\frac{\partial (M_t,\Lambda_t)}{\partial t}\right)_{t=0}\in \mathbb{H}^1(M,\Theta_M^\bullet)$ of the complex of sheaves (\ref{complex}) by using the following \u{C}ech resolution associated with the open covering $\mathcal{U}^0=\{U_j^0:=U_j\times 0\}$ (see Proposition \ref{gg} and Definition \ref{mapping}).

\begin{center}
$\begin{CD}
@A[\Lambda_0,-]AA\\
C^0(\mathcal{U}^0,\wedge^3 \Theta_M)@>-\delta>>\cdots\\
@A[\Lambda_0,-]AA @A[\Lambda_0,-]AA\\
C^0(\mathcal{U}^0,\wedge^2 \Theta_M)@>\delta>> C^1(\mathcal{U}^0,\wedge^2 \Theta_M)@>-\delta>>\cdots\\
@A[\Lambda_0,-]AA @A[\Lambda_0,-]AA @A[\Lambda_0,-]AA\\
C^0(\mathcal{U}^0,\Theta_M)@>-\delta>>C^1(\mathcal{U}^0,\Theta_M)@>\delta>>C^2(\mathcal{U}^0,\Theta_M)@>-\delta>>\cdots\\
\end{CD}$
\end{center}

We can also compute the hypercohomology group of the complex of sheaves $(\ref{complex})$ by using the following Dolbeault resolution.
\begin{center}
$\begin{CD}
@A[\Lambda_0,-]AA\\
A^{0,0}(M,\wedge^3 T_M)@>\bar{\partial}>>\cdots\\
@A[\Lambda_0,-]AA @A[\Lambda_0,-]AA\\
A^{0,0}(M,\wedge^2 T_M)@>\bar{\partial}>> A^{0,1}(M,\wedge^2 T_M)@>\bar{\partial}>>\cdots\\
@A[\Lambda_0,-]AA @A[\Lambda_0,-]AA @A[\Lambda_0,-]AA\\
A^{0,0}(M,T_M)@>\bar{\partial}>>A^{0,1}(M,T_M)@>\bar{\partial}>>A^{0,2}(M, T_M)@>\bar{\partial}>>\cdots\\
\end{CD}$
\end{center}
We describe how a $1$-cocycle in the \u{C}ech resolution look like in the Dolbeault resolution.
In the picture below, we connect two resolutions. We only depict a part of resolutions that we need in the following diagram. Recall that $\mathscr{A}^{0,p}(\wedge^q T_M)$ is the sheaf of germs of $C^{\infty}$-section of $\wedge^p \bar{T}_M^*\otimes \wedge^q T_M$ (see (\ref{tt76})).

{\tiny
\[
\xymatrixrowsep{0.2in}
\xymatrixcolsep{0.1in}
\xymatrix{
& & H^0(M,\wedge^3 \Theta_M) \ar[ld]  \ar[rr]  &  & C^0(\wedge^3 \Theta_M) \ar[ld]\\
&A^{0,0}(M,\wedge^3 T_M)  \ar[rr] & & C^0(\mathscr{A}^{0,0}(\wedge^3 T_M))\\
&& H^0(M,\wedge^2 \Theta_M) \ar@{.>}[uu] \ar@{.>}[ld] \ar@{.>}[rr] & & C^0(\wedge^2 \Theta_M) \ar[uu] \ar[ld] \ar[rr]^{\delta} && C^1(\wedge^2 \Theta_X) \ar[ld]\\
& A^{0,0}(M,\wedge^2  T_M)  \ar[uu] \ar[ld] \ar[rr] && C^0(\mathscr{A}^{0,0}(\wedge^2 T_M)) \ar[uu] \ar[ld] \ar[rr]^{\delta} && C^1(\mathscr{A}^{0,0}(T_M))   \\
A^{0,1}(M,\wedge^2 T_M) \ar[rr]  && C^0(A^{0,1}(\wedge^2 T_M)) && C^0(\Theta_M) \ar@{.>}[ld]\ar@{.>}[uu] \ar@{.>}[rr]^{-\delta} & & C^1(\Theta_X) \ar[uu] \ar[ld]\\
& A^{0,0}(M,T_M) \ar@{.>}[uu] \ar@{.>}[rr]  \ar@{.>}[ld] && C^0(\mathscr{A}^{0,0}(T_M)) \ar[uu]^{[\Lambda_0,-]} \ar[rr]^{-\delta}  \ar[ld]^{\bar{\partial}} && C^1(\mathscr{A}^{0,0}(T_M)) \ar[ld]^{\bar{\partial}} \ar[uu]\\
A^{0,1}(M,T_M) \ar[uu] \ar[rr] & & C^0(\mathscr{A}^{0,1}(T_M)) \ar[uu] \ar[rr]^{-\delta} & & C^1(\mathscr{A}^{0,1}(T_M))
}\]}
Note that each horizontal complex is exact except for edges of the ``real wall".

Now we explicitly construct the isomorphism of the first hypercohomology group from \u{C}ech resolution to the first hypercohomology group from Dolbeault resolution, namely
\begin{align}\label{isomorphism}
&\frac{ker( \mathcal{C}^0(\mathcal{U}^0, \wedge^2 \Theta_M)\oplus \mathcal{C}^1(\mathcal{U}^0,\Theta_M)\to  \mathcal{C}^0(\mathcal{U}^0, \wedge^3 \Theta_M)\oplus \mathcal{C}^1(\mathcal{U}^0,\wedge^2\Theta_M)\oplus  \mathcal{C}^2(\mathcal{U}^0, \Theta_M))}{  im(\mathcal{C}^0(\mathcal{U}^0, \Theta_M)\to\mathcal{C}^0(\mathcal{U}^0, \wedge^2 \Theta_M)\oplus \mathcal{C}^1(\mathcal{U}^0,\Theta_M))}\\
&\cong \frac{ker(A^{0,0}(M,\wedge^2 T_M)\oplus A^{0,1}(M, T_M)\to A^{0,0}(M,\wedge^3 T_M)\oplus A^{0,1}(M,\wedge^2 T_M)\oplus A^{0,2}(M, T_M))}{im(A^{0,0}(M,T_M)\to A^{0,0}(M,\wedge^2 T_M)\oplus A^{0,1}(M,T_M))}\notag\\
(b,a) &\mapsto   \notag ([\Lambda_0,c]-b,\bar{\partial} c)
\end{align}

We define the map in the following way:
let $(b,a) \in \mathcal{C}^0(\mathcal{U}^0, \wedge^2 \Theta_M)\oplus \mathcal{C}^1(\mathcal{U}^0,\Theta_M)$ be a cohomology class of \u{C}ech resolution. Since $\delta a=0$, there exists a $c\in C^0(\mathcal{U}^0,\mathscr{A}^{0,0}(T_M))$ such that $-\delta c=a$. Since $a$ is holomorphic $(\bar{\partial}a=0)$, by the commutativity $\bar{\partial} c\in A^{0,1}(M, T_M)$. We claim that $[\Lambda_0,c]-b\in A^{0,0}(M,\wedge^2 T_M)$. Indeed, $\delta([\Lambda_0,c]-b)=-[\Lambda_0,-\delta c]-\delta b=-[\Lambda_0,a]-\delta b=0$. We show that $(\bar{\partial} c, [\Lambda_0,c]-b)$ is a cohomology class from Dolbeault resolution. Indeed, $\bar{\partial}(\bar{\partial}c)=0.$ $[\Lambda_0,[\Lambda_0,c]-b]=0$. $\bar{\partial} ([\Lambda_0,c]-b)+[\Lambda_0, \bar{\partial} c]=-[\Lambda_0,\bar{\partial} c]+[\Lambda_0,\bar{\partial c}]=0$. We define the map by $(b,a)\mapsto ([\Lambda_0,c]-b,\bar{\partial} c)$. We show that this map is well defined. Indeed, let $(b',a')$ define the same class given by $(b,a)$. Then there exists $d\in C^0(\mathcal{U}^0,\Theta_M)$ such that $a-a'=-\delta d$ and $b-b'=[\Lambda_0,d]$. Let $-\delta c'=a'$. Then $\bar{\partial} c-\bar{\partial}c'=\bar{\partial}(c-c'-d)$, and $[\Lambda_0,c]-b-([\Lambda_0,c']-b')=[\Lambda_0,c-c']-(b-b')=[\Lambda_0,c-c'-d]$.

For the inverse map, let $(\beta,\alpha) \in A^{0,0}(M,\wedge^2  T_M)\oplus A^{0,1}(M,  T_M) $ be a cohomology class from Dolbeault resolution. 
Then there exists a $c\in C^0(\mathcal{U}^0,\mathscr{A}^{0,0}(T_M))$ such that $\bar{\partial} c =\alpha$. We define the inverse map $(\beta,\alpha) \mapsto ([\Lambda_0,c]-\beta,-\delta c)$.

\begin{theorem}\label{n}
$\left(- \left(\frac{\partial \Lambda(t)}{\partial t}\right)_{t=0},\left(\frac{\partial \varphi(t)}{\partial t}\right)_{t=0}\right)\in A^{0,0}(M,\wedge^2 T_M)\oplus A^{0,1}(M,T_M)$ satisfies 

$[\Lambda_0, -\left(\frac{\partial \Lambda(t)}{\partial t}\right)_{t=0}]=0$,  $\bar{\partial} \left( -(\frac{\partial \Lambda(t)}{\partial t})_{t=0}\right) +[\Lambda_0,\left(\frac{\partial \varphi(t)}{\partial t}\right)_{t=0}]=0$, $\bar{\partial} \left(\frac{\partial \varphi(t)}{\partial t}\right)_{t=0}=0$, and under the isomorphism $(\ref{isomorphism})$, $\left(\frac{\partial(M_t,\Lambda_t)}{\partial t}\right)_{t=0}\in \mathbb{H}^1(M, \Theta_M^\bullet)$ corresponds to $\left( -\left(\frac{\partial \Lambda(t)}{\partial t}\right)_{t=0},\left(\frac{\partial \varphi(t)}{\partial t}\right)_{t=0}\right)$

\end{theorem}

\begin{proof}
By Theorem \ref{1thm} (3), (\ref{ll00}) and (\ref{tt07}), we have $[\Lambda(t),\Lambda(t)]=0$, $\bar{\partial} \Lambda(t)-[\Lambda(t),\varphi(t)]=0$ and $\bar{\partial}\varphi(t)-\frac{1}{2}[\varphi(t),\varphi(t)]=0$ with $\varphi(0)=0, \Lambda(0)=\Lambda_0$. By taking the derivative of these equations with respect to $t$ and plugging $0$, we get the first claim. Next we show the second claim. Put
\begin{align*}
\theta_{jk}=\sum_{\alpha=1}^n \left(\frac{\partial f_{jk}^{\alpha}(\xi_k,t)}{\partial t}\right)_{t=0}\frac{\partial}{\partial \xi_j^{\alpha}},\,\,\,\,\,\,\,\,
\sigma_j=\sum_{r,s=1}^n \left(\frac{\partial g_{rs}(\xi,t)}{\partial t}\right)_{t=0} \frac{\partial}{\partial \xi_j^r}\wedge \frac{\partial}{\partial \xi_j^s}
\end{align*}
The infinitesimal deformation $\left(\frac{\partial (M_t,\Lambda_t)}{\partial t}\right)_{t=0}\in \mathbb{H}^1(M,\Theta_M^\bullet)$ is the cohomology class of the $(\{\sigma_j\},\{\theta_{jk}\},)\in C^0(\mathcal{U}^0,\wedge^2 \Theta_M)\oplus C^1(\mathcal{U}^0,\Theta_M)$ (see Proposition \ref{gg} and Definition \ref{mapping}). We fix a tangent vector $\frac{\partial}{\partial t}\in T_0(\Delta)$, denote $\left(\frac{\partial f(t)}{\partial t}\right)_{t=0}$ by $\dot{f}$ for a $C^{\infty}$ function $f(t),t\in \Delta$. With this notation, we put
\begin{align*}
\xi_j=\sum_{\alpha=1}^n \dot{\xi_j}^{\alpha}\frac{\partial}{\partial \xi_j^{\alpha}},\,\,\,\,\,\text{where}\,\,\,\,\dot{\xi_j}^{\alpha}=\left(\frac{\partial \xi_j^{\alpha}(z,t)}{\partial t}\right)_{t=0}
\end{align*}
for each $j$. Then we have 
\begin{align}\label{hh89}
\text{$\delta \{ \xi_j \}=-\{ \theta_{jk} \}$ and $\bar{\partial} \xi_j =\sum_{\lambda=1}^n \left( \frac{\partial \varphi^{\lambda}(z,t)}{\partial t}\right)_{t=0}\frac{\partial}{\partial z_{\lambda}}=\sum_{\lambda=1}^n \dot{\varphi}^{\lambda}\frac{\partial}{\partial z_{\lambda}}=\dot{\varphi}$} 
\end{align}
(for the detail, see \cite{Kod05} Theorem 5.4 p.266). On the other hand, 
\begin{lemma}\label{beta}
We have $\dot{\Lambda}-{\sigma_j}+[\Lambda_0, \xi_j]=0$. More precisely,

\begin{align*}
\sum_{r,s=1}^n \left(\frac{\partial h_{rs} (z,t)}{\partial t}\right)_{t=0} \frac{\partial}{\partial z_r}\wedge \frac{\partial}{\partial z_s}-\sum_{\alpha,\beta=1}^n \left(\frac{\partial g_{\alpha\beta}^j(\xi_j,t)}{\partial t}\right)_{t=0} \frac{\partial}{\partial \xi_j^{\alpha}} \wedge \frac{\partial}{\partial \xi_j^{\beta}}+[\sum_{r,s=1}^n g_{rs}^j(\xi_j,0)\frac{\partial}{\partial \xi_j^r}\wedge \frac{\partial}{\partial \xi_j^s},\sum_{c=1}^n \dot{\xi}_j^c\frac{\partial}{\partial \xi_j^c}]=0
\end{align*}
equivalently $($with the notation above$)$, 
\begin{align}\label{tt10}
\sum_{r,s=1}^n \dot{h_{rs}} \frac{\partial}{\partial z_r}\wedge \frac{\partial}{\partial z_s}-\sum_{\alpha,\beta=1}^n \dot{g}_{\alpha\beta}^j \frac{\partial}{\partial \xi_j^{\alpha}} \wedge \frac{\partial}{\partial \xi_j^{\beta}}+[\sum_{r,s=1}^n g_{rs}^j(\xi_j,0)\frac{\partial}{\partial \xi_j^r}\wedge \frac{\partial}{\partial \xi_j^s},\sum_{c=1}^n \dot{\xi}_j^c\frac{\partial}{\partial \xi_j^c}]=0
\end{align}
\end{lemma}

\begin{proof}
From $(\ref{holomorphic})$, the first term of $(\ref{tt10})$ is
\begin{align*}
\sum_{r,s=1}^n \dot{h_{rs}}\frac{\partial}{\partial z_r}\wedge \frac{\partial}{\partial z_s}=\sum_{r,s,a,b=1}^n \dot{h_{rs}} \frac{\partial \xi_j^a (z,0)}{\partial z_r}\frac{\partial \xi_j^b(z,0)}{\partial z_s}\frac{\partial}{\partial \xi_j^a}\wedge \frac{\partial}{\partial \xi_j^b}
\end{align*}
Let's compute the third term of $(\ref{tt10})$:
\begin{align*}
&\sum_{r,s,c=1}^n [g_{rs}^j(\xi_j,0) \frac{\partial}{\partial \xi_j^r}\wedge \frac{\partial}{\partial \xi_j^s}, \dot{\xi}_j^c\frac{\partial}{\partial \xi_j^c}]=\sum_{r,s,c=1}^n ([g_{rs}^j(\xi_j,0)\frac{\partial}{\partial \xi_j^r},\dot{\xi}_j^c \frac{\partial}{\partial \xi_j^c}]\wedge \frac{\partial}{\partial \xi_j^s}-g_{rs}^j(\xi_j,0)[\frac{\partial}{\partial \xi_j^s},\dot{\xi}_j^c \frac{\partial}{\partial \xi_j^c}]\wedge\frac{\partial}{\partial \xi_j^r})\\
&=\sum_{r,s,c=1}^n (g_{rs}^j(\xi_j,0)\frac{\partial \dot{\xi}_j^c}{\partial \xi_j^r}\frac{\partial}{\partial \xi_j^c}\wedge \frac{\partial}{\partial \xi_j^s}-\dot{\xi}_j^c\frac{\partial g_{rs}(\xi_j,0)}{\partial \xi_j^c}\frac{\partial}{\partial \xi_j^r}\wedge \frac{\partial}{\partial \xi_j^s}+g_{rs}^j(\xi_j,0)\frac{\partial \dot{\xi}_j^c}{\partial \xi_j^s}\frac{\partial}{\partial \xi_j^r}\wedge\frac{\partial}{\partial \xi_j^c})
\end{align*}
By considering the coefficients of $\frac{\partial}{\partial \xi_j^a}\wedge \frac{\partial}{\partial \xi_j^b}$, $(\ref{tt10})$ is equivalent to
\begin{align}\label{tt11}
\sum_{r,s=1}^n \dot{h_{rs}} \frac{\partial \xi_j^a (z,0)}{\partial z_r}\frac{\partial \xi_j^b(z,0)}{\partial z_s}-\dot{g}_{ab}^j-\sum_{c=1}^n \dot{\xi}_j^c\frac{\partial g_{ab}(\xi_j,0)}{\partial \xi_j^c}+\sum_{c=1}^n (g_{cb}^j(\xi_j,0)\frac{\partial \dot{\xi}_j^a}{\partial \xi_j^c}+g_{ac}^j(\xi_j,0)\frac{\partial \dot{\xi}_j^b}{\partial \xi_j^c})=0
\end{align}
On the other hand, from $(\ref{tt304})$, we have
\begin{align}\label{tt14}
g_{ab}^j(\xi_j^1(z,t),...,\xi_j^n(z,t),t_1,...,t_m)=\sum_{r,s=1}^n h_{rs}(z,t)\frac{\partial \xi_j^a(z,t)}{\partial z_r}\frac{\partial \xi_j^b(z,t)}{\partial z_s}
\end{align}
By taking the derivative of $(\ref{tt14})$ with respect to $t$ and putting $t=0$, we have
\begin{align*}
\sum_{c=1}^n \frac{\partial g_{ab}^j(\xi_j,0)}{\partial \xi_j^c}\dot{\xi}_j^c+\dot{g}_{ab}^j=\sum_{r,s=1}^n\dot{h}_{rs}\frac{\partial \xi_j^a(z,0)}{\partial z_r}\frac{\partial \xi_j^b(z,0)}{\partial z_s}+\sum_{r,s=1}^n h_{rs}(z,0)( \frac{\partial \dot{\xi}_j^a}{\partial z_r}\frac{\partial \xi_j^b(z,0)}{\partial z_s}+\frac{\partial \xi_j^a(z,0)}{\partial z_r}\frac{\partial \dot{\xi}_j^b}{\partial z_s})
\end{align*}
Hence $(\ref{tt11})$ is equivalent to
\begin{align}\label{vv8}
\sum_{c=1}^n g_{cb}^j(\xi_j,0)\frac{\partial \dot{\xi}_j^a}{\partial \xi_j^c}+g_{ac}^j(\xi_j,0)\frac{\partial \dot{\xi}_j^b}{\partial \xi_j^c}=\sum_{r,s=1}^n (h_{rs}(z,0)\frac{\partial \dot{\xi}_j^a}{\partial z_r}\frac{\partial \xi_j^b(z,0)}{\partial z_s}+ h_{rs}(z,0)\frac{\partial \xi_j^a(z,0)}{\partial z_r}\frac{\partial \dot{\xi}_j^b}{\partial z_s})
\end{align}
Indeed, the left hand side and right hand side of $(\ref{vv8})$ coincide: from $(\ref{tt14})$ and $(\ref{holomorphic})$,
\begin{align*}
\sum_{c=1}^n g_{cb}^j(\xi_j,0)\frac{\partial \dot{\xi}_j^a}{\partial \xi_j^c}+g_{ac}^j(\xi_j,0)\frac{\partial \dot{\xi}_j^b}{\partial \xi_j^c}&=\sum_{r,s,c=1}^n (h_{rs}(z,0)\frac{\partial \xi_j^c(z,0)}{\partial z_r}\frac{\partial \xi_j^b(z,0)}{\partial z_s}\frac{\partial \dot{\xi}_j^a}{\partial \xi_j^c}+ h_{rs}(z,0)\frac{\partial \xi_j^a(z,0)}{\partial z_r}\frac{\partial \xi_j^c(z,0)}{\partial z_s}\frac{\partial \dot{\xi}_j^b}{\partial \xi_j^c})\\
&=\sum_{r,s=1}^n (h_{rs}(z,0)\frac{\partial \dot{\xi}_j^a}{\partial z_r}\frac{\partial \xi_j^b(z,0)}{\partial z_s}+ h_{rs}(z,0)\frac{\partial \xi_j^a(z,0)}{\partial z_r}\frac{\partial \dot{\xi}_j^b}{\partial z_s})
\end{align*}
This completes Lemma \ref{beta}.
\end{proof}
Going back to the proof of Theorem \ref{n}, we defined the isomorphism $(\ref{isomorphism})$: $(b,a)\mapsto ([\Lambda_0,c]-b,\bar{\partial} c)$ where $-\delta c=a$. We take $(b,a)=(\{\sigma_j\},\{\theta_{jk}\})$ and $c=\{\xi_j\}$. Note $-\delta \{\xi_j\}=\{\theta_{jk}\}$ by $(\ref{hh89})$. Then by the isomorphism $(\ref{isomorphism})$, $(\{\sigma_j\},\{\theta_{jk}\})$ is mapped to $([\Lambda_0,\{\xi_j\}]-\{\sigma_j\}, \bar{\partial}\{\xi_j\})$ which is $(-\dot{\Lambda}, \dot{\varphi})$ by Lemma \ref{beta} and (\ref{hh89}). This completes the proof of Theorem \ref{n}.
\end{proof}

\subsection{Integrability condition}\label{ss00}\

We have showed that given a Poisson analytic family $(\mathcal{M},\Lambda,B,\omega)$, the deformations $(M_t,\Lambda_t)$ of $M=M_0$ near $(M_0,\Lambda_0)$ is represented by the $C^{\infty}$ vector $(0,1)$-form $\varphi(t)$ (\ref{b}) and the $C^{\infty}$ bivector field $\Lambda(t)$ of type $(2,0)$ (\ref{f}) on $M$ with $\varphi(0)=0$ and $\Lambda(0)=\Lambda_0$ satisfying the conditions: (1)$[\Lambda(t),\Lambda(t)]=0,(2)\bar{\partial} \Lambda(t)-[\Lambda(t),\varphi(t)]=0$ and (3)$\bar{\partial} \varphi(t)-\frac{1}{2}[\varphi(t),\varphi(t)]=0$ for each $t\in \Delta$ by Theorem \ref{1thm} (3), (\ref{ll00}) and (\ref{tt07}).

Conversely, we will show that on a compact holomorphic Poisson manifold $(M,\Lambda_0)$, a $C^{\infty}$ vector $(0,1)$-form $\varphi\in A^{0,1}(M,T_M)$ and a $C^{\infty}$ type $(2,0)$ bivector field $\Lambda\in A^{0,0}(M,\wedge^2 T_M)$ on $M$ such that $\varphi$ and $\Lambda_0+\Lambda$ satisfying the integrability condition (1),(2),(3) define another holomorphic Poisson structure on the underlying differentiable manifold $M$. Indeed, let $\varphi=\sum_{\lambda=1}^n \varphi^{\lambda}_{v}(z)d\bar{z}_v\frac{\partial}{\partial z_{\lambda}}$ be a $C^{\infty}$ vector $(0,1)$-form  and $\Lambda=\sum_{r,s=1}^n h_{rs}(z)\frac{\partial}{\partial z_r}\wedge \frac{\partial}{\partial z_s}$ be a $C^{\infty}$ bivector field of type $(2,0)$ on a compact holomorphic Poisson manifold $(M,\Lambda_0)$. Suppose $\det(\delta_v^{\lambda}-\sum_{\mu=1}^n \varphi_{v}^{\mu}(z)\overline{\varphi_{\mu}^{\lambda}}(z))_{\lambda,\mu=1,...,n}\ne 0$, and $\varphi$, $\Lambda$ satisfy the integrability condition:
\begin{align}
&[\Lambda_0+\Lambda,\Lambda_0+\Lambda]=0 \label{bb1}\\
&\bar{\partial} (\Lambda_0+\Lambda)-[\Lambda_0+\Lambda,\varphi]=0\label{bb2}\\
&\bar{\partial}\varphi-\frac{1}{2}[\varphi,\varphi]=0\label{bb3}
\end{align}

Then by the Newlander-Nirenberg theorem(\cite{New57},\cite{Kod05}), the condition (\ref{bb3}) gives a finite open covering $\{U_j\}$ of $M$  and $C^{\infty}$-functions $\xi_j^{\alpha}=\xi_j^{\alpha}(z),\alpha=1,...,n$ on each $U_j$ such that $\xi_j:z\to \xi_j(z)=(\xi_j^1(z),...,\xi_j^n(z))$ gives complex coordinates on $U_j$, and $\{\xi_1,...,\xi_j,...\}$ defines another complex structure on $M$, which we denote by $M_{\varphi}$. By Theorem \ref{m}, the conditions $(\ref{bb1})$ and $(\ref{bb2})$ gives a holomorphic Poisson structure $(\Lambda_0+\Lambda)^{2,0}$ on $M_\varphi$. Recall that $(\Lambda_0+\Lambda)^{2,0}$ means the type $(2,0)$-part of $\Lambda_0+\Lambda$ with respect to the complex structure induced by $\varphi$ (see Remark \ref{renj}).

\begin{remark}
If we replace $\varphi$ by $-\varphi$, then $(\ref{bb1}),(\ref{bb2})$, and $(\ref{bb3})$ are equivalent to
\begin{align}\label{yghb}
L(\Lambda+\varphi)+\frac{1}{2}[\Lambda+\varphi,\Lambda+\varphi]=0 \,\,\,\,\text{where}\,\,L=\bar{\partial}+[\Lambda_0,-]
\end{align}
which is a solution of the Maurer-Cartan equation of a differential graded Lie algebra $(\mathfrak{g}=\bigoplus_{i\geq 0} g^i=\bigoplus_{p+q-1=i, q \geq1} A^{0,p}(M,\wedge^q T_M),L,[-,-])$. This differential graded Lie algebra controls deformations of compact holomorphic Poisson manifolds in the language of functor of Artin rings $($see the second part of the author's Ph.D. thesis \cite{Kim14}$)$.
\end{remark}

\begin{example}[Hitchin-Goto Poisson analytic family] 
Let $(M,\sigma)$ be a compact holomorphic Poisson manifold which satisfies the $\partial\bar{\partial}$-lemma. Then any class $\sigma([\omega])\in H^1(M,\Theta_M)$ for $[\omega]\in H^1(M,\Theta_M^*)$ is tangent to a deformation of complex structure induced by  $\phi(t)=\sigma(\alpha)$ where $\alpha=t\omega+\partial (t^2\beta_2+t^3\beta_3+\cdots)$ for $(0,1)$-forms $\beta_i$ with respect to the original complex structure $($see \cite{Hit12} Theorem 1$)$. Suppose that $\phi(t)=\sigma(\alpha)$ converges for $t\in \Delta \subset \mathbb{C}$. We can consider $\psi=\psi(t):=-\phi(t)$ as a $C^{\infty}$ vector $(0,1)$-form on $M\times \Delta$, and $\sigma$ as a $C^{\infty}$ type $(2,0)$ bivector on $M\times \Delta$. We note that $(\psi(t),\sigma)$ satisfies $[\sigma,\sigma]=0,\bar{\partial}\sigma-[\sigma,\psi(t)]=0$ and $\partial \psi(t)-\frac{1}{2}[\psi(t),\psi(t)]=0$. Then by Newlander-Nirenberg Theorem$($\cite{New57},\cite{Kod05} $p.268$$)$, we can give a holomorphic coordinate on $M\times \Delta$ induced by $\psi$. Let's denote the complex manifold induced by $\psi$ by $\mathcal{M}$. On the other hand, the type $(2,0)$ part $\sigma^{2,0}$ of $\sigma$ with respect to the complex structure $\mathcal{M}$ defines a holomorphic Poisson structure on $\mathcal{M}$. Then the natural projection $\pi:(\mathcal{M},\sigma^{2,0})\to \Delta$ defines a Poisson analytic family of deformations of $(M,\sigma)$. Since $\sigma$ does not depend on $t$, we have $0$ in the Poisson direction under the Poisson Kodaira-Spencer map $\varphi_0: T_0\Delta\to \mathbb{H}^1(M, \Theta_M^\bullet)$ by Theorem $\ref{n}$. More precisely, we have $\varphi_0(\frac{\partial}{\partial t})=(0,-\sigma([\omega]))$.

\end{example}

\section{Theorem of existence for holomorphic Poisson structures}\label{section5}
In this section, we prove `Theorem of existence for holomorphic Poisson structures' as an analogue of `Theorem of existence for complex analytic structures' by Kodaira-Spencer under the assumption that the associated Laplacian operator $\Box$ (induced from the operator $\bar{\partial}+[\Lambda_0,-]$) is strongly elliptic and of diagonal type.
\subsection{Statement of Theorem of existence for holomorphic Poisson structures}\

\begin{theorem}[Theorem of existence for holomorphic Poisson structures]\label{theorem of existence}
Let $(M,\Lambda_0)$ be a compact holomorphic Poisson manifold such that the associated Laplacian operator $\Box$ $($induced from the operator $\bar{\partial}+[\Lambda_0,-]$$)$ is strongly elliptic and of diagonal type. Suppose that $\mathbb{H}^2(M,\Theta_M^\bullet)=0$. Then there exists a Poisson analytic family $(
\mathcal{M},\Lambda,B,\omega)$ with $0\in B\subset \mathbb{C}^m$ satisfying the following conditions:
\begin{enumerate}
\item $\omega^{-1}(0)=(M,\Lambda_0)$
\item The Poisson Kodaira-Spencer map $\varphi_0:\frac{\partial}{\partial t}\to \left(\frac{\partial (M_t,\Lambda_t)}{\partial t}\right)_{t=0}$ with $(M_t,\Lambda_t)=\omega^{-1}(t)$ is an isomorphism of $T_0(B)$ onto $\mathbb{H}^1(M,\Theta_M^\bullet):T_0 B\xrightarrow{\varphi_0} \mathbb{H}^1(M,\Theta_M^\bullet)$.
\end{enumerate}
\end{theorem}

 Let $\{(\pi_{1},\eta_{1}),...,(\pi_{m},\eta_{m})\}$ be a basis of $\mathbb{H}^1(M,\Theta_M^\bullet)$ where $(\pi_{\lambda},\eta_{\lambda})\in A^{0,0}(M,\wedge^2 T_M)\oplus A^{0,1}(M,T_M)$ for $\lambda=1,...,m$. Let $\Delta_\epsilon =\{t\in \mathbb{C}^m||t|<\epsilon\}$ for some $\epsilon>0$. Assume that there is a family $\{(\varphi(t),\Lambda(t))|t\in \Delta_\epsilon \}$ of $C^{\infty}$ vector $(0,1)$-forms $\varphi(t)=\sum_{\lambda=1}^n\sum_{v=1}^n \varphi^{\lambda}_v(z,t) d\bar{z}_v\frac{\partial}{\partial z_{\lambda}}\in A^{0,1}(M,T_M)$ and $C^{\infty}$ type $(2,0)$ bivectors $\Lambda(t)=\sum_{\alpha,\beta=1}^n \Lambda_{\alpha\beta}(z,t)\frac{\partial}{\partial z_\alpha}\wedge \frac{\partial}{\partial z_\beta}\in A^{0,0}(M,\wedge^2 T_M)$ on $M$, which satisfy 
\begin{enumerate}
\item $[\Lambda(t),\Lambda(t)]=0$\\
\item $\bar{\partial} \Lambda(t)-[\Lambda(t),\varphi(t)]=0$\\
\item $\bar{\partial}\varphi(t)-\frac{1}{2}[\varphi(t),\varphi(t)]=0$
\end{enumerate}

and the initial conditions
\begin{align*}
\varphi(0)=0, \Lambda(0)=\Lambda_0,( -\left(\frac{\partial \Lambda(t)}{\partial t_{\lambda}}\right)_{t=0},\left(\frac{\partial \varphi(t)}{\partial t_{\lambda}}\right)_{t=0})=(-\pi_{\lambda},-\eta_{\lambda}),\,\,\,\,\, \lambda=1,...,m,
\end{align*}

Since $\varphi(0)=0$, we may assume $\det(\delta^{\lambda}_v-\sum_{\mu=1}^n\varphi^{\mu}_v(z,t)\overline{\varphi^{\lambda}_{\mu}(z,t)})_{\lambda,\mu=1,...,n}\ne 0$ if $\Delta_\epsilon$ is sufficiently small. Therefore, as in subsection \ref{ss00}, by the Newlander-Nirenberg theorem(\cite{New57},\cite{Kod05} p.268), each $\varphi(t)$ determines a complex structure $M_{\varphi(t)}$ on $M$. The conditions $(2)$ and $(3)$ imply that $(2,0)$-part $\Lambda(t)^{2,0}$ of $\Lambda(t)$ with respect to the complex structure induced from $\varphi(t)$ is a holomorphic Poisson structure on $M_{\varphi(t)}$. If the family $\{(M_{\varphi(t)},\Lambda(t)^{2,0})|t\in \Delta_\epsilon\}$  is a Poisson analytic family, it satisfies the conditions (1) and (2) in Theorem \ref{theorem of existence} by Theorem \ref{n}. We will construct such a family $\{(\varphi(t),\Lambda(t))|t\in \Delta_\epsilon \}$ under the assumption $\mathbb{H}^2(M,\Theta_M^\bullet)=0$ and then show that  $\{(M_{\varphi(t)},\Lambda(t)^{2,0})|t\in \Delta_\epsilon\}$ is a Poisson analytic family in the subsection \ref{subsection}, which completes the proof of Theorem $\ref{theorem of existence}$. 
\begin{remark}
By replacing $\varphi(t)$ by $-\varphi(t)$, it is sufficient to construct $\varphi(t)$ and $\Lambda(t)$ satisfying
\begin{enumerate}
\item $[\Lambda(t),\Lambda(t)]=0$\\
\item $\bar{\partial} \Lambda(t)+[\Lambda(t),\varphi(t)]=0$\\
\item $\bar{\partial}\varphi(t)+\frac{1}{2}[\varphi(t),\varphi(t)]=0$
\end{enumerate}

and the initial conditions
\begin{align}\label{initial1}
\varphi(0)=0, \Lambda(0)=\Lambda_0,( \left(\frac{\partial \Lambda(t)}{\partial t_{\lambda}}\right)_{t=0},\left(\frac{\partial \varphi(t)}{\partial t_{\lambda}}\right)_{t=0})=(\pi_{\lambda},\eta_{\lambda}),\,\,\,\,\, \lambda=1,...,m,
\end{align}
We note that $(1),(2),(3)$ are equivalent to 

\begin{equation}\label{qet}
\bar{\partial} (\varphi(t)+\Lambda(t))+\frac{1}{2}[\varphi(t)+\Lambda(t),\varphi(t)+\Lambda(t)]=0 
\end{equation}
\end{remark}

We construct such $\alpha(t):=\varphi(t)+\Lambda(t)$ in the following subsection.

\subsection{Construction of $\alpha(t)=$$\varphi(t)+\Lambda(t)$}\

We use the Kuranishi's method presented in \cite{Mor71} to construct $\alpha(t)$. First we note the following: let $A^{p}=A^{0,p}(M,T_M)\oplus \cdots \oplus A^{0,0}(M,\wedge^{p+1} T_M)$ and $L=\bar{\partial} +[\Lambda_0,-]$. Then the sequence
\begin{align*}
A^0\xrightarrow{L} A^1\xrightarrow{L} A^2 \xrightarrow{L} \cdots
\end{align*}
is an elliptic complex. So we have the adjoint operator $L^*$, Green's operator $G$, Laplacian operator $\Box:=LL^*+L^*L$ and $H$ where $H$ is the orthogonal projection to the $\Box$-harmonic subspace $\mathbb{H}$ of $\bigoplus_{p\geq 0} A^p$. In particular we have $H:A^p\to \mathbb{H}^p\cong \mathbb{H}^p(M,\Theta_M^\bullet)$. For the detail, we refer to \cite{Wel08}.

We  introduce the H\"{o}lder norms in the spaces $A^{p}=A^{0,p}(M,T_M)\oplus \cdots \oplus A^{0,0}(M, \wedge^{p+1} T_M)$ in the following way: we fix a finite open covering $\{U_j\}$ of $M$ such that $z_j=(z_j^1,...,z_j^n)$ are local coordinates  on $U_j$. Let $\phi\in A^{p}$ which is locally expressed on $U_j$ as
\begin{align*}
\phi=\sum_{r+s=p+1, s\geq 1} \phi_{j \alpha_1\cdots\alpha_r\beta_1\cdots\beta_s}(z)d\bar{z}_j^{\alpha_1}\wedge \cdots \wedge d\bar{z}_j^{\alpha_r}\wedge \frac{\partial}{\partial z_j^{\beta_1}}\wedge\cdots \wedge\frac{\partial}{\partial z_j^{\beta_s}}
\end{align*}
Let $k\in \mathbb{Z},k\geq 0,\theta\in \mathbb{R},0<\theta<1$. Let $h=(h_1,...,h_{2n}),h_i\geq 0,|h|:=\sum_{i=1}^{2n} h_i$ where $n=\dim\, M$. Then denote
\begin{align*}
D_j^h=\left(\frac{\partial}{\partial x_j^1}\right)^{h_1}\cdots \left(\frac{\partial}{\partial x_j^{2n}}\right)^{h_{2n}},\,\,\,\,\,z_j^{\alpha}=x_j^{2\alpha-1}+ix_j^{2\alpha}
\end{align*}

Then the H\"{o}lder norm $||\varphi||_{k+\theta}$ is defined as follows:
{\small{\begin{align*}
||\phi||_{k+\theta}=\max_j \{ \sum_{h, |h|\leq k}\left( \sup_{z\in U_j}|D_j^h \phi_{j \alpha_1\cdots\alpha_r\beta_1\cdots\beta_s}(z)|\right)+\sup_{y,z\in U_j,|h|=k} 
\frac{|D_j^h \phi_{j \alpha_1\cdots\alpha_r\beta_1\cdots\beta_s}(y)-D_j^h \phi_{j \alpha_1\cdots\alpha_r\beta_1\cdots\beta_s}(z)|}{|y-z|^{\theta}} \},
\end{align*}}}
where the sup is over all $\alpha_1,...,\alpha_r,\beta_1,...,\beta_s$.

Now suppose that the associated Laplacian operator $\Box$ induced from  $L=\bar{\partial}+[\Lambda_0,-]$ is a strongly elliptic operator whose principal part is of diagonal type. Then by \cite{Kod05} Appendix Theorem 4.3 page 436, we have a $priori$ estimate
\begin{equation}\label{assumption}
||\phi||_{k+\theta}\leq C(|| \Box \phi||_{k-2+\theta}+||\phi||_0)
\end{equation}
where $k\geq 2$, $C$ is a constant which is independent of $\varphi$ and
\begin{center}
$||\phi||_0=\max_{j,\alpha_1,...,\beta_s} \sup_{z\in U_j} |\phi_{j \alpha_1\cdots\alpha_r\beta_1\cdots\beta_s}(z)|$.
\end{center}

 We will use the following two lemmas.

\begin{lemma}\label{lemma5.2.2}
For $\phi,\psi\in A^1$, we have $||[\phi,\psi]||_{k+\theta}\leq C||\phi||_{k+1+\theta}||\psi||_{k+1+\theta}$, where $C$ is independent of $\phi$ and $\psi$.
\end{lemma}

\begin{lemma}\label{lemma5.2.3}
For $\phi\in A^1$, we have $||G\phi||_{k+\theta}\leq C||\phi||_{k-2+\theta},k\geq 2$, where $C$ depends only on $k$ and $\theta$, not on $\phi$.
\end{lemma}

\begin{proof}
This follows from (\ref{assumption}). See \cite{Mor71} p.160 Proposition 2.3.
\end{proof}

With this preparation, we construct $\alpha(t):=\varphi(t)+\Lambda(t)=\Lambda_0+\sum_{\mu=1}^{\infty} (\varphi_{\mu}(t)+\Lambda_{\mu}(t))$, where
\begin{align*}
\varphi_{\mu}(t)+\Lambda_{\mu}(t)=\sum_{v_1+\cdots+v_m=\mu} (\varphi_{v_1\cdots v_m}+\Lambda_{v_1\cdots v_m})t_1^{v_1}\cdots t_m^{v_m}
\end{align*}
with $\varphi_{v_1\cdots v_m}+\Lambda_{v_1\cdots v_m}\in A^{0,1}(M,T_M)\oplus A^{0,0}(M,\wedge^2  T_M)$ such that

\begin{align}\label{qetyu}
\bar{\partial} \alpha(t)+\frac{1}{2}[\alpha(t),\alpha(t)]=0,\,\,\,\,\,\,\,
\alpha_1(t)=\varphi_1(t)+\Lambda_1(t)=\sum_{v=1}^m (\eta_v+\pi_v)t_v,
\end{align}
where $\{\eta_v+\pi_v\}$ is a basis for $\mathbb{H}^1\cong \mathbb{H}^1(M,\Theta_M^\bullet)$ (see (\ref{initial1})). Let $\beta(t):=\alpha(t)-\Lambda_0=\sum_{\mu=1}^{\infty} (\varphi_{\mu}(t)+\Lambda_{\mu}(t))$. Then $(\ref{qetyu})$ is equivalent to
\begin{align}\label{ghj}
L\beta(t)+\frac{1}{2}[\beta(t),\beta(t)]=0,\,\,\,\,\,\beta_1(t)=\alpha_1(t)
\end{align}

Constructing $\alpha(t)$ is equivalent to constructing $\beta(t)$. We will construct $\beta(t)$ satisfying $(\ref{ghj})$.
Consider the equation
\begin{equation}\label{qety}
\beta(t)=\beta_1(t)-\frac{1}{2}L^*G[\beta(t),\beta(t)],
\end{equation}
where $\beta_1(t)=\alpha_1(t)$.
Then $(\ref{qety})$ has a unique formal power series solution $\beta(t)=\sum_{\mu=1}^{\infty} \beta_{\mu}(t)$, and there exists a $\epsilon>0$ such that  for $t\in \Delta_{\epsilon}=\{t\in \mathbb{C}^m||t|<\epsilon\}$, $\beta(t)=\sum_{\mu=1}^{\infty} \beta_{\mu}(t)$ converges in the norm $||\cdot||_{k+\theta}$ (for the detail, see \cite{Mor71} p.162 Proposition 2.4. By virtue of the integrability condition (\ref{ghj}), we can formally apply their argument.).

\begin{proposition}\label{yui}
$\beta(t)$ satisfies $L\beta(t)+\frac{1}{2}[\beta(t),\beta(t)]=0$ if and only if $H[\beta(t),\beta(t)]=0$, where $H:A^2=A^{0,2}(M,T_M)\oplus A^{0,1}(M,\wedge^2 T_M)\oplus A^{0,0}(M,\wedge^3 T_M)\to \mathbb{H}^2\cong \mathbb{H}^2(M,\Theta_M^\bullet)$ is the orthogonal projection to the harmonic subspace of $A^2$.
\end{proposition}

\begin{proof}
We simply note that $(\bigoplus_{i\geq0} g_i, g_i=\bigoplus_{p+q-1=i,p\geq 0, q\geq 1} A^{0,p}(M,\wedge^q T_M),L=\bar{\partial}+[\Lambda_0,-],[-,-])$ is a differential graded Lie algebra and so the argument in the proof of \cite{Mor71} p.163 Proposition 2.5 is formally applied to our case by Lemma \ref{lemma5.2.2} and Lemma \ref{lemma5.2.3}.
\end{proof}

Now suppose that $\mathbb{H}^2(M,\Theta_M^\bullet)=0$. Then by Proposition \ref{yui}, $\beta(t)$ satisfies $(\ref{ghj})$ for $t\in \Delta_{\epsilon}$. Hence $\alpha(t)=\beta(t)+\Lambda_0=\varphi(t)+\Lambda(t)$ is the desired one satisfying $(\ref{qetyu})$. We note that $\alpha(t)$ has the following property which we need in the construction of a Poisson analytic family in the next subsection.
\begin{proposition}\label{pr}
$\alpha(t)=\beta(t)+\Lambda_0=\varphi(t)+\Lambda(t)$ is $C^{\infty}$ in $(z,t)$ and holomorphic in $t$.
\end{proposition}

\begin{proof}
We note that $\Box$ is a strongly elliptic differential operator whose principal part is of diagonal type by our assumption. So we can formally apply the argument of \cite{Mor71} p.163 Proposition 2.6. See also \cite{Kod05} Appendix p.452 \S 8.
\end{proof}

\subsection{Construction of a Poisson analytic family}\label{subsection}\

 In the previous subsection, we have constructed a family $\{(\varphi(t),\Lambda(t))|t\in \Delta_{\epsilon}\}$ of $C^{\infty}$ vector $(0,1)$-forms $\varphi(t)$ and $C^{\infty}$ type $(2,0)$ bivectors $\Lambda(t)$
 \begin{align*}
 \varphi(t)=\sum_{\lambda=1}^n\sum_{v=1}^n \varphi_v^{\lambda}(z,t)d\bar{z}_v\frac{\partial}{\partial z_{\lambda}},\,\,\,\,\,
 \Lambda(t)=\sum_{\alpha,\beta=1}^n \Lambda_{\alpha\beta}(z,t)\frac{\partial}{\partial z_{\alpha}}\wedge \frac{\partial}{\partial z_{\beta}}
 \end{align*}
 satisfying the integrability condition $[\Lambda(t),\Lambda(t)]=0,\bar{\partial}\Lambda(t)=[\Lambda(t),\varphi(t)],\bar{\partial}\varphi(t)=\frac{1}{2}[\varphi(t),\varphi(t)]$ and the initial conditions $\varphi(0)=0, \Lambda(0)=\Lambda_0, (-\left(\frac{\partial \Lambda(t)}{\partial t_{\lambda}}\right)_{t=0},\left(\frac{\partial \varphi(t)}{\partial t_{\lambda}}\right)_{t=0})=(-\pi_{\lambda},-\eta_{\lambda}),\lambda=1,...,m$, where $\varphi_v^{\lambda}(z,t)$ and $\Lambda_{\alpha\beta}(z,t)$ are $C^{\infty}$ functions of $z^1,...,z^n,t_1,...,t_m$ and holomorphic in $t_1,...,t_m$.

 $(\varphi(t),\Lambda(t))$ determines a holomorphic Poisson structure $(M_{\varphi(t)},\Lambda(t)^{2,0})$ on $M$ for each $t\in \Delta_\epsilon$. In order to show that $\{(M_{\varphi(t)},\Lambda(t)^{2,0})|t\in \Delta_{\epsilon}\}$ is a Poisson analytic family, we consider $\varphi:=\varphi(t)$ as a vector $(0,1)$-form on the complex manifold $M\times \Delta_{\epsilon}$, and $\Lambda:=\Lambda(t)$ as a $(2,0)$ bivector on $M\times \Delta_{\epsilon}$. Then since $\varphi_v^{\lambda}=\varphi_v^{\lambda}(z,t)$ are holomorphic in $t_1,...,t_m$ (Proposition \ref{pr}), we have $\frac{\partial \varphi_v^{\lambda}}{\partial \bar{t}_{\mu}}=0$ in
\begin{align*}
\bar{\partial}\varphi=\sum_{\lambda,v=1}^n \left(\sum_{\beta=1}^n\frac{\partial \varphi_v^{\lambda}}{\partial \bar{z}_{\beta}}d\bar{z}_{\beta}+\sum_{\mu=1}^m \frac{\partial \varphi_v^{\lambda}}{\partial \bar{t}_{\mu}}d\bar{t}_{\mu}\right) \wedge d\bar{z}_v\frac{\partial}{\partial z_{\lambda}}
\end{align*}

Similarly since $\Lambda_{\alpha\beta}(z,t)$ is holomorphic in $t_1,...,t_m$ (Proposition \ref{pr}), we have $\frac{\partial \Lambda_{\alpha\beta}}{\partial \bar{t}_{\mu}}=0$ in

\begin{align*}
\bar{\partial} \Lambda=\sum_{\alpha,\beta} \left(\sum_{v=1}^n \frac{\partial \Lambda_{\alpha\beta}}{\partial \bar{z}_v}d\bar{z}_v+\sum_{\mu=1}^m \frac{\partial \Lambda_{\alpha\beta}}{\partial \bar{t}_{\mu}}d\bar{t}_{\mu}\right)\frac{\partial}{\partial z_{\alpha}}\wedge \frac{\partial}{\partial z_{\beta}}
\end{align*}
Hence $\varphi$ and $\Lambda$ satisfy $\bar{\partial}\varphi=\frac{1}{2}[\varphi,\varphi]$, $ \bar{\partial}\Lambda=[\Lambda,\varphi]$, and $[\Lambda,\Lambda]=0$. Then by the Newlander-Nirenberg theorem(\cite{New57},\cite{Kod05} p.268), $\varphi$ defines a complex structure $\mathcal{M}$ on $M\times \Delta_{\epsilon}$ and
$(2,0)$-part $\Lambda^{2,0}$ of $\Lambda$ defines a holomorphic Poisson structure on $\mathcal{M}$. Let $\omega: \mathcal{M}\to \Delta_\epsilon$ be the natural projection. Then $\{(M_{\varphi(t)},\Lambda(t)^{2,0})|t\in \Delta_{\epsilon}\}$ forms a Poisson analytic family $(\mathcal{M},\Lambda^{2,0},\Delta_{\epsilon},\omega)$ (for the detail, see \cite{Kod05} p.282). This completes the proof of Theorem $\ref{theorem of existence}$.

\section{Theorem of completeness for holomorphic Poisson structures}\label{section6}
\subsection{Statement of Theorem of completeness for holomorphic Poisson structures}
\subsubsection{Change of parameters}(compare \cite{Kod05} p.205)

Consider a Poisson analytic family $\{(M_t,\Lambda_t)|(M_t,\Lambda_t)=\omega^{-1}(t),t\in B\}=(\mathcal{M},\Lambda,B,\omega)$ of compact holomorphic Poisson manifolds, where $B$ is a domain of $\mathbb{C}^m$. Let $D$ be a domain of $\mathbb{C}^r$ and $h:s\to t=h(s),s\in D$, a holomorphic map of $D$ into $B$. Then by changing the parameter from $t$ to $s$, we will construct a Poisson analytic family $\{(M_{h(t)},\Lambda_{h(t)})|s\in D\}$ on the parameter space $D$ in the following.

Let $\mathcal{M}\times_B D:=\{(p,s)\in \mathcal{M}\times B|\omega(p)=h(s)\}$. Then we have the following commutative diagram
\begin{center}
$\begin{CD}
\mathcal{M}\times_B D @>p>> \mathcal{M}\\
@V\pi VV @VV\omega V\\
D @>h>> B
\end{CD}$
\end{center}
such that $(\mathcal{M}\times_B D,D,\pi)$ is a complex analytic family in the sense of Kodaira-Spencer and $\pi^{-1}(s)=M_{h(s)}$. We show that $(\mathcal{M}\times_B D,D,\pi)$ is naturally a Poisson analytic family such that $\pi^{-1}(s)=(M_{h(s)},\Lambda_{h(s)})$ and $p$ is a Poisson map. Note that the bivector field $\Lambda$ on $\mathcal{M}$ can be considered as a bivector field on $\mathcal{M}\times D$ which gives a holomorphic Poisson structure on $\mathcal{M}\times D$. So $(\mathcal{M}\times D,\Lambda)$ is a holomorphic Poisson manifold. We show that $\mathcal{M}\times_B D$ is a holomorphic Poisson submanifold of $(\mathcal{M}\times D,\Lambda)$ and defines a Poisson analytic family. Let $(p_0,s_0)\in \mathcal{M}\times_B D$. Taking a sufficiently small coordinate polydisk $\Delta$ with $h(s_0)\in \Delta$, we represent $(\mathcal{M}_{\Delta},\Lambda_{\Delta})=\omega^{-1}(\Delta)$ in the form of
\begin{align*}
(\mathcal{M}_{\Delta},\Lambda_{\Delta})=(\bigcup_{j=1}^l U_j\times \Delta, \sum_{\alpha,\beta=1}^n g_{\alpha\beta}^j(z_j,t)\frac{\partial}{\partial z_j^{\alpha}}\wedge\frac{\partial}{\partial z_j^{\beta}})
\end{align*}
where each $U_j$ is a polydisk independent of $t$, and $(z_j,t)\in U_j\times \Delta$ and $(z_k,t)\in U_k\times \Delta$ are the same point on $\mathcal{M}_{\Delta}$ if $z_j^{\alpha}=f_{jk}^{\alpha}(z_k,t), \alpha=1,...,n$. Let $E$ be a sufficiently small polydisk of $D$ such that $s_0\in E$ and $h(E)\subset \Delta$. Then we can represent $(\mathcal{M}\times D,\Lambda)$ around $(p_0,s_0)$ in the form of 
\begin{align*}
(\mathcal{M}_{\Delta}\times E,\Lambda|_{\mathcal{M}_{\Delta}\times E})=(\bigcup_{j=1}^l U_j\times \Delta\times E, \sum_{\alpha,\beta=1}^n g_{\alpha\beta}^j(z_j,t)\frac{\partial}{\partial z_j^{\alpha}}\wedge\frac{\partial}{\partial z_j^{\beta}})
\end{align*}
where $(z_j,t,s)\in U_j\times \Delta\times E$ and $(z_k,t,s)\in U_k\times \Delta\times E$ are the same point on $\mathcal{M}_{\Delta}\times E$ if $z_j=f_{jk}(z_k,t)$.
Then we can represent $\mathcal{M}\times_B D$ around $(p_0,s_0)$ in the form of $\bigcup_{j=1}^l U_j\times G_E$, where $G_E=\{(h(s),s)|s\in E\}\subset \Delta\times E$, and $(z_j,h(s),s)\in U_j\times G_E$ and $(z_k,h(s),s)\in U_k\times G_E$ are the same point if $z_j=f_{jk}(z_k,h(s))$. We note that at $(p_0,s_0)\in \mathcal{M}\times_B D\subset \mathcal{M}\times D$, we have $\Lambda_{(p_0,s_0)}=\sum_{\alpha,\beta=1}^n g_{\alpha\beta}^j(p_0,h(s_0))\frac{\partial}{\partial z_j^{\alpha}}|_{p_0}\wedge\frac{\partial}{\partial z_j^{\beta}}|_{p_0}\in \wedge^2 T_{\mathcal{M}\times_B D}$. Hence $\mathcal{M}\times_B D$ is a holomorphic Poisson submanifold of $(\mathcal{M}\times D,\Lambda)$, and $p:(\mathcal{M}\times_B D,\Lambda|_{\mathcal{M}\times_B D})\to (\mathcal{M},\Lambda)$ is a Poisson map.

Since $G_E$ is biholomorphic to $E$. The holomorphic Poisson manifold $(\mathcal{M}\times_B D,\Lambda|_{\mathcal{M}\times_B D})$ is represented locally by the form
\begin{align*}
(\bigcup_{j=1}^l U_j\times E, \sum_{\alpha,\beta=1}^n g_{\alpha\beta}^j(z_j,h(s))\frac{\partial}{\partial z_j^{\alpha}}\wedge\frac{\partial}{\partial z_j^{\beta}})
\end{align*}
where $(z_k,s)\in U_k\times E$ and $(z_j,s)\in U_j\times E$ are the same point if $z_j=f_{jk}(z_k,h(s))$, which shows that $(\mathcal{M}\times_B D,D, \Lambda|_{\mathcal{M}\times_B D},\pi)$ is a Poisson analytic family and $\pi^{-1}(s)=(M_{h(s)},\Lambda_{h(s)})$.

\begin{definition}
The Poisson analytic family $(\mathcal{M}\times_B D,D, \Lambda|_{\mathcal{M}\times_B D},\pi)$ is called the Poisson analytic family induced from $(\mathcal{M},B,\Lambda,\omega)$ by the holomorphic map $h:D\to B$.
\end{definition}

We point out that change of variable formula holds for infinitesimal Poisson deformations as in infinitesimal deformations of complex structures (\cite{Kod05} Theorem 4.7 p.207).

\begin{theorem}
For any tangent vector $\frac{\partial}{\partial s}=c_1\frac{\partial}{\partial s_1}+\cdots +c_r\frac{\partial}{\partial s_r}\in T_s(D)$, the infinitesimal Poisson deformation of $(M_{h(s)},\Lambda_{h(s)})$ along $\frac{\partial}{\partial s}$ is given by
\begin{align*}
\frac{\partial(M_{h(s)},\Lambda_{h(s)})}{\partial{s}}=(\sum_{\lambda=1}^{m} \frac{\partial t_{\lambda}}{\partial s} \frac{\partial M_t}{\partial t_{\lambda}},\sum_{\lambda=1}^{m} \frac{\partial t_{\lambda}}{\partial s}\frac{\partial{\Lambda_t}}{\partial t_{\lambda}})
\end{align*}
\end{theorem}

With this preparation, we discuss a concept of completeness and `Theorem of completeness' in the context of deformations of compact holomorphic Poisson manifolds in the next subsection.

\subsubsection{Statement of `Theorem of completeness for holomorphic Poisson structures'}

\begin{definition}
Let $(\mathcal{M},\Lambda_{\mathcal{M}},B, \omega)$ be a Poisson analytic family of compact holomorphic Poisson manifolds, and $t^0\in B$. Then $(\mathcal{M},\Lambda_{\mathcal{M}},B,\omega)$ is called complete at $t^0\in B$ if for any Poisson analytic family $(\mathcal{N},\Lambda_{\mathcal{N}},D,\pi)$ such that $D$ is a domain of $\mathbb{C}^l$ containing $0$ and that $\pi^{-1}(0)=\omega^{-1}(t^0)$, there is a sufficiently small domain $\Delta$ with $0\in \Delta\subset D$, and a holomorphic map $h:s\to t=h(s)$ with $h(0)=t^0$ such that $(\mathcal{N}_{\Delta},{\Lambda_{\mathcal{N}}}_{\Delta},\Delta,\pi)$ is the Poisson analytic family induced from $(\mathcal{M},\Lambda_{\mathcal{M}},B,\omega)$ by $h$ where $(\mathcal{N}_{\Delta},{\Lambda_{\mathcal{N}}}_{\Delta},\Delta,\pi)$ is the restriction of $(\mathcal{N},\Lambda_{\mathcal{N}},D,\pi)$ to $\Delta$ $($see Remark $\ref{restriction}$$)$. 
\end{definition}

We will prove the following theorem which is an analogue of `Theorem of completeness' by Kodaira-Spencer (see Theorem \ref{kodairacomplete}).

\begin{theorem}[Theorem of completeness for holomorphic Poisson structures]\label{theorem of completeness}\label{complete9}
Let $(\mathcal{M},\Lambda_{\mathcal{M}},B,\omega)$ be a Poisson analytic family of deformations of a compact holomorphic Poisson manifold $(M_0,\Lambda_0)=\omega^{-1}(0)$, $B$ a domain of $\mathbb{C}^m$ containing $0$. If the Poisson Kodaira-Spencer map $\varphi_0:T_0 (B)\to \mathbb{H}^1(M_0,\Theta_{M_0}^\bullet)$ is surjective, the Poisson analytic family $(\mathcal{M},\Lambda_{\mathcal{M}}, B,\omega)$ is complete at $0\in B$.
\end{theorem}

\begin{remark}\label{remark55}
 In order to prove Theorem $\ref{complete9}$, as in \cite{Kod05} Lemma $6.1$ $p.284$,  it suffices to show that for any given Poisson analytic family $(\mathcal{N},\Lambda_{\mathcal{N}},D,\pi)$ with $\pi^{-1}(0)=(M_0,\Lambda_0)$, if we take a sufficiently small domain $\Delta$ with $0\in \Delta \subset D$, we can construct a holomorphic map $h:s\to t=h(s),h(0)=0,$ of $\Delta$ into $B$, and a Poisson holomorphic map $g$ of $(\mathcal{N}_{\Delta},{\Lambda_{\mathcal{N}}}_\Delta)=\pi^{-1}(\Delta)$ into $(\mathcal{M},\Lambda_{\mathcal{M}})$ satisfying the following condition: $g$ is a Poisson holomorphic map extending the identity $g_0:\pi^{-1}(0)=(M_0,\Lambda_0)\to (M_0,\Lambda_0)$, and $g$ maps each $(N_s,\Lambda_{N_s})=\pi^{-1}(s)$ Poisson biholomorphically onto $(M_{h(s)},\Lambda_{M_{h(s)}})$. We will construct such $h$ and $g$ by extending Kodaira's elementary method $($see \cite{Kod05} Chapter $6$$)$.
 \end{remark}
 
 \subsection{Preliminaries}\label{preli}\

We extend the argument of \cite{Kod05} p.285-286 (to which we refer for the detail) in the context of a Poisson analytic family. We tried to keep notational consistency with \cite{Kod05}.  

Since the problem is local with respect to $B$, we may assume that $B=\{t\in \mathbb{C}^m||t|<1\}$, and $(\mathcal{M},\Lambda_{\mathcal{M}}, B,\omega)$ is written in the following form
\begin{align*}
(\mathcal{M},\Lambda_{\mathcal{M}})=\bigcup_j(\mathcal{U}_j,\Lambda_{M_j}),\,\,\,\,\, \mathcal{U}_j=\{(\xi_j,t)\in \mathbb{C}^n\times B||\xi_j|<1\}
\end{align*}
where the Poisson structure $\Lambda_{\mathcal{M}}$ is given by $\Lambda_{M_j}=\sum_{r,s=1}^n \Lambda_{M_j}^{r,s}(\xi_j,t)\frac{\partial}{\partial \xi_j^r}\wedge \frac{\partial}{\partial \xi_j^s}$ on $\mathcal{U}_j$ with $\Lambda_{M_j}^{r,s}(\xi_j,t)=-\Lambda_{M_j}^{s,r}(\xi_j,t)$, and $\omega(\xi_j,t)=t$. For $\mathcal{U}_j\cap \mathcal{U}_k\ne \emptyset,(\xi_j,t)$ and $(\xi_k,t)$ are the same point of $\mathcal{M}$ if
\begin{align}\label{vv2}
\xi_j=g_{jk}(\xi_k,t)=(g_{jk}^1(\xi_k,t),...,g_{jk}^n(\xi_k,t)),
\end{align}
where $g_{jk}^{\alpha}(\xi_k,t)$, $\alpha=1,...,n$, are holomorphic functions on $\mathcal{U}_j\cap \mathcal{U}_k$, and we have the following relations
\begin{align}\label{vv10}
\Lambda_{M_j}^{r,s}(g_{jk}(\xi_k,t),t)=\sum_{p,q=1}^n\Lambda_{M_k}^{p,q}(\xi_k,t)\frac{\partial g_{jk}^r}{\partial \xi_k^p}\frac{\partial g_{jk}^s}{\partial \xi_k^q}.
\end{align}

Similarly we assume that $D=\{s\in \mathbb{C}^l||s|<1\}$, and $(\mathcal{N},\Lambda_{\mathcal{N}},D,\pi)$ is written in the following form
\begin{align*}
(\mathcal{N},\Lambda_{\mathcal{N}})=\bigcup_j (\mathcal{W}_j,\Lambda_{N_j}),\,\,\,\,\, \mathcal{W}_j=\{(z_j,s)\in \mathbb{C}^n\times D||z_j|<1\}
\end{align*} 
where the Poisson structure $\Lambda_{\mathcal{N}}$ is given by $\Lambda_{N_j}=\sum_{\alpha,\beta=1}^n \Lambda_{N_j}^{\alpha,\beta}(z_j,t)\frac{\partial}{\partial z_j^{\alpha}}\wedge \frac{\partial}{\partial z_j^{\beta}}$ on $\mathcal{W}_j$ with $\Lambda_{N_j}^{\alpha,\beta}(z_j,t)=-\Lambda_{N_j}^{\beta,\alpha}(z_j,t)$, and $\pi(z_j,s)=s$. For $\mathcal{W}_j\cap \mathcal{W}_k\ne\emptyset, (z_j,s)$ and $(z_k,s)$ are the same point of $\mathcal{N}$ if
\begin{align}\label{vv1}
z_j=f_{jk}(z_k,s)=(f_{jk}^1(z_k,s),...,f_{jk}^n(z_k,s)).
\end{align}
and we have
\begin{align}\label{vv11}
\Lambda_{N_j}^{\alpha,\beta}(f_{jk}(z_k,s),s)=\sum_{a,b=1}^n\Lambda_{N_k}^{a,b}(z_k,s)\frac{\partial f_{jk}^{\alpha}}{\partial z_k^a}\frac{\partial f_{jk}^{\beta}}{\partial z_k^b}
\end{align}
Since $(N_0,\Lambda_{N_0})=\pi^{-1}(0)=(M_0,\Lambda_0)=\omega^{-1}(0)=(M_0, \Lambda_{M_0})$, we may assume $(\mathcal{W}_j\cap N_0,\Lambda_{N_{0j}})=(\mathcal{U}_j\cap M_0,\Lambda_{M_{0j}})$ where $\Lambda_{N_{0j}}:=\sum_{\alpha,\beta=1}^n \Lambda_{N_j}^{\alpha,\beta}(z_j,0)\frac{\partial}{\partial z_j^{\alpha}}\wedge \frac{\partial}{\partial z_j^{\beta}}$, and $\Lambda_{M_{0j}}:=\sum_{r,s=1}^n \Lambda_{M_j}^{r,s}(\xi_j,0)\frac{\partial}{\partial \xi_j^{r}}\wedge \frac{\partial}{\partial \xi_j^{s}}$, and assume that the local coordinates $(\xi_j,0)$ and $(z_j,0)$ coincide on $\mathcal{W}_j \cap N_0=\mathcal{U}_j\cap M_0$. In other words, if $\xi_j^1=z_j^1,...,\xi_j^n=z_j^n$, $(\xi_j,0)$ and $(z_j,0)$ are the same point of $\mathcal{W}_j\cap N_0=\mathcal{U}_j \cap M_0$, and we have $\Lambda_{N_j}^{\alpha,\beta}(z_j,0)=\Lambda_{M_j}^{\alpha,\beta}(\xi_j,0)$. Putting
\begin{align}\label{ii33}
b_{jk}(\xi_k):=g_{jk}(\xi_k,0),\,\,\,\,\,\,\, \Lambda_{M_{0j}}^{\alpha,\beta}(\xi_j):= \Lambda_{M_j}^{\alpha,\beta}(\xi_j,0)
\end{align}
Then from $(\ref{vv2})$ and $(\ref{vv1})$, we have
\begin{align}\label{ii34}
b_{jk}(z_k)=f_{jk}(z_k,0),\,\,\,\,\,\,\, \Lambda_{M_{0j}}^{\alpha,\beta}(z_j)=\Lambda_{M_j}^{\alpha,\beta}(z_j,0)=\Lambda_{N_j}^{\alpha,\beta}(z_j,0)
\end{align}
In conclusion, we have 
\begin{align}\label{covering}
(N_0,\Lambda_{N_0})=(M_0,\Lambda_{M_0}=\Lambda_0)=\bigcup_j (U_j,\Lambda_{M_{0j}}),\,\,\,\,\,U_j=\mathcal{W}_j\cap N_0=\mathcal{U}_j\cap M_0,
\end{align}
such that $\{z_j\},z_j=(z_j^1,...,z_j^n)$, is a system of local complex coordinates of the complex manifold $N_0=M_0$ with respect to $\{U_j\}$, and the Poisson structure is given by $\Lambda_{M_{0j}}=\sum_{\alpha,\beta=1}^n\Lambda_{M_{0j}}^{\alpha,\beta}(z_j)\frac{\partial}{\partial z_j^{\alpha}}\wedge \frac{\partial}{\partial z_j^{\beta}}$ on $U_j$ with $\Lambda_{M_{0j}}^{\alpha,\beta}(z_j)=-\Lambda_{M_{0j}}^{\beta,\alpha}(z_j)$. The coordinate transformation on $U_j\cap U_k$ is given by $z_j^{\alpha}=b_{jk}^{\alpha}(z_k),\alpha=1,...,n,$ and we have 
\begin{align}\label{vv3}
\Lambda_{M_{0j}}^{\alpha,\beta}(b_{jk}(z_k))=\sum_{a,b=1}^n\Lambda_{M_{0k}}^{a,b}(z_k)\frac{\partial b_{jk}^{\alpha}}{\partial z_k^a}\frac{\partial b_{jk}^{\beta}}{\partial z_k^b}.
\end{align}

\subsection{Construction of Formal Power Series}\

As in Remark \ref{remark55}, we have to define a holomorphic map $h:s\to t=h(s)$ with $h(0)=0$ of $\Delta=\{s\in D||s|<\epsilon\}$ into $B$ for a sufficiently small $\epsilon>0$, and to extend the identity $g_0:(N_0,\Lambda_0)\to (M_0=N_0,\Lambda_0)$ to a Poisson holomorphic map $g:\pi^{-1}(\Delta)=(\mathcal{N}_{\Delta},\Lambda_{N_{\Delta}})\to (\mathcal{M},\Lambda)$ such that $\omega\circ g=h\circ \pi$. 

 We begin with constructing formal power series $h(s)=\sum_{v=1}^\infty h_v(s)$ of $s_1,...,s_l$ where $h_v(s)$ is a homogenous polynomial of degree $v$ of $s_1,...,s_l$, and formal power series $g_j(z_j,s)=z_j+\sum_{v=1}^\infty g_{j|v}(z_j,s)$ in terms of $s_1,...,s_l$ for each $U_j$ in (\ref{covering}), whose coefficients are vector valued holomorphic functions on $U_j$ where $g_{j|v}(z_j,s)=\sum_{v_1+\cdots+v_l=v} g_{jv_1\cdots v_l}(z_j)s_1^{v_1}\cdots s_l^{v_l}$
is a homogeneous polynomial of degree $v$ of $s_1,...,s_l$, and each component $g_{jv_1\cdots v_n}^{\alpha}(z_j),\alpha=1,...,n$ of the coefficient $
g_{jv_1\dots v_l}(z_j)=(g_{jv_1\cdots v_l}^{1}(z_j),...,g_{jv_1\cdots v_l}^{n}(z_j))$ is a holomorphic function of $z_j^1,...,z_j^n$ defined on $U_j$. The formal power series $h(s)$ and $g_j(z_j,s)$ will satisfy
\begin{align}
g_j(f_{jk}(z_k,s),s)=g_{jk}(g_k(z_k,s),h(s)) \,\,\,\,\,\text{on}\,\,\,U_j\cap U_k\ne \emptyset \label{aa90}\\
\Lambda_{M_j}^{r,s}(g_j(z_j,s),h(s))=\sum_{\alpha,\beta=1}^n \Lambda_{N_j}^{\alpha,\beta}(z_j,s)\frac{\partial g_j^r}{\partial z_j^{\alpha}}\frac{\partial g_j^s}{\partial z_j^{\beta}}\,\,\,\,\,\text{on}\,\,\,U_j.\label{aa91}
\end{align}
For the meaning of (\ref{aa90}), we refer to \cite{Kod05} p.286-288. (\ref{aa90}) is a crucial condition for the proof of `Theorem of completeness for complex analytic structures' (Theorem \ref{kodairacomplete}). However, in order to prove `Theorem of completeness for holomorphic Poisson structures' (Theorem \ref{complete9}), we need to impose additional condition $(\ref{aa91})$ which means that $g_j(z_j,s)$ is a Poisson map.

We will write
\begin{align*}
h^v(s)&:=h_1(s)+\cdots+h_v(s).\\
g_j^v(z_j,s)&:=z_j+g_{j|1}(z_j,s)+\cdots g_{j|v}(z_j,s).
\end{align*}

The equalities (\ref{aa90}) and (\ref{aa91}) are equivalent to the following system of the infinitely many congruence:
\begin{align}\label{aa11}
g_j^v(f_{jk}(z_k,s),s)&\equiv_v g_{jk}(g_k^v(z_k,s),h^v(s))\\
\Lambda_{M_j}^{r,s}(g_j^v(z_j,s),h^v(s))&\equiv_{v}\sum_{\alpha,\beta=1}^n \Lambda_{N_j}^{\alpha,\beta}(z_j,s)\frac{\partial {g_j^r}^v}{\partial z_j^{\alpha}}\frac{\partial {g_j^s}^v}{\partial z_j^{\beta}}\label{aa12}
\end{align}
for $v=0,1,2,3,...$ where we indicate by $\equiv_v$ that the power series expansions with respect to $s$ of both sides of (\ref{aa11}) and (\ref{aa12}) coincide up to the term of degree $v$.

 We will construct $h^v(s), g_j^v(z_j,s)$ satisfying $(\ref{aa11})_v$ and $(\ref{aa12})_v$ inductively on $v$. Then the resulting formal power series $h(s)$ and $g_j(z_j,s)$ will satisfy $(\ref{aa90})$ and $(\ref{aa91})$. For $v=0$, since $h^0(s)=0$ and $g_j^0(z_j,s)=z_j$, $(\ref{aa11})_0$ and $(\ref{aa12})_0$ hold by $(\ref{ii33}),(\ref{ii34})$. Now suppose that $h^{v-1}(s)$ and $g_j^{v-1}(z_j,s)$ are already constructed in such a manner that, for each $U_j\cap U_k\ne \emptyset$,
\begin{align}
g_j^{v-1}(f_{jk}(z_k,s),s)\equiv_{v-1} g_{jk}(g_k^{v-1}(z_k,s),h^{v-1}(s))
\end{align}
and for each $U_j$,
\begin{align}\label{aa33}
\Lambda_{M_j}^{r,s}(g_j^{v-1}(z_j,s),h^{v-1}(s))\equiv_{v-1}\sum_{\alpha,\beta=1}^n \Lambda_{N_j}^{\alpha,\beta}(z_j,s)\frac{\partial {g_j^r}^{v-1}}{\partial z_j^{\alpha}}\frac{\partial {g_j^s}^{v-1}}{\partial z_j^{\beta}}
\end{align}
hold. We will find $h_{v}(s)$ and $g_{j|v}(z_j,s)$ such that  $h^v(s)=h^{v-1}(s)+h_v(s)$, and  $g_j^v(z_j,s)=g_j^{v-1}(z_j,s)+g_{j|v}(z_j,s)$ satisfy $(\ref{aa11})_v$ on each $U_j\cap U_k$ and $(\ref{aa12})_v$ on each $U_j$.

For this purpose, we start from finding the equivalent conditions to $(\ref{aa11})_v$ and $(\ref{aa12})_v$, and then interpret them cohomologically by using \u{C}ech resolution of the complex of sheaves (\ref{complex}) with respect to the open covering $(\ref{covering})$ of $M_0=N_0$ (see Lemma \ref{hu7} below).

 For the equivalent condition to $(\ref{aa11})_v$, we briefly summarize Kodaira's result in the following: if we let $\Gamma_{jk|v}$ denote the sum of the terms of degree $v$ of $g_j^{v-1}(f_{jk}(z_k,s),s)-g_{jk}(g_k^{v-1}(z_k,s),h^{v-1}(s))$:
 \begin{align}\label{tt002}
\Gamma_{jk}(z_j,s)\equiv_v g_j^{v-1}(f_{jk}(z_k,s),s)-g_{jk}(g_k^{v-1}(z_k,s),h^{v-1}(s)),
\end{align}
then $(\ref{aa11})_v$ is equivalent to the following:
\begin{align}\label{vv9}
\Gamma_{jk|v}(z_j,s)=\sum_{\beta=1}^n \frac{\partial z_j}{\partial z_k^{\beta}}g_{k|v}^{\beta}(z_k,s)-g_{j|v}(z_j,s)+\sum_{u=1}^m \left(\frac{\partial g_{jk}(z_k,t)}{\partial t_u}\right)_{t=0} h_{u|v}(s)
\end{align}
where $z_k$ and $z_j=b_{jk}(z_k)$ are the local coordinates of the same point of $N_0=M_0$ (for the detail, see \cite{Kod05} p.289-290).

On the other hand, let's find the equivalent condition to $(\ref{aa12})_v$. We note that
\begin{align}\label{aa13}
\Lambda_{M_j}^{r,s}(g_j^v(z_j,s),h^v(s))=\Lambda_{M_j}^{r,s}(g_j^{v-1}(z_j,s)+g_{j|v}(z_j,s),h^{v-1}(s)+h_v(s))
\end{align}
By expanding $\Lambda_{M_j}^{r,s}(\xi_j+\xi, t+\omega)$ into power series of $\xi^1,...,\xi^n,\omega_1,...,\omega_m$, we obtain
\begin{align}\label{aa14}
\Lambda_{M_j}^{r,s}(\xi_j+\xi,t+\omega)=\Lambda_{M_j}^{r,s}(\xi_j,t)+\sum_{\beta=1}^n\frac{\partial \Lambda_{M_j}^{r,s}}{\partial \xi_j^{\beta}}(\xi_j,t)\xi^{\beta}+\sum_{u=1}^m \frac{\partial \Lambda_{M_j}^{r,s}}{\partial t_u}(\xi_j,t)\omega_u+\cdots
\end{align}
where $\cdots$ denotes the terms of degree $\geq 2$ in $\xi^1,...,\xi^n,\omega_1,...,\omega_m$. Let's consider the left hand side of $(\ref{aa12})_v$. Then from $(\ref{aa13})$, $(\ref{aa14})$, and $(\ref{ii34})$, we have
\begin{align}\label{aa30}
&\Lambda_{M_j}^{r,s}(g_j^v(z_j,s),h^v(s))-\Lambda_{M_j}^{r,s}(g_j^{v-1}(z_j,s),h^{v-1}(s))\\ &\equiv_v \sum_{\beta=1}^n\frac{\partial \Lambda_{M_j}^{r,s}}{\partial \xi_j^{\beta}}(g_j^{v-1}(z_j,s),h^{v-1}(s))g_{j|v}^{\beta}(z_j,s)+\sum_{u=1}^m \frac{\partial \Lambda_{M_j}^{r,s}}{\partial t_u}(g_j^{v-1}(z_j,s),h^{v-1}(s))h_{u|v}(s)\notag\\
&\equiv_v \sum_{\beta=1}^n\frac{\partial \Lambda_{M_j}^{r,s}}{\partial \xi_j^{\beta}}(g_j^{v-1}(z_j,0),h^{v-1}(0))g_{j|v}^{\beta}(z_j,s)+\sum_{u=1}^m \frac{\partial \Lambda_{M_j}^{r,s}}{\partial t_u}(g_j^{v-1}(z_j,0),h^{v-1}(0))h_{u|v}(s)\notag\\
&=\sum_{\beta=1}^n\frac{\partial \Lambda_{M_{0j}}^{r,s}}{\partial z_j^{\beta}}g_{j|v}^{\beta}(z_j,s)+\sum_{u=1}^m \left( \frac{\partial\Lambda_{M_{j}}^{r,s}(z_j,t)}{\partial t_u}\right)_{t=0} h_{u|v}(s)\notag
\end{align}
On the other hand, let's consider the right hand side of $(\ref{aa12})_v$. Then from $(\ref{ii34})$, we have
\begin{align}\label{aa31}
&\sum_{\alpha,\beta=1}^n \Lambda_{N_j}^{\alpha,\beta}(z_j,s)\frac{\partial {g_j^r}^{v}}{\partial z_j^{\alpha}}\frac{\partial {g_j^s}^{v}}{\partial z_j^{\beta}}=\sum_{\alpha,\beta=1}^n \Lambda_{N_j}^{\alpha,\beta}(z_j,s)\frac{\partial ({g_j^r}^{v-1}+g_{j|v}^r)}{\partial z_j^{\alpha}}\frac{\partial ({g_j^s}^{v-1}+g_{j|v}^s)}{\partial z_j^{\beta}}\\
&\equiv_v \sum_{\alpha,\beta=1}^n\Lambda_{N_j}^{\alpha,\beta}(z_j,s)\frac{\partial {g_j^r}^{v-1}}{\partial z_j^{\alpha}}\frac{\partial {g_j^s}^{v-1}}{\partial z_j^{\beta}}+\sum_{\alpha,\beta=1}^n \Lambda_{N_j}^{\alpha,\beta}(z_j,s)\frac{\partial {g_j^r}^{v-1}}{\partial z_j^{\alpha}}\frac{\partial g_{j|v}^s}{\partial z_j^{\beta}}+\sum_{\alpha,\beta=1}^n \Lambda_{N_j}^{\alpha,\beta}(z_j,s)\frac{\partial g_{j|v}^r}{\partial z_j^{\alpha}}\frac{\partial {g_j^s}^{v-1}}{\partial z_j^{\beta}}\notag\\
&\equiv_v \sum_{\alpha,\beta=1}^n\Lambda_{N_j}^{\alpha,\beta}(z_j,s)\frac{\partial {g_j^r}^{v-1}}{\partial z_j^{\alpha}}\frac{\partial {g_j^s}^{v-1}}{\partial z_j^{\beta}}+\sum_{\beta=1}^n\Lambda_{M_{0j}}^{r,\beta}(z_j)\frac{\partial g_{j|v}^s}{\partial z_j^{\beta}}+\sum_{\alpha=1}^n \Lambda_{M_{0j}}^{\alpha,s} (z_j)\frac{\partial g_{j|v}^r}{\partial z_j^{\alpha}}\notag
\end{align}
Then from $(\ref{aa30})$ and $(\ref{aa31})$,  the congruence $(\ref{aa12})_v$ is equivalent to the following:
\begin{align}\label{aa32}
&-\Lambda_{M_j}^{r,s}(g_j^{v-1}(z_j,s),h^{v-1}(s))+\sum_{\alpha,\beta=1}^n\Lambda_{N_j}^{\alpha,\beta}(z_j,s)\frac{\partial {g_j^r}^{v-1}}{\partial z_j^{\alpha}}\frac{\partial {g_j^s}^{v-1}}{\partial z_j^{\beta}}\\
&\equiv_v \sum_{\beta=1}^n\frac{\partial \Lambda_{M_{0j}}^{r,s}}{\partial z_j^{\beta}}g_{j|v}^{\beta}(z_j,s)+\sum_{u=1}^m \left( \frac{\partial\Lambda_{M_{j}}^{r,s}(z_j,t)}{\partial t_u}\right)_{t=0} h_{u|v}(s)-\sum_{\beta=1}^n\Lambda_{M_{0j}}^{r,\beta}(z_j)\frac{\partial g_{j|v}^s}{\partial z_j^{\beta}}-\sum_{\alpha=1}^n \Lambda_{M_{0j}}^{\alpha,s} (z_j)\frac{\partial g_{j|v}^r}{\partial z_j^{\alpha}}\notag
\end{align}
By induction hypothesis (\ref{aa33}), the left hand side of (\ref{aa32}) $\equiv_{v-1} 0$. Hence if we let $\lambda_{j|v}^{r,s}$ denote the terms of degree $v$ of the left hand side of (\ref{aa32}), we have
\begin{align}\label{aa35}
\lambda_{j|v}^{r,s} (z_j,s) \equiv_v -\Lambda_{M_j}^{r,s}(g_j^{v-1}(z_j,s),h^{v-1}(s))+\sum_{\alpha,\beta=1}^n \Lambda_{N_j}^{\alpha,\beta}(z_j,s)\frac{\partial {g_j^r}^{v-1}}{\partial z_j^{\alpha}}\frac{\partial {g_j^s}^{v-1}}{\partial z_j^{\beta}}
\end{align}
Hence from $(\ref{aa32})$ and $(\ref{aa35})$, the congruence $(\ref{aa12})_v$ is equivalent to the following:
\begin{align}\label{aa34}
\lambda_{j|v}^{r,s}(z_j,s)=\sum_{\beta=1}^n\frac{\partial \Lambda_{M_{0j}}^{r,s}}{\partial z_j^{\beta}}g_{j|v}^{\beta}(z_j,s)+\sum_{u=1}^m \left( \frac{\partial\Lambda_{M_{j}}^{r,s}(z_j,t)}{\partial t_u}\right)_{t=0} h_{u|v}(s)-\sum_{\beta=1}^n \Lambda_{M_{0j}}^{r,\beta}(z_j)\frac{\partial g_{j|v}^s}{\partial z_j^{\beta}}-\sum_{\alpha=1}^n \Lambda_{M_{0j}}^{\alpha,s} (z_j)\frac{\partial g_{j|v}^r}{\partial z_j^{\alpha}}
\end{align}
where $z_k$ and $z_j=b_{jk}(z_k)$ are the local coordinates of the same point of $N_0=M_0$. We note that $\lambda_{j|v}^{r,s}(z_j,s)=-\lambda_{j|v}^{s,r}(z_j,s)$.

As in \cite{Kod05} p.291, to interpret the meaning of $(\ref{vv9})_v$, and $(\ref{aa34})_v$ in terms of \u{C}ech resolution of the complex of sheaves $(\ref{complex})$ with respect to the open covering $(\ref{covering})$ of $M_0=N_0$, we introduce holomorphic vector fields and bivector fields as follows:
\begin{align}
\theta_{ujk}&=\sum_{\alpha=1}^n \theta_{ujk}^{\alpha}(z_j)\frac{\partial}{\partial z_j^{\alpha}}=\sum_{\alpha=1}^n \left(\frac{\partial g_{jk}^{\alpha}(z_k,t)}{\partial t_u}\right)_{t=0}\frac{\partial}{\partial z_j^{\alpha}},\,\,\,\,\, z_k=b_{jk}(z_j)\label{p1}\\
\Lambda_{uj}'&=\sum_{r,s=1}^n \Lambda_{uj}'^{r,s}(z_j)\frac{\partial}{\partial z_j^r}\wedge\frac{\partial}{\partial z_j^s}:=\sum_{r,s=1}^n \left( \frac{\partial\Lambda_{M_{j}}^{r,s}(z_j,t)}{\partial t_u}\right)_{t=0} \frac{\partial}{\partial z_j^r}\wedge \frac{\partial}{\partial z_j^s}\label{p2}\\
\Gamma_{jk|v}(s)&=\sum_{\alpha=1}^n \Gamma_{jk|v}^{\alpha}(z_j,s)\frac{\partial}{\partial z_j^{\alpha}}\label{p3}\\
g_{k|v}(s)&=\sum_{\beta=1}^n g_{k|v}^{\beta}(z_k,s)\frac{\partial}{\partial z_k^{\beta}}\label{p4}\\
\lambda_{j|v}(s)&=\sum_{r,s=1}^n \lambda_{j|v}^{r,s}(z_j,s)\frac{\partial}{\partial z_j^r}\wedge \frac{\partial}{\partial z_j^s}\label{p5}
\end{align}
By $(\ref{covering})$, $\mathcal{U}:=\{U_j\}$ is a finite open covering of $M_0=N_0$. Since we assume that $\xi_j^{\alpha}=z_j^{\alpha},\alpha=1,...,n$ in subsection \ref{preli}, the 1-cocycle $(\{\Lambda_{uj}'\},\{\theta_{ujk}\})\in C^0(\mathcal{U},\wedge^2\Theta_{M_0})\oplus C^1(\mathcal{U},\Theta_{M_0})$ in $(\ref{p1})$ and $(\ref{p2})$ represents the infinitesimal Poisson deformation $(\Lambda_u',\theta_u)=\varphi_0(\frac{\partial}{\partial t_u}) \in \mathbb{H}^1(M_0,\Theta_{M_0}^\bullet)$ where $\varphi_0$ is the Poisson Kodaira-Spencer map of the Poisson analytic family $(\mathcal{M},\Lambda_{\mathcal{M}},B, \omega)$ (see Proposition \ref{gg} and Definition \ref{mapping}). Since the coefficients $\Gamma_{jkv_1\cdots v_l}$ of the homogeneous polynomial $\Gamma_{jk|v}(s)=\sum_{v_1+\cdots v_l=v} \Gamma_{jkv_1\cdots v_l} s_1^{v_1}\cdots s_l^{v_l}$ are holomorphic vector fields on $U_j \cap U_k$, $\{\Gamma_{jk|v}(s)\}=\sum_{v_1+\cdots v_l=v} \{\Gamma_{jkv_1\cdots v_l} \} s_1^{v_1}\cdots s_l^{v_l}$ is a homogenous polynomial of degree $v$ whose coefficients are $\{\Gamma_{jkv_1\cdots v_l}\}\in C^1(\mathcal{U},\Theta_{M_0})$. Since the coefficients $\lambda_{jv_1\cdots v_l}$ of the homogenous polynomial $\lambda_{j|v}(s)=\sum_{v_1+\cdots +v_l=v} \lambda_{jv_1\cdots v_l} s_1^{v_1}\cdots s_l^{v_l}$ are holomorphic bivector fields on $U_j$, $\{\lambda_{j|v}(s)\}=\sum_{v_1+\cdots +v_l=v} \{\lambda_{jv_1\cdots v_l}\}s_1^{v_1}\cdots s_l^{v_l}$ is a homogenous polynomial of degree $v$ whose coefficients are $\{\lambda_{jv_1\cdots v_l}\}\in C^0(\mathcal{U},\wedge^2 \Theta_
{M_0})$. Similarly $\{g_{j|v}(s)\}=\sum_{v_1+\cdots +v_l=v} \{g_{jv_1\cdots v_l} \} s_1^{v_1}\cdots s_l^{v_l}$ is a homogenous polynomial of degree $v$ whose coefficients are $\{g_{jv_1\cdots v_l}\}\in C^0(\mathcal{U},\Theta_{M_0})$.  We claim that
\begin{lemma}\label{lemmai}
The following equation holds
\begin{align}\label{hu7}
(\{\lambda_{j|v}(s)\},\{\Gamma_{jk|v}(s)\})=\sum_{u=1}^m h_{u|v}(s)(\{\Lambda_{uj}'\},\{\theta_{ujk}\})-\delta_{HP}(\{g_{j|v}(s)\})
\end{align}
where $\delta_{HP}(\{g_{j|v}(s)\}):=(-\delta(\{g_{j|v}\})=\{g_{j|v}(s)-g_{k|v}(s)\}, \{[\sum_{r,s=1}^n \Lambda_{M_{0j}}^{r,s}(z_j)\frac{\partial}{\partial z_j^r}\wedge \frac{\partial}{\partial z_j^s},g_{j|v}(s)]\})$. Here $\delta$ is the \u{C}ech map.

\end{lemma}

\begin{proof}
First, we have $\{\Gamma_{jk|v}(s)\}=\sum_{u=1}^m h_{u|v}(s)\{\theta_{rjk}\}+\delta\{g_{j|v}(s)\}$ (see \cite{Kod05} p.291). 

It remains to show that $\{\lambda_{j|v}(s)\}=\sum_{u=1}^m h_{u|v}(s)\{\Lambda'_{uj}\}-\{[\sum_{r,s} \Lambda_{M_{0j}}^{r,s}(z_j)\frac{\partial}{\partial z_j^r}\wedge \frac{\partial}{\partial z_j^s},g_{j|v}(s)]\}$. 
Indeed,
\begin{align*}
&\sum_{u=1}^m h_{u|v}(s)\Lambda'_{uj}-\sum_{r,s,\beta=1}^n[\Lambda_{M_{0j}}^{r,s}(z_j)\frac{\partial}{\partial z_j^r}\wedge \frac{\partial}{\partial z_j^s}, g_{j|v}^{\beta}(z_j,s)\frac{\partial}{\partial z_j^{\beta}}]\\
&=\sum_{u=1}^m h_{u|v}(s)\Lambda'_{uj}-\sum_{r,s,\beta=1}^n \Lambda_{M_{0j}}^{r,s}\frac{\partial g_{j|v}^{\beta}}{\partial z_j^r}\frac{\partial}{\partial z_j^{\beta}}\wedge \frac{\partial}{\partial z_j^s}+\sum_{r,s,\beta=1}^n g_{j|v}^{\beta}\frac{\partial \Lambda_{M_{0j}}^{r,s}}{\partial z_j^{\beta}}\frac{\partial}{\partial z_j^r}\wedge \frac{\partial}{\partial z_j^s}-\sum_{r,s,\beta=1}^n \Lambda_{M_{0j}}^{r,s}\frac{\partial g_{j|v}^{\beta}}{\partial z_j^s}\frac{\partial }{\partial z_j^r}\wedge \frac{\partial}{\partial z_j^{\beta}}\\
&=\sum_{u=1}^m h_{u|v}(s)\Lambda'_{uj}-\sum_{r,s,\beta=1}^n \Lambda_{M_{0j}}^{\beta,s}\frac{\partial g_{j|v}^{r}}{\partial z_j^{\beta}}\frac{\partial}{\partial z_j^{r}}\wedge \frac{\partial}{\partial z_j^s}+\sum_{r,s,\beta=1}^n g_{j|v}^{\beta}\frac{\partial \Lambda_{M_{0j}}^{r,s}}{\partial z_j^{\beta}}\frac{\partial}{\partial z_j^r}\wedge \frac{\partial}{\partial z_j^s}-\sum_{r,s,\beta=1}^n \Lambda_{M_{0j}}^{r,\beta}\frac{\partial g_{j|v}^{s}}{\partial z_j^{\beta}}\frac{\partial }{\partial z_j^r}\wedge \frac{\partial}{\partial z_j^s}\\
&=\lambda_{j|v}(s)
\end{align*}
by $(\ref{aa34})$, $(\ref{p2})$ and $(\ref{p5})$.
\end{proof}

Thus in order to construct $h^v(s)=h^{v-1}(s)+h_v(s)$, $g_j^v(z_j,s)=g_j^{v-1}(z_j,s)+g_{j|v}(z_j,s)$ so that $(\ref{aa11})_v$ and $(\ref{aa12})_v$ hold, it suffices to obtain solutions $h_{u|v},u=1,...,m,\{g_{j|v}(s)\}$ of the equations (\ref{hu7}).

If solutions $h_{u|v}(s),u=1,...,m,\{g_{j|v}(s)\}$ exist, from (\ref{hu7}), we have 
\begin{equation}\label{equ5}
\begin{cases}
[\sum_{r,s=1}^n \Lambda_{M_{0j}}^{r,s}(z_j)\frac{\partial}{\partial z_j^r}\wedge \frac{\partial}{\partial z_j^s}, \lambda_{j|v}(s)]=0\\
\lambda_{k|v}(s)-\lambda_{j|v}(s)+[\sum_{r,s=1}^n \Lambda_{M_{0j}}^{r,s}(z_j)\frac{\partial}{\partial z_j^r}\wedge \frac{\partial}{\partial z_j^s},\Gamma_{jk|v}(s)]=0\\
\Gamma_{jk|v}(s)-\Gamma_{ik|v}(s)+\Gamma_{ij|v}(s)=0
\end{cases}
\end{equation}

Conversely, 
\begin{lemma}\label{lemma10}
If $(\{\lambda_{j|v}(s)\},\{\Gamma_{jk|v}(s)\})$ satisfies $(\ref{equ5})$, then 
\begin{equation}\label{eq22}
 (\{\lambda_{j|v}(s)\},\{\Gamma_{jk|v}(s)\})=\sum_{u=1}^m h_{u|v}(s)(\{\Lambda_{uj}'\},\{\theta_{ujk}\})-\delta_{HP}(\{g_{j|v}(s)\})
\end{equation}
has solutions $h_{u|v}(s),u=1,...,m,\{g_{j|v}(s)\}$ when the Poisson Kodaira-Spencer map $\varphi_0:T_0(B)\to \mathbb{H}^1(M_0,\Theta_{M_0}^\bullet)$ is surjective.
\end{lemma}
\begin{proof}
Let $h_{u|v}(s)=\sum_{v_1+\cdots v_l=v} h_{uv_1\cdots v_l} s_1^{v_1}\cdots s_l^{v_l}$. Then by considering the coefficients of $s_1^{v_1}\cdots s_l^{v_l}$, $(\ref{eq22})$ can be written as
\begin{align*}
(\{\lambda_{jv_1\cdots v_l}\},\{\Gamma_{jkv_1,...,v_l}\})=\sum_{u=1}^m h_{uv_1\cdots v_l}(\{\Lambda_{uj}'\},\{\theta_{ujk}\})-\delta_{HP}(\{g_{jv_1\cdots v_l}\}).
\end{align*}
Thus it suffices to prove that any $1$-cocycle $(\{\lambda_j\},\{\Gamma_{jk}\})\in C^0(\mathcal{U},\wedge^2 \Theta_0)\oplus C^1(\mathcal{U},\Theta_0)$ such that $[\Lambda_0,\lambda_j]=0, \lambda_k-\lambda_j+[\Lambda_0,\Gamma_{jk}]=0,\Gamma_{jk}-\Gamma_{ik}+\Gamma_{ij}=0$ can be written in the form
\begin{align*}
(\{\lambda_{j}\},\{\Gamma_{jk}\})=\sum_{u=1}^m h_{u}(\{\Lambda_{uj}'\},\{\theta_{ujk}\})-\delta_{HP}(\{g_{j}\}),\,\,\,\,\,\text{for some}\,\,\, h_u\in \mathbb{C}, \{g_j\}\in C^0(\mathcal{U},\Theta_0)
\end{align*}
Let $(\eta,\gamma)\in \mathbb{H}^1(M_0,\Theta_{M_0}^\bullet)$ be the cohomology class of $(\{\lambda_j\},\{\Gamma_{jk}\})$. Since $\varphi_0:T_0(B)\to \mathbb{H}^1(M_0,\Theta_{M_0}^\bullet)$ is surjective, $(\eta,\gamma)$ is written in the form of a linear combination of the $(\Lambda'_u,\theta_u)$(= the cohomology class of $(\{\Lambda'_{uj}\},\{\theta_{ujk}\})$,$u=1,...,m$ in $(\ref{p1}),(\ref{p2})$) as
\begin{align*}
(\eta,\gamma)=\sum_{u=1}^m h_u(\Lambda_u',\theta_u),\,\,\,\,\, h_u\in \mathbb{C}
\end{align*}
So $\sum_{u=1}^m h_u(\{\Lambda'_{uj}\},\{\theta_{ujk}\})$ is cohomologous to $(\{\lambda_j\},\{\Gamma_{jk}\})$. Therefore there exists $\{g_j\}\in C^0(\mathcal{U},\Theta_0)$ such that $\delta_{HP}(\{g_j\})=\sum_{u=1}^m h_u(\{\Lambda'_{uj}\},\{\theta_{ujk}\})-(\{\lambda_j\},\{\Gamma_{jk}\}).$
\end{proof}

Next we will prove that

\begin{lemma}\label{lemma11}
$(\{\lambda_{j|v}(s)\},\{\Gamma_{jk|v}(s)\})$ satisfies $(\ref{equ5})$. 
\end{lemma}
\begin{proof}
First, we have $\Gamma_{jk|v}(s)-\Gamma_{ik|v}(s)+\Gamma_{ij|v}(s)=0$ (see \cite{Kod05} p.292). Second, we show that $[\Lambda_0, \lambda_{j|v}(s)]=0$. We note that for $\Pi\in \Gamma(U_j,\wedge^3 \Theta_{M_0})$, $\Pi=0$ if and only if $\Pi(z_j^a,z_j^b,z_j^c):=\Pi(dz_j^a\wedge dz_j^b\wedge dz_j^c)$ for any $a,b,c$. Then from $[\Lambda_{N_j},\Lambda_{N_j}]=0$, $[\Lambda_{M_j},\Lambda_{M_j}]=0$, and Lemma \ref{formula},
 {\small{\begin{align*}
&[\Lambda_0, \lambda_{j|v}](z_j^a,z_j^b,z_j^c)\\
&=\Lambda_0(\lambda_{j|v}(z_j^a,z_j^b), z_j^c)-\Lambda_0(\lambda_{j|v}(z_j^a,z_j^c), z_j^b)+\Lambda_0(\lambda_{j|v}(z_j^b,z_j^c), z_j^a)\\
&+\lambda_{j|v}(\Lambda_0(z_j^a,z_j^b),z_j^c)-\lambda_{j|v}(\Lambda_0(z_j^a,z_j^c),z_j^b)+\lambda_{j|v}(\Lambda_0(z_j^b,z_j^c),z_j^a)\\
&\equiv_v \Lambda_{N_j}(\lambda_{j|v}(z_j^a,z_j^b), g_j^{cv-1})-\Lambda_{N_j}(\lambda_{j|v}(z_j^a,z_j^c), g_j^{bv-1})+\Lambda_{N_j}(\lambda_{j|v}(z_j^b,z_j^c), g_j^{cv-1})\\
&+\lambda_{j|v}(\Lambda_{M_j}(z_j^a,z_j^b),z_j^c)-\lambda_{j|v}(\Lambda_{M_j}(z_j^a,z_j^c),z_j^b)+\lambda_{j|v}(\Lambda_{M_j}(z_j^b,z_j^c),z_j^a)\\
&\equiv_v -2\Lambda_{N_j}(\Lambda_{M_j}^{a,b}(g_j^{v-1},h^{v-1}),g_j^{cv-1})+\Lambda_{N_j}(\Lambda_{N_j}(g_j^{av-1},g_j^{bv-1}),g_j^{cv-1})\\
&+2\Lambda_{N_j}(\Lambda_{M_j}^{a,c}(g_j^{v-1},h^{v-1}),g_j^{bv-1})-\Lambda_{N_j}(\Lambda_{N_j}(g_j^{av-1},g_j^{cv-1}),g_j^{bv-1})\\
&-2\Lambda_{N_j}(\Lambda_{M_j}^{b,c}(g_j^{v-1},h^{v-1}),g_j^{av-1})+\Lambda_{N_j}(\Lambda_{N_j}(g_j^{bv-1},g_j^{cv-1}),g_j^{av-1})\\
&+2\lambda_{j|v}(\Lambda_{M_j}^{a,b}(z_j,s),z_j^c)-2\lambda_{j|v}(\Lambda_{M_j}^{a,c}(z_j,s),z_j^b)+2\lambda_{j|v}(\Lambda_{M_j}^{b,c}(z_j,s),z_j^a)\\
&\equiv_v -2\Lambda_{N_j}(\Lambda_{M_j}^{a,b}(g_j^{v-1},h^{v-1}),g_j^{cv-1})
+2\Lambda_{N_j}(\Lambda_{M_j}^{a,c}(g_j^{v-1},h^{v-1}),g_j^{bv-1})
-2\Lambda_{N_j}(\Lambda_{M_j}^{b,c}(g_j^{v-1},h^{v-1}),g_j^{av-1})\\
&+4\sum_{r=1}^n \left( -\Lambda_{M_j}^{r,c}(g_j^{v-1},h^{v-1})+\sum_{\alpha,\beta=1}^n \Lambda_{N_j}^{\alpha,\beta}(z_j,s)\frac{\partial g_j^{rv-1}}{\partial z_j^\alpha}\frac{\partial g_j^{cv-1}}{\partial z_j^\beta}         \right)\frac{\partial \Lambda_{M_j}^{a,b}}{\partial z_j^r}(g_j^{v-1},h^{v-1})\\
&-4\sum_{r=1}^n \left( -\Lambda_{M_j}^{r,b}(g_j^{v-1},h^{v-1})+\sum_{\alpha,\beta=1}^n \Lambda_{N_j}^{\alpha,\beta}(z_j,s)\frac{\partial g_j^{rv-1}}{\partial z_j^\alpha}\frac{\partial g_j^{bv-1}}{\partial z_j^\beta}         \right)\frac{\partial \Lambda_{M_j}^{a,c}}{\partial z_j^r}(g_j^{v-1},h^{v-1})\\
&+4\sum_{r=1}^n \left( -\Lambda_{M_j}^{r,a}(g_j^{v-1},h^{v-1})+\sum_{\alpha,\beta=1}^n \Lambda_{N_j}^{\alpha,\beta}(z_j,s)\frac{\partial g_j^{rv-1}}{\partial z_j^\alpha}\frac{\partial g_j^{av-1}}{\partial z_j^\beta}         \right)\frac{\partial \Lambda_{M_j}^{b,c}}{\partial z_j^r}(g_j^{v-1},h^{v-1})\\
&\equiv_v -2\Lambda_{N_j}(\Lambda_{M_j}^{a,b}(g_j^{v-1},h^{v-1}),g_j^{cv-1})
+2\Lambda_{N_j}(\Lambda_{M_j}^{a,c}(g_j^{v-1},h^{v-1}),g_j^{bv-1})
-2\Lambda_{N_j}(\Lambda_{M_j}^{b,c}(g_j^{v-1},h^{v-1}),g_j^{av-1})\\
&-4\sum_{r=1}^n \left(   \Lambda_{M_j}^{r,c}(g_j^{v-1},h^{v-1})  \frac{\partial \Lambda_{M_j}^{a,b}}{\partial z_j^r}(g_j^{v-1},h^{v-1})  - \Lambda_{M_j}^{r,b}(g_j^{v-1},h^{v-1})  \frac{\partial \Lambda_{M_j}^{a,c}}{\partial z_j^r}(g_j^{v-1},h^{v-1})  + \Lambda_{M_j}^{r,a}(g_j^{v-1},h^{v-1})  \frac{\partial \Lambda_{M_j}^{b,c}}{\partial z_j^r}(g_j^{v-1},h^{v-1})            \right)\\
&+4\sum_{\alpha,\beta=1}^n\Lambda_{N_j}^{\alpha,\beta}(z_j,s)\frac{\partial \Lambda_{M_j}^{a,b}(g_j^{v-1},h^{v-1})}{\partial z_j^\alpha}\frac{\partial g_j^{cv-1}}{\partial z_j^\beta}-4\sum_{\alpha,\beta=1}^n\Lambda_{N_j}^{\alpha,\beta}(z_j,s)\frac{\partial \Lambda_{M_j}^{a,c}(g_j^{v-1},h^{v-1})}{\partial z_j^\alpha}\frac{\partial g_j^{bv-1}}{\partial z_j^\beta}\\
&+4\sum_{\alpha,\beta=1}^n\Lambda_{N_j}^{\alpha,\beta}(z_j,s)\frac{\partial \Lambda_{M_j}^{b,c}(g_j^{v-1},h^{v-1})}{\partial z_j^\alpha}\frac{\partial g_j^{av-1}}{\partial z_j^\beta}\\
&=0
\end{align*}}}
Next we will show that 
\begin{align}\label{eq45}
\lambda_{k|v}(s)-\lambda_{j|v}(s)+[\sum_{r,s=1}^n \Lambda_{M_{0j}}^{r,s}(z_j)\frac{\partial}{\partial z_j^r}\wedge \frac{\partial}{\partial z_j^s},\Gamma_{jk|v}(s)]=0.
\end{align}
First we compute the third term of (\ref{eq45})
\begin{align}\label{equ46}
&[\sum_{r,s=1}^n \Lambda_{M_{0j}}^{r,s}(z_j)\frac{\partial}{\partial z_j^r}\wedge \frac{\partial}{\partial z_j^s},\Gamma_{jk|v}(s)]=\sum_{r,s,\beta=1}^n [\Lambda_{M_{0j}}^{r,s}(z_j)\frac{\partial}{\partial z_j^r}\wedge \frac{\partial}{\partial z_j^s},\Gamma_{jk|v}^{\beta}\frac{\partial}{\partial z_j^{\beta}}]\\
&=\sum_{r,s,\beta=1}^n \left(\Lambda_{M_{0j}}^{r,s}\frac{\partial \Gamma_{jk}^{\beta}}{\partial z_j^{r}}\frac{\partial}{\partial z_j^{\beta}}\wedge \frac{\partial}{\partial z_j^s}
-\Gamma_{jk|v}^{\beta} \frac{\partial \Lambda_{M_{0j}}^{r,s}}{\partial z_j^{\beta}}\frac{\partial}{\partial z_j^r}\wedge \frac{\partial}{\partial z_j^s}+\Lambda_{M_{0j}}^{r,s}\frac{\partial \Gamma_{jk|v}^{\beta}}{\partial z_j^s}\frac{\partial}{\partial z_j^{r}}\wedge \frac{\partial}{\partial z_j^{\beta}}\right)\notag\\
&=\sum_{r,s,\beta=1}^n \left( \Lambda_{M_{0j}}^{\beta,s}\frac{\partial \Gamma_{jk}^{r}}{\partial z_j^{\beta}}\frac{\partial}{\partial z_j^{r}}\wedge \frac{\partial}{\partial z_j^s}
-\Gamma_{jk|v}^{\beta} \frac{\partial \Lambda_{M_{0j}}^{r,s}}{\partial z_j^{\beta}}\frac{\partial}{\partial z_j^r}\wedge \frac{\partial}{\partial z_j^s}+\Lambda_{M_{0j}}^{r,\beta}\frac{\partial \Gamma_{jk|v}^{s}}{\partial z_j^{\beta}}\frac{\partial}{\partial z_j^{r}}\wedge \frac{\partial}{\partial z_j^s}\right)
\notag
\end{align}

We consider the first term $\lambda_{k|v}(s)$ of ($\ref{eq45}$). From $(\ref{aa35})$ and $(\ref{ii34})$, we have 
\begin{align}\label{eq47}
\lambda_{k|v}(s)&\equiv_v \sum_{p,q=1}^n\left(-\Lambda_{M_k}^{p,q}(g_k^{v-1}(z_k,s), h^{v-1}(s))+\sum_{a,b=1}^n \Lambda_{N_k}^{a,b}(z_k,s)\frac{\partial {g_k^p}^{v-1}}{\partial z_k^a}\frac{\partial {g_k^q}^{v-1}}{\partial z_k^b}\right)\frac{\partial}{\partial z_k^p}\wedge\frac{\partial}{\partial z_k^q}\\
&\equiv_v \sum_{r,s=1}^n\left(\sum_{p,q=1}^n\left(-\Lambda_{M_k}^{p,q}(g_k^{v-1}(z_k,s), h^{v-1}(s))+\sum_{a,b=1}^n \Lambda_{N_k}^{a,b}(z_k,s)\frac{\partial {g_k^p}^{v-1}}{\partial z_k^a}\frac{\partial {g_k^q}^{v-1}}{\partial z_k^b}\right)\frac{\partial b_{jk}^r}{\partial z_k^p} \frac{\partial b_{jk}^s}{\partial z_k^q}\right)\frac{\partial}{\partial z_j^r}\wedge\frac{\partial}{\partial z_j^s}\notag
\end{align}

We consider the second term $-\lambda_{j|v}(s)$ of (\ref{eq45}). We note that since $z_j=b_{jk}(z_k)=f_{jk}(z_k,0)$ from $(\ref{ii34})$, we have $\lambda_{j|v}(z_j,s)\equiv_v \lambda_{j|v}(f_{jk}(z_k,s),s)$. Then from $(\ref{aa35})$ and by induction hypothesis $(\ref{aa33})$, we have
{\small{\begin{align}\label{eq48}
&-\lambda_{j|v}(z_j,s)\equiv_v\sum_{r,s=1}^n\left(\Lambda_{M_j}^{r,s}(g_j^{v-1}(z_j,s),h^{v-1}(s))-\sum_{\alpha,\beta=1}^n \Lambda_{N_j}^{\alpha,\beta}(z_j,s)\frac{\partial {g_j^{r}}^{v-1}}{\partial z_j^{\alpha}}\frac{\partial {g_j^{s}}^{v-1}}{\partial z_j^{\beta}}\right)\frac{\partial}{\partial z_j^{r}}\wedge \frac{\partial}{\partial z_j^s}\\
&\equiv_v \sum_{r,s=1}^n\left(\Lambda_{M_j}^{r,s}(g_j^{v-1}(f_{jk}(z_k,s),s),h^{v-1}(s))-\sum_{\alpha,\beta=1}^n \Lambda_{N_j}^{\alpha,\beta}(f_{jk}(z_k,s),s)\frac{\partial {g_j^{r}}^{v-1}}{\partial z_j^{\alpha}}(f_{jk}(z_k,s),s)\frac{\partial {g_j^{s}}^{v-1}}{\partial z_j^{\beta}}(f_{jk}(z_k,s),s)\right) \frac{\partial}{\partial z_j^{r}}\wedge \frac{\partial}{\partial z_j^s}\notag
\end{align}}}
We consider the first term of $(\ref{eq48})$. From $(\ref{tt002})$ and $(\ref{vv10})$, we have
\begin{align}\label{eq49}
&\Lambda_{M_j}^{r,s}(g_j^{v-1}(f_{jk}(z_k,s),s),h^{v-1}(s))\equiv_v \Lambda_{M_j}^{r,s}(g_{jk}(g_k^{v-1}(z_k,s),h^{v-1}(s))+\Gamma_{jk|v}(z_j,s),h^{v-1}(s))\\
&\equiv_v  \Lambda_{M_j}^{r,s}(g_{jk}(g_k^{v-1}(z_k,s),h^{v-1}(s)),h^{v-1}(s))+\sum_{\beta=1}^n \frac{\partial \Lambda_{M_{0j}}^{r,s}}{\partial z_j^{\beta}}\Gamma_{jk|v}^{\beta}(z_j,s)\notag\\
&=\sum_{p,q=1}^n \Lambda_{M_k}^{p,q}(g_k^{v-1}(z_k,s),h^{v-1}(s))\frac{\partial g_{jk}^r}{\partial \xi_k^p}(g_k^{v-1}(z_k,s),h^{v-1}(s))\frac{\partial g_{jk}^s}{\partial \xi_k^q}(g_k^{v-1}(z_k,s),h^{v-1}(s))+\sum_{\beta=1}^n \frac{\partial \Lambda_{M_{0j}}^{r,s}}{\partial z_j^{\beta}}\Gamma_{jk|v}^{\beta}(z_j,s)\notag
\end{align}
On the other hand, we consider the second term of $(\ref{eq48})$. We note that from $(\ref{vv11})$, $(\ref{tt002})$ and $(\ref{vv3})$, we have
\begin{align}\label{eq50}
&\sum_{\alpha,\beta=1}^n \Lambda_{N_j}^{\alpha,\beta}(f_{jk}(z_k,s),s)\frac{\partial {g_j^{r}}^{v-1}}{\partial z_j^{\alpha}}(f_{jk}(z_k,s),s)\frac{\partial {g_j^{s}}^{v-1}}{\partial z_j^{\beta}}(f_{jk}(z_k,s),s)\\
& =\sum_{\alpha,\beta,a,b=1}^n \Lambda_{N_k}^{a,b}(z_k,s)\frac{\partial f_{jk}^\alpha(z_k,s)}{\partial z_k^a}\frac{\partial f_{jk}^\beta(z_k,s)}{\partial z_k^b }\frac{\partial {g_j^{r}}^{v-1}}{\partial z_j^{\alpha}}(f_{jk}(z_k,s),s)\frac{\partial {g_j^{s}}^{v-1}}{\partial z_j^{\beta}}(f_{jk}(z_k,s),s)\notag\\
&=\sum_{a,b=1}^n \Lambda_{N_k}^{a,b}(z_k,s)\frac{\partial (g_j^{rv-1}(f_{jk}(z_k,s),s))}{\partial z_k^a}\frac{\partial (g_j^{sv-1}(f_{jk}(z_k,s),s))}{\partial z_k^b}\notag\\
&\equiv_v \sum_{a,b=1}^n\Lambda_{N_k}^{a,b}(z_k,s)\frac{\partial(g_{jk}^r(g_k^{v-1}(z_k,s),h^{v-1}(s))+\Gamma_{jk|v}^r(z_j,s))}{\partial z_k^a}\frac{\partial (g_{jk}^s(g_k^{v-1}(z_k,s),h^{v-1}(s))+\Gamma_{jk|v}^s(z_j,s))}{\partial z_k^b}\notag\\
&\equiv_v\sum_{a,b,p,q=1}^n \Lambda_{N_k}^{a,b}(z_k,s)\frac{\partial g_{jk}^r}{\partial \xi_k^p}\frac{\partial g_k^{pv-1}}{\partial z_k^a}\frac{\partial g_{jk}^s}{\partial \xi_k^q}\frac{\partial g_{k}^{qv-1}}{\partial z_k^b}+\sum_{a,b=1}^n \Lambda_{M_{0k}}^{a,b}(z_k)\frac{\partial \Gamma_{jk|v}^r}{\partial z_k^a}\frac{\partial z_j^s}{\partial z_k^b}+\sum_{a,b=1}^n \Lambda_{M_{0k}}^{a,b}(z_k)\frac{\partial z_j^r}{\partial z_k^a}\frac{\partial \Gamma_{jk|v}^s}{\partial z_k^b}\notag\\
&\equiv_v\sum_{a,b,p,q=1}^n \Lambda_{N_k}^{a,b}(z_k,s)\frac{\partial g_{jk}^r}{\partial \xi_k^p}\frac{\partial g_k^{pv-1}}{\partial z_k^a}\frac{\partial g_{jk}^s}{\partial \xi_k^q}\frac{\partial g_{k}^{qv-1}}{\partial z_k^b}+\sum_{a,b,\beta=1}^n \Lambda_{M_{0k}}^{a,b}(z_k)\frac{\partial \Gamma_{jk|v}^r}{\partial z_j^\beta}\frac{\partial z_j^\beta}{\partial z_k^a}\frac{\partial z_j^s}{\partial z_k^b}+\sum_{a,b,\beta=1}^n\Lambda_{M_{0k}}^{a,b}(z_k)\frac{\partial z_j^r}{\partial z_k^a}\frac{\partial \Gamma_{jk|v}^s}{\partial z_j^\beta}\frac{\partial z_j^{\beta}}{\partial z_k^b}\notag\\
&\equiv_v\sum_{a,b,p,q=1}^n \Lambda_{N_k}^{a,b}(z_k,s)\frac{\partial g_{jk}^r}{\partial \xi_k^p}\frac{\partial g_k^{pv-1}}{\partial z_k^a}\frac{\partial g_{jk}^s}{\partial \xi_k^q}\frac{\partial g_{k}^{qv-1}}{\partial z_k^b}+\sum_{\beta=1}^n \Lambda_{M_{0j}}^{\beta,s}(z_j)\frac{\partial \Gamma_{jk|v}^r}{\partial z_j^\beta}+\sum_{\beta=1}^n \Lambda_{M_{0j}}^{r,\beta}(z_j)\frac{\partial \Gamma_{jk|v}^s}{\partial z_j^\beta}\notag
\end{align}
where we mean $\frac{\partial g_{jk}^r}{\partial \xi_k^p}$ and $\frac{\partial g_{jk}^s}{\partial \xi_k^q}$ by 
\begin{align}\label{tt89}
\frac{\partial g_{jk}^r}{\partial \xi_k^p}:=\frac{\partial g_{jk}^r}{\partial \xi_k^p}(g_k^{v-1}(z_k,s),h^{v-1}(s)),\,\,\,\,\,\frac{\partial g_{jk}^s}{\partial \xi_k^q}:=\frac{\partial g_{jk}^s}{\partial \xi_k^q}(g_k^{v-1}(z_k,s),h^{v-1}(s))
\end{align}

Hence from (\ref{eq48}),(\ref{eq49}), and (\ref{eq50}), we have
\begin{align}\label{eq51}
&-\lambda_{j|v}(s)\equiv_v \sum_{r,s=1}^n\left(\sum_{p,q=1}^n(\Lambda_{M_k}^{p,q}(g_k^{v-1}(z_k,s), h^{v-1}(s))\frac{\partial g_{jk}^r}{\partial \xi_k^p}\frac{\partial g_{jk}^s}{\partial \xi_k^q}+\sum_{\beta=1}^n \frac{\partial \Lambda_{M_{0j}}^{r,s}}{\partial z_j^{\beta}}\Gamma_{jk|v}^{\beta}(z_j,s)\right)\frac{\partial}{\partial z_j^r}\wedge\frac{\partial}{\partial z_j^s}\\
&+\sum_{r,s=1}^n\left(-\sum_{a,b,p,q=1}^n \Lambda_{N_k}^{a,b}(z_k,s)\frac{\partial g_{jk}^r}{\partial \xi_k^p}\frac{\partial g_k^{pv-1}}{\partial z_k^a}\frac{\partial g_{jk}^s}{\partial \xi_k^q}\frac{\partial g_{k}^{qv-1}}{\partial z_k^b}-\sum_{\beta=1}^n \Lambda_{M_{0j}}^{\beta,s}(z_j)\frac{\partial \Gamma_{jk|v}^r}{\partial z_j^\beta}-\sum_{\beta=1}^n\Lambda_{M_{0j}}^{r,\beta}(z_j)\frac{\partial \Gamma_{jk|v}^s}{\partial z_j^\beta}\right)\frac{\partial}{\partial z_j^r}\wedge\frac{\partial}{\partial z_j^s}\notag
\end{align}
From (\ref{equ46}),(\ref{eq47}) and $(\ref{eq51})$, to show (\ref{eq45}) is equivalent to show that for each $r,s$,
\begin{align}\label{eq34}
&\sum_{p,q=1}^n\left(-\Lambda_{M_k}^{p,q}(g_k^{v-1}(z_k,s), h^{v-1}(s))+\sum_{a,b=1}^n \Lambda_{N_k}^{a,b}(z_k,s)\frac{\partial {g_k^p}^{v-1}}{\partial z_k^a}\frac{\partial {g_k^q}^{v-1}}{\partial z_k^b}\right) \frac{\partial b_{jk}^r}{\partial z_k^p} \frac{\partial b_{jk}^s}{\partial z_k^q}\\
&+\sum_{p,q=1}^n \Lambda_{M_k}^{p,q}(g_k^{v-1}(z_k,s),h^{v-1}(s))\frac{\partial g_{jk}^r}{\partial \xi_k^p}\frac{\partial g_{jk}^s}{\partial \xi_k^q}-\sum_{a,b,p,q=1}^n \Lambda_{N_k}^{a,b}(z_k,s)\frac{\partial g_{jk}^r}{\partial \xi_k^p}\frac{\partial g_k^{pv-1}}{\partial z_k^a}\frac{\partial g_{jk}^s}{\partial \xi_k^q}\frac{\partial g_{k}^{qv-1}}{\partial z_k^b}\equiv_v 0
\notag
\end{align}
$(\ref{eq34})$ is equivalent to
\begin{align}\label{poi}
\sum_{p,q=1}^n\left(-\Lambda_{M_k}^{p,q}(g_k^{v-1}(z_k,s), h^{v-1}(s))+\sum_{a,b=1}^n \Lambda_{N_k}^{a,b}(z_k,s)\frac{\partial {g_k^p}^{v-1}}{\partial z_k^a}\frac{\partial {g_k^q}^{v-1}}{\partial z_k^b}\right)\left( \frac{\partial b_{jk}^r}{\partial z_k^p} \frac{\partial b_{jk}^s}{\partial z_k^q}-\frac{\partial g_{jk}^r}{\partial \xi_k^p}\frac{\partial g_{jk}^s}{\partial \xi_k^q}\right)\equiv_v 0
\end{align}
By induction hypothesis $(\ref{aa33})$, we have $\Lambda_{M_k}^{p,q}(g_k^{v-1}(z_k,s),h^v(s))\equiv_{v-1}\sum_{a,b=1}^n \Lambda_{N_k}^{a,b}(z_k,s)\frac{\partial {g_k^p}^{v-1}}{\partial z_k^{a}}\frac{\partial {g_k^q}^{v-1}}{\partial z_k^{b}}$, and we have $\left( \frac{\partial b_{jk}^r}{\partial z_k^p} \frac{\partial b_{jk}^s}{\partial z_k^q}-\frac{\partial g_{jk}^r}{\partial \xi_k^p}\frac{\partial g_{jk}^s}{\partial \xi_k^q}\right)\equiv_0 0$ since from $(\ref{tt89})$, we have $\frac{\partial g_{jk}^r}{\partial \xi_k^p}(g_k^{v-1}(z_k,0), h^{v-1}(0))=\frac{\partial g_{jk}^r}{\partial \xi_k^p}(z_k,0)=\frac{\partial b_{jk}^r}{\partial z_k^p}$ and similarly for $\frac{\partial g_{jk}^s}{\partial \xi_k^q}$. Hence we have (\ref{poi}). This completes Lemma \ref{lemma11}.
\end{proof}

\subsection{Proof of Convergence}\

By Lemma \ref{lemma10} and Lemma \ref{lemma11}, we can find $h_{u|v},u=1,...,m$, and $\{g_{j|v}(z_j,s)\}$ inductively on $v$ such that $h^v(s)=h^{v-1}(s)+h_{u|v}(s)$ and $g_j^v(z_j,s)=g_j^{v-1}(z_j,s)+g_{j|v}(z_j,s)$ satisfy $(\ref{aa11})_v$ and $(\ref{aa12})_v$ so that we have formal power series  $h(s)$ and $g_j(z_j,s)$ satisfying (\ref{aa90}) and (\ref{aa91}). In this subsection, we will prove that we can choose appropriate solutions $h_{u|v}(s)$ and $\{g_{j|v}(z_j,s)\}$ in each inductive step so that $h(s)$ and $g_j(z_j,s)$ converge absolutely in $|s|<\epsilon$ if $\epsilon>0$ is sufficiently small. As in \cite{Kod05} p.294-302, our approach is to estimate $\Gamma_{jk|v}(z_j,s)$, $\lambda_{j|v}(s)$ and use Lemma \ref{lemma3} below concerning the ``magnitude" of the solutions $h_{u|v}(s),u=1...,m,\{g_{j|v}(z_j,s)\}$ of the equation $(\ref{hu7})$.

\begin{definition}
Let $\mathcal{U}:=\{U_j\}$ be a finite open covering of $M_0$ in $(\ref{covering})$. We may assume that $U_j=\{z_j\in\mathbb{C}^n|z_j|<1\}$ and $M_0=\bigcup_j U_j^\delta$, where $U_j^\delta=\{z_j\in U_j||z_j|<1-\delta\}$ for a sufficiently small number $\delta>0$. We denote a $1$-cocycle $(\{\lambda_j\},\{\Gamma_{jk}\})\in C^{0}(\mathcal{U},\wedge^2 \Theta_{M_0})\oplus C^1(\mathcal{U},\Theta_{M_0})$ by $(\lambda,\Gamma)$ in the \u{C}ech resolution of the complex of sheaves $(\ref{complex})$, and define its norm by
\begin{align}\label{yy3}
|(\lambda,\Gamma)|:=\max_j \sup_{z_j\in U_j^\delta} |\lambda_j(z_j)|+\max_{j,k}\sup_{z_j\in U_j\cap U_k}|\Gamma_{jk}(z_j)|
\end{align}
\end{definition}

\begin{remark}
We explain the meaning of $|\lambda_j(z_j)|$ and $|\Gamma_{jk}(z_j)|$ in $(\ref{yy3})$. We regard holomorphic vector field $\Gamma_{jk}(z_j)=\sum_{\alpha=1}^n \Gamma_{jk}^{\alpha}(z_j) \frac{\partial}{\partial z_j^{\alpha}}$ as a vector-valued holomorphic function $(\Gamma_{jk}^1(z_j),...,\Gamma_{jk}^n(z_j))$ and regard holomorphic bivector field $\lambda_j(z_j)=\sum_{r,s=1}^n \lambda_j^{r,s}(z_j)\frac{\partial}{\partial z_j^r}\wedge \frac{\partial}{\partial z_j^s}$ as a holomorphic vector valued function $(\lambda_j^{1,1}(z_j),\cdots,\lambda_j^{r,s}(z_j),\cdots, \lambda_j^{n,n}(z_j))$. For $z_j\in U_j\cap U_k$, we define $|\Gamma_{jk}(z_j)|:=\max_\alpha |\Gamma_{jk}^\alpha(z_j)|$. On the other hand, for $ z_j\in U_j^\delta$ we define $|\lambda_j(z_j)|:=\max_{r,s} |\lambda_j^{r,s}(z_j)|$.
\end{remark}

\begin{remark}
Since each $U_j=\{z_j\in \mathbb{C}^n||z_j|<1\}$ in $(\ref{covering})$ is a coordinate polydisk, we may assume that the coordinate function $z_j$ is defined on a domain of $M_0$ containing $\overline{U}_{j}$ $($the closure of $U_j$$)$. Hence there exists a constant $L_1>0$ such that for all $\alpha,\beta=1,...,n$, and for all $U_k\cap U_j\ne \emptyset$,
\begin{align}\label{hy1}
\left|\frac{\partial z_j^{\alpha}}{\partial z_k^{\beta}}(z_k)\right|=\left|\frac{\partial b_{jk}^\alpha}{\partial z_k^\beta}(z_k) \right|< L_1,\,\,\,\,\,z_k\in U_k\cap U_j,
\end{align}
and there exist constants $C,C'>0$ such that for all $r,s,\beta=1,...,n$ and for all $U_j$,
\begin{align}\label{ii99}
\left|\Lambda_{M_{0j}}^{r,s}(z_j)\right| < C,\,\,\,\,\,\, \left|\frac{\partial \Lambda_{M_{0j}}^{r,s}}{\partial z_j^{\beta}}(z_j)\right|<C',\,\,\,\,\,z_j\in U_j.
\end{align}
 We define the norm of the matrix $B_{jk}(z_j):=\left(\frac{\partial z_j^{\alpha}}{\partial z_k^{\beta}}(z_k)\right)_{\alpha,\beta=1,...,n}$ by $|B_{jk}(z_k)|=\max_{\alpha}\sum_{\beta}\left|\frac{\partial z_j^{\alpha}}{\partial z_k^{\beta}}(z_k)\right|$. Then there exists a constant $K_1>1$ such that for all $U_j\cap U_k\ne \emptyset$, 
\begin{align}\label{yt2}
|B_{jk}(z_k)|<K_1 \,\,\,\,\,\text{ $z_k\in U_k\cap U_j$}
\end{align}
 Since $\theta_{ujk}^{\alpha}(z_j)=\left(\frac{\partial g_{jk}^{\alpha}(z_k,t)}{\partial t_u}\right)_{t=0}$ in $(\ref{p1})$ are bounded on $U_j\cap U_k\ne \emptyset$, there exists a constant $K_2$ such that 
\begin{align}\label{yt1}
|\theta_{ujk}(z_j)|=|\sum_{\alpha=1}^n \theta_{ujk}^{\alpha}(z_j)\frac{\partial}{\partial z_j^{\alpha}}|:=\max_{\alpha}|\theta_{ujk}^{\alpha}(z_j)|<K_2
\end{align}

Since $\Lambda_{uj}'^{r,s}(z_j)=\left(\frac{\partial \Lambda_{M_j}^{r,s}(z_j,t)}{\partial t_u}\right)_{t=0}$ in $(\ref{p2})$ are bounded on $U_j$, there exists a constant $L_2$ such that 
\begin{align}\label{yt3}
|\Lambda_{uj}'(z_j)|=|\sum_{r,s=1}^n \Lambda_{uj}'^{r,s}(z_j)\frac{\partial}{\partial z_j^r}\wedge \frac{\partial}{\partial z_j^s}|:=\max_{r,s} |\Lambda_{uj}'^{r,s}(z_j)|<L_2
\end{align}
\end{remark}

\begin{lemma}[compare \cite{Kod05} Lemma 6.2 p.295]\label{lemma3}
There exist solutions $h_u,u=1,...,m$, and $\{g_j(z_j)=\sum_{\alpha=1}^n g_j^\alpha(z_j)\frac{\partial}{\partial z_j^\alpha}\}$ of the equation 
\begin{align}\label{p55}
(\{\lambda_j\},\{\Gamma_{jk}\})=\sum_{u=1}^m h_u(\{\Lambda'_{uj}\}, \{\theta_{ujk}\})-\delta_{HP}\{g_j(z_j)\}
\end{align}
which satisfy 
\begin{align*}
|h_u|\leq M|(\lambda,\Gamma)|,\,\,\,\,\,\,\, |g_j(z_j)|:=\max_\alpha |g_j^{\alpha}(z_j)| \leq M|(\lambda,\Gamma)| \,\,\,\,\text{for $z_j\in U_j$}
\end{align*}
where $M$ is a constant independent of a $1$-cocycle $(\lambda,\Gamma)=(\{\lambda_j\},\{\Gamma_{jk}\})$. 

\end{lemma}

\begin{proof} 
The proof is similar to \cite{Kod05} Lemma 6.2 p.295 to which we refer for the detail. For a $1$-cocycle $(\lambda,\Gamma)=(\{\lambda_j\},\{\Gamma_{jk}\})$ with $|(\lambda,\Gamma)|<\infty$, we define $\iota(\lambda,\Gamma)$ by
\begin{align*}
\iota(\lambda,\Gamma)=\inf \max_{u,j}\{|h_u|,\sup_{z_j\in U_j}|g_j(z_j)|\},
\end{align*}
where $\inf$ is taken with respect to all the solutions $h_u,u=1,...,m$, and $g_j(z_j)$ of (\ref{p55}). We will show that there exists a constant $M$ such that for all $1$-cocycles $(\lambda,\Gamma)\in \mathcal{C}^0(\mathcal{U},\wedge^2 \Theta_{M_0})\oplus \mathcal{C}^1(\mathcal{U},\Theta_{M_0})$, we have 
\begin{align*}
\iota(\lambda,\Gamma)\leq M|(\lambda,\Gamma)|
\end{align*}
Suppose there is no such constant $M$. Then we can find a sequence of $1$-cocycles $(\lambda^{(v)},\Gamma^{(v)})=(\{\lambda_j^{(v)}\},\{\Gamma_{jk}^{(v)}\})\in \mathcal{C}^0(\mathcal{U},\wedge^2 \Theta_{M_0})\oplus \mathcal{C}^1(\mathcal{U},\Theta_{M_0}),v=1,2,,3\cdots$, and their solutions $g_j^{(v)}(z_j), h_u^{(v)},u=1,...,m$ such that
\begin{align}\label{yt7}
\iota(\lambda^{(v)},\Gamma^{(v)})=1,\,\,\,\,\,\,\,\,\, |(\lambda^{(v)},\Gamma^{(v)})|<\frac{1}{v},
\end{align}
and the sequence $\{h_u^{(v)}\},u=1,...,m$ converge and $\{g_j^{(v)}(z_j)\}$ converges uniformly on $U_j$.

Put $h_u=\lim_{v\to \infty} h_u^{(v)}$, and $g_j(z_j)=\lim_{v\to \infty}g_j^{(v)}(z_j)$ and note that 
\begin{align*}
\Gamma_{jk}^{(v)}(z_j)&=\sum_{u=1}^m h_u^{(v)}\theta_{ujk}(z_j)+B_{jk}(z_k)g_k^{(v)}(z_k)-g_j^{(v)}(z_j),\\
\lambda_j^{(v)}(z_j)&=\sum_{u=1}^mh_u^{(v)}\Lambda_{uj}'(z_j)-[\sum_{r,s=1}^n \Lambda_{M_{0j}}^{r,s}(z_j)\frac{\partial}{\partial z_j^r}\wedge \frac{\partial}{\partial z_j^s},\sum_{\alpha=1}^n g_j^{\alpha(v)}(z_j)\frac{\partial}{\partial z_j^{\alpha}}]
\end{align*}
Since $|\Gamma_{jk}^{(v)}(z_j)|\leq |(\lambda^{(v)},\Gamma^{(v)})|\to 0$ for $z_j\in U_j\cap U_k$, and $|\lambda_j^{(v)}(z_j)|\leq |(\lambda^{(v)},\Gamma^{(v)})|\to 0$ for $z_j\in U_j^\delta$ as $v\to \infty$, we have
\begin{align*}
0&=\sum_{u=1}^m h_u\theta_{ujk}(z_j)+B_{jk}(z_k)g_k(z_k)-g_j(z_j),\,\,\,\,\,\text{$z_j\in U_j\cap U_k$}\\
0&=\sum_{u=1}^mh_u\Lambda_{uj}'(z_j)-[\sum_{r,s=1}^n \Lambda_{M_{0j}}^{r,s}(z_j)\frac{\partial}{\partial z_j^r}\wedge \frac{\partial}{\partial z_j^s},\sum_{\alpha=1}^n g_j^{\alpha}(z_j)\frac{\partial}{\partial z_j^{\alpha}}],\,\,\,\,\,\text{$z_j\in U_j^\delta$}
\end{align*}
By identity theorem, we have
\begin{align*}
0=\sum_{u=1}^mh_u\Lambda_{uj}'(z_j)-[\sum_{r,s=1}^n \Lambda_{M_{0j}}^{r,s}(z_j)\frac{\partial}{\partial z_j^r}\wedge \frac{\partial}{\partial z_j^s},\sum_{\alpha=1}^n g_j^{\alpha}(z_j)\frac{\partial}{\partial z_j^{\alpha}}],\,\,\,\,\,\text{$z_j\in U_j$}
\end{align*}

By putting $\tilde{h}_u^{(v)}=h_u^{(v)}-h_u$, and $\tilde{g}_j^{(v)}(z_j)=g_j^{(v)}(z_j)-g_j(z_j)$, we obtain
\begin{align*}
\Gamma_{jk}^{(v)}(z_j)&=\sum_{u=1}^m \tilde{h}_u^{(v)}\theta_{ujk}(z_j)+B_{jk}(z_k)\tilde{g}_k^{(v)}(z_k)-\tilde{g}_j^{(v)}(z_j),\\
\lambda_j^{(v)}(z_j)&=\sum_{u=1}^m\tilde{h}_u^{(v)}\Lambda_{uj}'(z_j)-[\sum_{r,s=1}^n \Lambda_{M_{0j}}^{r,s}(z_j)\frac{\partial}{\partial z_j^r}\wedge \frac{\partial}{\partial z_j^s},\sum_{\alpha=1}^n \tilde{g}_j^{\alpha(v)}(z_j)\frac{\partial}{\partial z_j^{\alpha}}]
\end{align*}
Hence $\tilde{h}_u^{(v)},u=1,...,m$, and $\{\tilde{g}_j^{(v)}(z_j)\}$ satisfy the equation (\ref{p55}) for $(\lambda,\Gamma)=(\{\lambda^{(v)}\},\{\Gamma^{(v)}\})$. This is a contradiction to $\iota(\lambda,\Gamma)=1$ ((\ref{yt7})) since we have $\tilde{h}_u^{(v)}\to 0$ and $\sup_{z_j\in U_j}|\tilde{g}_j^{(v)}(z_j)|\to 0$.

\end{proof}

Next we will prove that we can choose appropriate solutions $h_{u|v}(s),u=1,...,m$ and $\{g_{j|v}(z_j,s)\}$ in each inductive step by estimating $\Gamma_{jk|v}(z_j,s),\lambda_{j|k}(z_j,s)$ and using Lemma \ref{lemma3} so that the formal power series $h(s)$ and $g_j(z_j,s)$ converge absolutely in $|s|<\epsilon$ if $\epsilon>0$ is sufficiently small. Before the proof, we remark the following.
\begin{remark}\label{remark123}\
\begin{enumerate}
\item For two power series of $s_1,...,s_l$, 
\begin{align*}
P(s)&=\sum_{v_1,...,v_l=0}^\infty P_{v_1,...,v_l}s_1^{v_1}\cdots s_l^{v_l},\,\,\,\,\, P_{v_1,...,v_l}\in \mathbb{C}^n,\\
a(s)&=\sum_{v_1,...,v_l=0}^{\infty} a_{v_1,...,v_l}s_1^{v_1}\cdots s_l^{v_l},\,\,\,\,\,,a_{v_1,...,v_l}\geq 0,
\end{align*}
we write $P(s)\ll a(s)$ if $|P_{v_1,...,v_l}|\leq a_{v_1,...,v_l},\,\,\,\,\,v_1,...,v_l=0,1,2,...$.
\item For a power series $P(s)$, we denote by $[P(s)]_v$ the term of homogeneous part of degree $v$ with respect to $s$.
\item For $A(s)=\frac{b}{16c}\sum_{v=1}^{\infty}\frac{c^v(s_1+\cdots s_l)^v}{v^2},b>0,c>0$, we have $A(s)^v\ll \left(\frac{b}{c}\right)^{v-1}A(s),v=2,3,...$
\item Recall that for each $U_j=\{z_j\in \mathbb{C}^n||z_j|<1\}$, we set $U_j^\delta=\{z_j\in U_j||z_j|<1-\delta\}$ for a given $\delta$. Then $M_0=\bigcup_j U_j^\delta$ for a sufficiently small $\delta$.
\end{enumerate}
\end{remark}

To prove the convergence of $h(s)$ and $g_j(z_j,s)$, we will show the estimates $h(s)\ll A(s),\,\,\,g_j(z_j,s)-z_j\ll A(s)$ for suitable constants $b$ and $c$ in Remark \ref{remark123} (3), equivalently
\begin{align}\label{induction}
h^v(s)\ll A(s),\,\,\,\,\, g_j^v(z_j,s)-z_j\ll A(s)
\end{align}
for $v=1,2,3,...$. We will prove this by induction on $v=1,2,3,...$. For $v=1$, since the linear term of $A(s)$ is $\frac{b}{16}(s_1+\cdots +s_l)$, the estimate holds if $b$ is sufficiently large. Let $v\geq 2$ and assume that the induction holds for  $v-1$. In other words,
\begin{align}\label{induction2}
h^{v-1}(s)\ll A(s),\,\,\,\,\, g_j^{v-1}(z_j,s)-z_j\ll A(s)
\end{align}
 We will prove that (\ref{induction}) holds. For this, we estimate $\Gamma_{jk|v}(z_j,s)$ and $\lambda_{j|v}(z_j,s)$. For the estimation of $\Gamma_{jk|v}(z_j,s)$, we briefly summarize Kodaira's estimation presented in \cite{Kod05} p.298-302 in the following: since  $f_{jk}(z_k,s)=b_{jk}(z_k)+\sum_{v=1}^\infty f_{jk|v}(z_j,s)$ are given vector-valued holomorphic functions, we  may assume that
\begin{align*}
f_{jk}(z_k,s)-b_{jk}(z_k)\ll A_0(s),\,\,\,\,\,\, A_0(s)=\frac{b_0}{16c_0}\sum_{v=1}^\infty \frac{c_0^v(s_1+\cdots+s_l)^v}{v^2}
\end{align*}
holds for $z_k\in U_k\cap U_j$ with $b_0>0$ and $c_0>0$ such that $\frac{b_0}{c_0\delta}<\frac{1}{2}$, where $\delta$ from Remark \ref{remark123} (4). If we take $b$ and $c$ such that $b>b_0,c>c_0,\frac{ba_0(m+n)}{c}<\frac{1}{2}$, we can estimate
\begin{align}\label{yt6}
\Gamma_{jk|v}(z_j,s)\ll 2K_1K^*A(s),\,\,\,\,\, z_j\in U_j\cap U_k.
\end{align}
where $K^*=\frac{2^{n+1}b_0}{c\delta}+\frac{b_0}{b}+\frac{2ba_0^2(m+n)^2}{c}$ and $K_1$ from (\ref{yt2}) (for the detail, see page \cite{Kod05} 298-302).

Next we estimate $\lambda_{j|v}(z_j,s)$ (see (\ref{p5})). To estimate it, we estimate $\lambda_{j|v}^{r,s}(z_j,s)$ for each pair $(r,s)$ where $r,s=1,...,n$. We note that from (\ref{aa35}), we have
\begin{align}\label{jj11}
\lambda_{j|v}^{r,s} (z_j,s) =[ -\Lambda_{M_j}^{r,s}(g_j^{v-1}(z_j,s),h^{v-1}(s))]_v+[\sum_{p,q=1}^n\Lambda_{N_j}^{p,q}(z_j,s)\frac{\partial {g_j^r}^{v-1}}{\partial z_j^{p}}\frac{\partial {g_j^s}^{v-1}}{\partial z_j^{q}}]_v
\end{align}
First we estimate $[\Lambda_{M_j}^{r,s}(g_j^{v-1}(z_j,s),h^{v-1}(s))]_v$ in $(\ref{jj11})$. We expand $\Lambda_{M_{j}}^{r,s}(z_j+\xi,t)$ into power series in $\xi_1,...,\xi_n,t_1,...,t_m$, and let $L(\xi,t)$ be its linear term. Since $\Lambda_{M_j}^{r,s}(z_j,0)=\Lambda_{M_{0j}}^{r,s}(z_j)$ from $(\ref{ii34})$, we may assume that for all the pairs $(r,s)$,
\begin{align*}
\Lambda_{M_j}^{r,s}(z_j+\xi,t)-\Lambda_{M_{0j}}^{r,s}(z_j)-L(\xi,t)\ll \sum_{\mu=2}^{\infty} d_0^\mu(\xi_1+\cdots+\xi_n+t_1+\cdots +t_m)^\mu\,\,\,\,\,\text{for some constant $d_0>0$}
\end{align*}
Set $\xi=g_j^{v-1}(z_j,s)-z_j$, and $t=h^{v-1}(s)$. Since $\xi\ll A(s)$ and $t\ll A(s)$ by induction hypothesis (\ref{induction2}), we have from Remark \ref{remark123} (3)
\begin{align*}
&\Lambda_{M_j}^{r,s}(g_j^{v-1}(z_j,s),h^{v-1}(s))-\Lambda_{M_{0j}}^{r,s}(z_j)-L(g_j^{v-1}(z_j,s)-z_j,h^{v-1}(s))\\
&\ll \sum_{\mu=2}^\infty {d_0}^\mu(n+m)^\mu A(s)^\mu\ll \sum_{\mu=2}^\infty d_0^{\mu}(m+n)^\mu\left(\frac{b}{c}\right)^{\mu-1}A(s)= \frac{bd_0^2(m+n)^2}{c}\sum_{\mu=2}^\infty\left(\frac{bd_0(m+n)}{c}\right)^{\mu-2} A(s)
\end{align*}
Hence we have 
\begin{align*}
[\Lambda_{M_j}^{r,s}(g_j^{v-1}(z_j,s),h^{v-1}(s))]_v\ll \frac{bd_0^2(m+n)^2}{c}\sum_{\mu=0}^\infty\left(\frac{bd_0(m+n)}{c}\right)^\mu A(s)
\end{align*}
Choose a constant $c$ such that $\frac{bd_0(m+n)}{c}<\frac{1}{2}$. Then we have
\begin{align}\label{jj16}
[\Lambda_{M_j}^{r,s}(g_j^{v-1}(z_j,s),h^{v-1}(s))]_v\ll\frac{2bd_0^2(m+n)^2}{c}A(s),\,\,\,\,\,\,\,\, z_j\in U_j
\end{align}

Next we estimate $[\sum_{p,q=1}^n\Lambda_{N_j}^{p,q}(z_j,s)\frac{\partial g_j^{rv-1}(z_j,s)}{\partial z_j^p}\frac{\partial g_j^{sv-1}(z_j,s)}{\partial z_j^q}]_v$ in $(\ref{jj11})$. By induction hypothesis (\ref{induction2}), set 
\begin{align}\label{jj14}
\alpha_j^{rv-1}(z_j,s):=g_j^{rv-1}(z_j,s)-z_j^r \ll A(s)\,\,\,\,\,\text{for all $r=1,...,n$}
\end{align}
 Since $\Lambda_{N_j}^{p,q}(z_j,s)$ is holomorphic, and $\Lambda_{N_j}^{p,q}(z_j,0)=\Lambda_{M_{0j}}^{p,q}(z_j)$ from $(\ref{ii34})$, we may assume that for all the pairs $(p,q)$,
\begin{align} \label{jj13}
 \Pi_j^{p,q}(z_j,s):=\Lambda_{N_j}^{p,q}(z_j,s)-\Lambda_{M_{0j}}^{p,q}(z_j)\ll A_1(s)=\frac{b_1}{16c_1}\sum_{v=1}^\infty\frac{c_1^v(s_1+\cdots +s_l)^v}{v^2}
\end{align} 
for some constants $b_1,c_1>0$. If we choose $b>b_1$ and $c>c_1$, then we have 
\begin{align}\label{jj123}
\Pi_j^{p,q}(z_j,s)\ll \frac{b_1}{b}A(s).
\end{align}
Now assume that $z_j=(z_j^1,...,z_j^n) \in U_j^{\delta}$ from Remark \ref{remark123} (4). Then by Cauchy's integral formula, and $(\ref{jj14})$, we have, for $p=1,...,n$,
\begin{align*}
\frac{\partial \alpha_j^{rv-1}(z_j,s)}{\partial z_j^p}=\frac{1}{2\pi i}\int_{|\xi-z_j^p|=\delta}\frac{\alpha_j^{rv-1}(z_j^1,...,\overset{\text{$p$-th}}{\xi},...,z_j^n,s)}{(\xi-z_j^p)^2}d\xi
\end{align*}
Hence we have, for $p=1,...,n,$
\begin{align}\label{jj12}
\frac{\partial \alpha_j^{rv-1}(z_j,s)}{\partial z_j^p}\ll \frac{A(s)}{\delta}
\end{align}
Then from (\ref{jj14}),(\ref{jj13}), (\ref{jj123}), (\ref{jj12}) and $(\ref{ii99})$, we have 
\begin{align}\label{jj15}
&[\sum_{p,q=1}^n\Lambda_{N_j}^{p,q}(z_j,s)\frac{\partial g_j^{rv-1}(z_j,s)}{\partial z_j^p}\frac{\partial g_j^{sv-1}(z_j,s)}{\partial z_j^q}]_v\\
&=[\sum_{p,q=1}^n(\Pi_j^{p,q}(z_j,s)+\Lambda_{M_{0j}}^{p,q}(z_j))\frac{\partial (\alpha_j^{rv-1}(z_j,s)+z_j^r)}{\partial z_j^p}\frac{\partial (\alpha_j^{sv-1}(z_j,s)+z_j^s)}{\partial z_j^q}]_v \notag\\
&=[\sum_{p,q=1}^n\Pi_j^{p,q}(z_j,s)\frac{\partial \alpha_j^{rv-1}(z_j,s)}{\partial z_j^p}\frac{\partial \alpha_j^{sv-1}(z_j,s)}{\partial z_j^q}]_v+[\sum_{p,q=1}^n\Pi_j^{p,q}(z,s)\frac{\partial \alpha_j^{rv-1}(z,s)}{\partial z_j^p}\frac{\partial z_j^s}{\partial z_j^q}]_v \notag\\
&+[\sum_{p,q=1}^n\Pi_j^{p,q}(z_j,s)\frac{\partial z_j^r}{\partial z_j^p}\frac{\partial \alpha_j^{sv-1}(z_j,s)}{\partial z_j^q}]_v+[\sum_{p,q=1}^n\Pi_j^{p,q}(z_j,s)\frac{\partial z_j^r}{\partial z_j^p}\frac{\partial z_j^s}{\partial z_j^q}]_v+[\sum_{p,q=1}^n\Lambda_{M_{0j}}^{p,q}(z_j)\frac{\partial \alpha_j^{rv-1}(z,s)}{\partial z_j^p}\frac{\partial \alpha_j^{sv-1}(z_j,s)}{\partial z_j^q}]_v\notag\\
&\ll \frac{n^2b_1}{b\delta^2}A(s)^3+\frac{2nb_1}{b\delta} A(s)^2+\frac{b_1}{b}A(s)+\frac{Cn^2}{\delta^2}A(s)^2 \,\,\,\, \text{for $C>0$ from (\ref{ii99}) which does not depend on $b,c.$}\notag\\
& \ll \left(\frac{n^2b_1b^2}{bc^2\delta^2}+\frac{2nb_1b}{bc\delta}+\frac{b_1}{b}+\frac{Cn^2b}{c\delta^2}\right)A(s)=\left(\frac{n^2b_1b}{c^2\delta^2}+\frac{2nb_1}{c\delta}+\frac{b_1}{b}+\frac{Cn^2b}{c\delta^2}\right)A(s)\,\,\,\,\,\text{from Remark \ref{remark123} (3)}\notag
\end{align}

Hence from (\ref{jj11}),(\ref{jj16}), (\ref{jj15}), we have

\begin{align}\label{jj19}
\lambda^{r,s}_{j|v}(z_j,s)\ll LA(s),\,\,\,\,\,z_j\in U_j^\delta
\end{align}
where $L=\frac{2bd_0^2(m+n)^2}{c}+\frac{n^2b_1b}{c^2\delta^2}+\frac{2nb_1}{c\delta}+\frac{b_1}{b}+\frac{Cn^2b}{c\delta^2}$.

Then by Lemma \ref{lemma3}, $(\ref{yt6})$ and $(\ref{jj19})$, we can choose solutions $h_{u|v}(s)$, $u=1,...,m$, $\{g_{j|v}(s)\}$ such that
\begin{align*}
h_{u|v}(s)\ll NA(s),\,\,\,\,\, g_{j|v}(s) \ll NA(s), \,\,\,\text{where}\,\,\, N=M\left(2K_1K^*+    L\right)    
\end{align*}
Note that $N$ is independent of $v$ and $K^*=\frac{2^{n+1}b_0}{c\delta}+\frac{b_0}{b}+\frac{2ba_0^2(m+n)^2}{c}$, $L=\frac{2bd_0^2(m+n)^2}{c}+\frac{n^2b_1b}{c^2\delta^2}+\frac{2nb_1}{c\delta}+\frac{b_1}{b}+\frac{Cn^2b}{c\delta^2}$. If we first choose a sufficiently large $b$, and then choose $c$ so that $\frac{c}{b}$ be sufficiently large (so that $\frac{b}{c}$ is sufficiently small and $\frac{b}{c^2}$ is sufficiently small), then we obtain $N \leq 1$. Note that $b$ and $c$ satisfy $b>\max \{b_0,b_1\}$, $c> \max \{c_0,c_1\}$, $\frac{ba_0(m+n)}{c}< \frac{1}{2}$ and $\frac{bd_0(m+n)}{c}<\frac{1}{2}$.

Hence the above solution $h_{u|v}(s),u=1,...,m, \{ g_{j|v}(s) \}$ satisfy the inequalities
\begin{align*}
h_{v}(s)\ll A(s),\,\,\,\,\, g_{j|v}(s)\ll A(s)
\end{align*}
Since $h^v(s)=h^{v-1}(s)+h_v(s),\,\,\,\,\, g_j^v(z_j,s)=g_j^{v-1}(z_j,s)+g_{j|v}(z_j,s)$, we have $h^v(s)\ll A(s)$, and $g_j^v(z_j,s)\ll A(s)$. This completes the induction, and so we have $h(s)\ll A(s)$ and $g_j(z_j,s)-z_j\ll A(s)$. These inequalities imply that, if $|s|<\frac{1}{lc}$, $h(s)$ converges absolutely, and $g_j(z_j,s)$ converges absolutely and uniformly for $z_j\in U_j$.

\subsection{Proof of Theorem \ref{complete9}}\

 By the same argument presented in page 303-304, we can glue together each $g_j$ on $U_j^{\delta}\times \Delta_{\epsilon}$ to construct a Poisson holomorphic map $g:\pi^{-1}(\Delta_{\epsilon})=(\bigcup_j U_j^{\delta}\times \Delta_{\epsilon},\Lambda_{\mathcal{N}}|_{\Delta_{\epsilon}} )\to (\mathcal{M},\Lambda_{\mathcal{M}})$  which extends the identity map $g_0:(N_0,\Lambda_0)\to (M_0=N_0,\Lambda_0)$ (see \cite{Kod05} page 303-304 for the detail and notations). This completes the proof of Theorem $\ref{complete9}$.

 \begin{example}
Let $U_i=\{[z_0,z_1,z_2]|z_i\ne0\}$ $i=0,1,2$ be an open cover of complex projective plane $\mathbb{P}_{\mathbb{C}}^2$. Let $x=\frac{z_1}{z_0}$ and $w=\frac{z_2}{z_0}$ be coordinates on $U_0$. Then the holomorphic Poisson structures on $U_0$ are parametrized by $t=(t_1,...,t_{10})\in \mathbb{C}^{10}$
\begin{align*}
(t_1+t_2x+t_3w+t_4x^2+t_5xw+t_6w^2+t_7x^3+t_8x^2w+t_9xw^2+t_{10}w^3)\frac{\partial}{\partial x}\wedge \frac{\partial}{\partial w}
\end{align*}
This parametrizes the whole holomorphic Poisson structures on $\mathbb{P}_{\mathbb{C}}^2$ $($see \cite{Pin11} Proposition 2.2$)$. Let $\Lambda_0=x\frac{\partial}{\partial x}\wedge \frac{\partial}{\partial w}$ be the holomorphic Poisson structure on $\mathbb{P}_{\mathbb{C}}^2$.   Then $\mathbb{H}^1(\mathbb{P}_{\mathbb{C}}^2,\Theta_{\mathbb{P}_\mathbb{C}^2}^\bullet)=5$, $\mathbb{H}^2(\mathbb{P}_{\mathbb{C}}^2,\Theta_{\mathbb{P}_\mathbb{C}^2}^\bullet)=0$$($see \cite{Pin11} Example $3.5$ $)$. $w^2\frac{\partial}{\partial x}\wedge\frac{\partial}{\partial w}$, $x^3\frac{\partial}{\partial x}\wedge\frac{\partial}{\partial w}$, $x^2w\frac{\partial}{\partial x}\wedge\frac{\partial}{\partial w}$, $xw^2\frac{\partial}{\partial x}\wedge\frac{\partial}{\partial w}$  and $w^3\frac{\partial}{\partial x}\wedge\frac{\partial}{\partial w}$ are the representatives of the cohomology classes consisting of the basis of $\mathbb{H}^1(\mathbb{P}_{\mathbb{C}}^2,\Theta_{\mathbb{P}_{\mathbb{C}}}^2)$. Let $t=(t_1,t_2,t_3,t_4,t_5)\in \mathbb{C}^5$. Let $\Lambda(t)=(t_1w^2+x+t_2x^3+t_3x^2w+t_5xw^2+t_5w^3)\frac{\partial}{\partial x}\wedge \frac{\partial}{\partial w}$ be the holomorphic Poisson structure on $\mathbb{P}_{\mathbb{C}}^2\times \mathbb{C}^5$. Then $(\mathbb{P}_{\mathbb{C}}^2\times \mathbb{C}^5,\Lambda(t),\mathbb{C}^5, \omega)$, where $\omega$ is the natural projection, is a Poisson analytic family with $\omega^{-1}(0)=(\mathbb{P}_{\mathbb{C}}^2,\Lambda_0)$. Since the complex structure does not change in the family, the Poisson Kodaira-Spencer map is an isomorphism. Hence the Poisson analytic family is complete at $0$. \end{example}

\bibliographystyle{amsalpha}
\bibliography{References-Ret4}

\end{document}